\documentclass[11pt, reqno]{amsart}

\usepackage[top=3.75cm, bottom=3cm, left=3cm, right=3cm]{geometry}
\usepackage{amsmath}
\usepackage{amsthm}
\usepackage{amsfonts}
\usepackage{amssymb}
\usepackage{mathtools}
\usepackage{bm}
\usepackage[colorlinks=true, citecolor=citation, urlcolor=citation, linkcolor=reference]{hyperref}
\usepackage{color}
\usepackage[all,cmtip]{xy}
\usepackage{enumitem}
\usepackage{tikz}  
\usepackage{graphicx}
\usepackage{scalerel}
\usepackage{todonotes}
\usepackage{comment}
\usepackage[mathscr]{eucal}

\definecolor{citation}{rgb}{0,.40,.80}
\definecolor{reference}{rgb}{.80,0,.40}

\numberwithin{equation}{section}

\theoremstyle{plain}
\newtheorem{theorem}{Theorem}[section]
\newtheorem{lemma}[theorem]{Lemma}
\newtheorem{proposition}[theorem]{Proposition}
\newtheorem{corollary}[theorem]{Corollary}
\newtheorem{sublemma}[theorem]{Sublemma}

\theoremstyle{definition}
\newtheorem{definition}[theorem]{Definition}
\newtheorem{example}[theorem]{Example}
\newtheorem{remark}[theorem]{Remark}

\newcommand{\st}{\mid} 
\newcommand{\set}[1]{\left\{ \, #1 \, \right\}}

\newcommand{\Dperf}{\mathrm{D}_{\mathrm{perf}}}
\newcommand{\Dqc}{\mathrm{D}_{\mathrm{qc}}}
\newcommand{\Db}{\mathrm{D^b}}
\newcommand{\llangle}{\left \langle}
\newcommand{\rrangle}{\right \rangle}
\DeclareMathOperator{\Forg}{Forg}

\newcommand{\Vect}{\mathrm{Vect}}
\newcommand{\fd}{\mathrm{fd}}
\newcommand{\Mod}{\mathrm{Mod}}

\newcommand{\Sch}{\mathrm{Sch}}
\newcommand{\Gpd}{\mathrm{Gpd}}

\newcommand{\Cl}{{\mathop{\mathcal{C}\!\ell}}}
\newcommand{\cB}{\mathcal{B}}

\newcommand{\Gr}{\mathrm{Gr}}

\DeclareMathOperator{\Pic}{Pic}

\newcommand{\tX}{\widetilde{X}}
\newcommand{\wtilde}{\widetilde}

\DeclareMathOperator{\colim}{colim}

\newcommand{\Spec}{\mathrm{Spec}}
\newcommand{\pug}{\mathrm{pug}}

\newcommand{\cHom}{\mathcal{H}\!{\it om}}
\DeclareMathOperator{\Hom}{Hom}
\DeclareMathOperator{\Ext}{Ext}

\DeclareMathOperator{\Aut}{Aut}

\newcommand{\id}{\mathrm{id}}

\newcommand{\Coh}{\mathrm{Coh}}

\newcommand{\Ku}{\mathcal{K}u}

\DeclareMathOperator{\characteristic}{char}

\newcommand{\num}{\mathrm{num}}
\DeclareMathOperator{\Knum}{\rK_{\num}}

\DeclareMathOperator{\Stab}{Stab}

\newcommand{\Inf}{\mathrm{Inf}}

\newcommand{\HH}{\mathrm{HH}}

\newcommand{\Cat}{\mathrm{Cat}}
\newcommand{\cl}{\mathrm{cl}}

\newcommand{\tr}{\mathrm{tr}}
\newcommand{\stable}{\mathrm{st}}

\newcommand{\br}{\mathrm{br}}
\newcommand{\op}{\mathrm{op}}

\newcommand{\GL}{\mathrm{GL}}

\newcommand{\ch}{\mathrm{ch}}

\newcommand{\td}{\mathrm{td}}

\newcommand{\cO}{\mathcal{O}}
\newcommand{\cA}{\mathcal{A}}
\newcommand{\cC}{\mathscr{C}}
\newcommand{\cD}{\mathscr{D}}

\newcommand{\cF}{\mathcal{F}}
\newcommand{\cG}{\mathcal{G}}
\newcommand{\cH}{\mathcal{H}}

\newcommand{\cM}{\mathcal{M}}

\newcommand{\cP}{\mathcal{P}}
\newcommand{\cQ}{\mathcal{Q}}
\newcommand{\cR}{\mathcal{R}}
\newcommand{\cS}{\mathcal{S}}
\newcommand{\cT}{\mathscr{T}}
\newcommand{\cU}{\mathcal{U}}

\newcommand{\cW}{\mathcal{W}}
\newcommand{\cX}{\mathcal{X}}
\newcommand{\cY}{\mathcal{Y}}

\newcommand{\cI}{\mathcal{I}}

\newcommand{\rH}{\mathrm{H}}
\newcommand{\rh}{\mathrm{h}}
\newcommand{\rK}{\mathrm{K}}
\newcommand{\rS}{\mathrm{S}}
\newcommand{\rL}{\mathrm{L}}

\newcommand{\rR}{\mathrm{R}}

\newcommand{\fC}{\mathfrak{C}}
\newcommand{\fD}{\mathfrak{D}}

\newcommand{\bA}{\mathbf{A}}
\newcommand{\bC}{\mathbf{C}}

\newcommand{\bG}{\mathbf{G}}

\newcommand{\bZ}{\mathbf{Z}}

\newcommand{\bP}{\mathbf{P}}
\newcommand{\bQ}{\mathbf{Q}}
\newcommand{\bR}{\mathbf{R}}

\newcommand{\bv}{\mathbf{v}}

\newcommand{\sO}{\mathsf{O}}

\begin{document}

\title{Moduli spaces of stable objects in Enriques categories}

\author{Alexander Perry}
\address{Department of Mathematics, University of Michigan, Ann Arbor, MI 48109, USA \smallskip}
\email{arper@umich.edu}
\urladdr{http://www-personal.umich.edu/~arper/}

\author{Laura Pertusi} 
\address{Dipartimento di Matematica ``F.\ Enriques'' \\
Via Cesare Saldini 50 \\
Universit\`a degli Studi di Milano \\
20133 Milano, Italy \smallskip 
}
\email{laura.pertusi@unimi.it}
\urladdr{http://www.mat.unimi.it/users/pertusi/index.html}

\author{Xiaolei Zhao}
\address{Department of Mathematics \\
South Hall, Room 6607 \\
University of California \\
Santa Barbara, CA 93106, USA \smallskip
}
\email{xlzhao@math.ucsb.edu}
\urladdr{https://sites.google.com/site/xiaoleizhaoswebsite/}

\thanks{A.P. was partially supported by NSF grants DMS-2112747, DMS-2052750, and DMS-2143271, and a Sloan Research Fellowship. L.P. was supported by the national research project PRIN 2017 Moduli and Lie Theory. 
X.Z.\ was partially supported by the Simons Collaborative Grant 636187, NSF grant DMS-2101789, and NSF FRG grant DMS-2052665. }

\maketitle

\begin{abstract}
We study moduli spaces of stable objects in Enriques categories by exploiting their relation to moduli spaces of stable objects in associated K3 categories. 
In particular, we settle the nonemptiness problem for moduli spaces of stable objects in the Kuznetsov components of several interesting classes of Fano varieties, and deduce the nonemptiness of fixed loci of certain antisymplectic involutions on modular hyperk\"{a}hler varieties. 
\end{abstract}

\section{Introduction}

The goal of this paper is to study moduli spaces of stable objects 
in Enriques categories arising as the Kuznetsov components of Fano varieties. 
We develop a general approach to this problem based on the relation to moduli spaces of stable objects in K3 categories. 
We work over the complex numbers for the overview of our results below. 

\subsection{K3 categories and their moduli spaces}
\label{K3-categories}
In the celebrated work \cite{Mukai:BundlesK3}, Mukai proved that moduli spaces of stable sheaves of primitive class on K3 surfaces are smooth projective hyperk\"ahler varieties of the expected dimension, deformation equivalent to Hilbert schemes of points on a K3 surface (see also \cite{Beauville:remarksonc1zero, OG:weight2, Yoshioka:Irreducibility, Yoshioka:Abelian}). This result has been extended in \cite{bayer-macri-projectivity} to moduli spaces of stable objects, in the sense of Bridgeland \cite{bridgeland, bridgeland-K3}, in the bounded derived category of coherent sheaves $\Db(S)$ on a K3 surface $S$. 

Recently, these results have been further extended to the case where $\Db(S)$ is replaced by a suitable \emph{K3 category}. 
The first such example arises as the \emph{Kuznetsov component} $\Ku(Y)$ of a cubic fourfold $Y \subset \bP^5$ \cite{kuznetsov-cubic}, defined by the semiorthogonal decomposition 
\begin{equation*}
\Db(Y) = \langle \Ku(Y), \cO_Y, \cO_Y(1), \cO_Y(2) \rangle. 
\end{equation*} 
The component $\Ku(Y)$ is a K3 category in the sense that it has the same homological invariants, namely Serre functor and Hochschild homology, as 
$\Db(S)$ for a K3 surface $S$. 
Stability conditions on $\Ku(Y)$ were constructed in \cite{BLMS}, 
and in \cite{BLMNPS} an exact analog of Mukai's results were proved in this setting, leading to new unirational locally complete families of polarized hyperk\"ahler varieties, as well as tight connections between the geometry of $Y$ and the structure of $\Ku(Y)$. 

Other examples of K3 categories can be constructed as Kuznetsov components of Fano varieties using the results of \cite{kuznetsov-CY}. Among them, the best understood are those associated to even-dimensional \emph{Gushel--Mukai (GM) varieties} \cite{KuzPerry:dercatGM}, which are obtained by taking a quadric section or double cover of a linear section of the Grassmannian $\Gr(2,5) \subset \bP^{10}$ (see Example~\ref{ex_GM}). 
In \cite{PPZ}, we constructed stability conditions on such K3 categories and proved analogs for their moduli spaces of the known results for cubic fourfolds. 

In general, the results of \cite{IHC-CY2} show that given a family of K3 categories equipped with a family of stability conditions, relative moduli spaces of stable objects are smooth. 
This allows one to deduce Mukai-type structure results for moduli spaces on a generic fiber when there are special fibers of the form $\Db(S)$ for a K3 surface $S$. 

\subsection{Kuznetsov components of Fano threefolds} 
Besides the above examples, Fano threefolds are the main case where stability conditions have been constructed on Kuznetsov components \cite{BLMS}. 
In fact, for many Fano threefolds $Y$ there exists a stability condition $\sigma$ on $\Ku(Y)$ such that:
\begin{enumerate}
\item $\sigma$ is \emph{Serre invariant}, i.e. fixed by the Serre functor modulo the natural $\widetilde{\GL}_2^{+}(\bR)$-action on the space of numerical stability conditions. 
\item $\sigma$ 
is the unique Serre invariant stability condition modulo the $\widetilde{\GL}_2^{+}(\bR)$-action. 
\end{enumerate} 
Namely, by a combination of various works \cite{PY, FP, JLLZ, PR} (see also \cite[\S5.5]{categorical-torelli-survey}), 
this holds for Fano threefolds of Picard number $1$, index $2$, and degree $d \geq 2$, 
as well as for Fano threefolds of Picard number $1$, index $1$, and and degree $d \geq 10$.\footnote{In general, however, Serre invariant stability conditions need not exist on $\Ku(Y)$; see \cite[Corollary 1.9]{kuznetsov-perry-SD}, which implies for instance that they do not exist when $Y$ has Picard number $1$, index $1$, and degree $6$.} 
Note that stability conditions which differ by the $\widetilde{\GL}_2^{+}(\bR)$-action have the same set of (semi)stable objects, so they are the same from the perspective of their moduli spaces; therefore, when $\sigma$ as above exists, its moduli spaces are canonical to $\Ku(Y)$.  

There is a rich emerging theory of moduli spaces of stable objects in Kuznetsov components of Fano threefolds, with applications including the realization of classical moduli spaces in these terms \cite{arithmetically-CM-LMS, rota, Zhang, LiuZhang, FP} and the proofs of (categorical) Torelli theorems \cite{BMMS:categoricalcubic3, feyz:desing, zhang-hochschild:}; see \cite{categorical-torelli-survey} for a survey. 
What is still lacking, however, are general structure theorems for these moduli spaces, parallel to the situation for K3 categories discussed in \S\ref{K3-categories}. 
In particular, a fundamental open problem is to show that these 
moduli spaces are nonempty when the expected dimension is nonnegative. 
In comparison to the case of K3 categories, the main difficulty is that the categories $\Ku(Y)$ arising from Fano threefolds within a given deformation class are typically not equivalent to the derived category of a variety for any choice of $Y$; this precludes any direct deformation-theoretic reduction of the problem to a more geometric one. 

In this paper, we solve the nonemptiness problem in two interesting cases, 
based on a general strategy that also allows access to other structural properties of the resulting moduli spaces, like their singularities. 

The first case is \emph{quartic double solids}, i.e. smooth double covers $Y \to \bP^3$ branched along a quartic surface, whose Kuznetsov component is defined by 
\begin{equation*}
\Db(Y) = \langle \Ku(Y), \cO_Y, \cO_Y(1) \rangle. 
\end{equation*} 
Among Fano threefolds of Picard number $1$ and index $2$, these are the ones of lowest degree, namely $2$, for which Serre-invariant stability conditions are known to exist and be unique modulo $\widetilde{\GL}_2^{+}(\bR)$.  
To state our result, we recall from \cite[Proposition 3.9]{Kuz_Fano} that the numerical Grothendieck group $\Knum(\Ku(Y))$ has rank $2$, and there is a basis $\mu_1, \mu_2$, described in~\eqref{eq_basisKnumqds}, with respect to which the Euler form is given by 
\begin{equation*}
\begin{pmatrix}
-1 & -1 \\ 
-1 & -2
\end{pmatrix}.
\end{equation*}

\begin{theorem}
\label{theorem-M-quartic-double}
Let $Y$ be a quartic double solid. 
Let $v \in \Knum(\Ku(Y))$ and let $\sigma$ be a Serre invariant stability condition on $\Ku(Y)$. 
\begin{enumerate}
\item \label{non-empty} The moduli space $M_{\sigma}(\Ku(Y), v)$ of $\sigma$-semistable objects of class $v$ is nonempty. 
\item \label{v-smooth-M-KuY}
If $v =  a\mu_1 + b\mu_2$ where $\gcd(a,b)=1$ and $a$ is odd, then for generic $Y$ the moduli space 
$M_{\sigma}(\Ku(Y), v)$ is smooth and projective of dimension $-\chi(v,v) + 1$. 
\end{enumerate} 
\end{theorem} 

Above, by the moduli space $M_{\sigma}(\Ku(Y), v)$ we mean more precisely the good moduli space for the algebraic stack $\cM_{\sigma}(\Ku(Y), v)$ of $\sigma$-semistable objects of class $v$, whose existence follows from the results of \cite{BLMNPS}. 

\begin{remark}
The moduli space $M_\sigma(\Ku(Y), v)$ may fail to be smooth for all $Y$ when $v$ is not of the form described in part~\eqref{v-smooth-M-KuY} of the theorem; 
indeed, by \cite[Theorem 5.2]{rota}, $M_\sigma(\Ku(Y), \mu_2)$ is singular for every $Y$. 
On the other hand, we expect that $M_\sigma(\Ku(Y), v)$ is projective for every $v$ and $Y$. 
This would follow from projectivity of moduli spaces of semistable objects in K3 categories with 
respect to non-generic stability conditions (Remark~\ref{remark-projectivity}). 
\end{remark} 

The second class of Fano threefolds we consider are GM threefolds $X$, which can be described as either a smooth quadric section of $\Gr(2,5) \cap M$ where $M \subset \bP^{9}$ is a codimension $2$ linear subspace of Pl\"{u}cker space, or as a smooth double cover of $\Gr(2,5) \cap M$ branched over a quadric section where $M \subset \bP^{9}$ is a codimension $3$ subspace. 
Among Fano threefolds of Picard number $1$ and index $1$, these are the ones of lowest degree, namely $10$, for which Serre invariant stability conditions are known to exist and be unique modulo $\widetilde{\GL}_2^{+}(\bR)$.  
The Kuznetsov component is defined by 
\begin{equation*}
\Db(X) = \langle \Ku(X) , \cO_X, \cU_X^{\vee} \rangle 
\end{equation*} 
where $\cU_X$ is the pullback of the tautological rank $2$ bundle on $\Gr(2,5)$. 

\begin{theorem}
\label{theorem-Msigma-GM}
Let $X$ be a GM threefold. 
Let $v \in \Knum(\Ku(X))$ and let $\sigma$ be a Serre invariant stability condition on $\Ku(X)$. 
\begin{enumerate}
\item \label{nonempty-GM} The moduli space $M_{\sigma}(\Ku(X), v)$ of $\sigma$-semistable objects of class $v$ is nonempty. 
\item \label{Msigma-GM-smooth}
If $v$ is primitive, i.e. not divisible in $\Knum(\Ku(X))$, then for generic $X$ the moduli space $M_{\sigma}(\Ku(X), v)$ is smooth and projective of dimension $-\chi(v,v) + 1$. 
\end{enumerate} 
\end{theorem} 

\begin{remark} 
The moduli space $M_{\sigma}(\Ku(X), v)$ may fail to be smooth when $X$ is not generic; indeed, by \cite[Theorem 1.5]{Zhang}, $M_\sigma(\Ku(Y), \kappa_1)$ is singular for certain $X$, 
where $\kappa_1$ is one of the basis vectors for $\Knum(\Ku(X))$ described in \eqref{eq_basisKnumGM3}. 
On the other hand, similar to the quartic double solid case, we expect that $M_{\sigma}(\Ku(X), v)$ is projective for 
every $v$ and $X$ (Remark~\ref{remark-projectivity}). 
\end{remark} 

\begin{remark}
\label{remark-GM-5fold}
By the duality conjecture proved in \cite[Theorem 1.6]{KuzPerry:cones}, for any GM fivefold there exists a GM threefold with equivalent Kuznetsov component. 
Therefore, Theorem~\ref{theorem-Msigma-GM} is also true verbatim for GM fivefolds. 
\end{remark} 

Theorem~\ref{theorem-Msigma-GM} has an interesting consequence for moduli spaces of stable objects in Kuznetsov components of GM fourfolds. 
For such a fourfold $W$, in \cite[Theorem 4.12 and Proposition 4.16]{PPZ} we constructed full numerical stability conditions $\sigma_{\alpha, \beta}$ on $\Ku(W)$, depending on suitable real parameters $\alpha$ and $\beta$; let $\Stab^{\circ}(\Ku(W))$ denote the collection of these stability conditions. 
We proved in \cite[Theorem 1.5]{PPZ} that if $w \in \Knum(\Ku(W))$ is primitive and $\sigma \in \Stab^{\circ}(\Ku(W))$ is $w$-generic (i.e. the notions of $\sigma$-semistability and $\sigma$-stability coincide for objects of class $w$), then the moduli space $M_{\sigma}(\Ku(W), w)$ is a smooth projective hyperk\"{a}hler variety of dimension $-\chi(w,w) + 2$.\footnote{In fact we showed that this statement holds for any $\sigma$ in the connected component $\Stab^{\dagger}(\Ku(W))$ of $\Stab(\Ku(W))$ containing $\Stab^{\circ}(\Ku(W))$, but below we will be interested in $\Stab^{\circ}(\Ku(W))$.}
Moreover, there is a canonical sublattice 
\begin{equation*}
-A_1^{\oplus 2} = \begin{pmatrix}
-2 & 0 \\ 
0 & -2
\end{pmatrix} 
\subset \Knum(\Ku(W)) , 
\end{equation*} 
described as the projection of $\rK_0(\Gr(2,5))$ \cite[Lemma 2.27]{KuzPerry:dercatGM}, and we observed in \cite[Proposition 5.16]{PPZ} that for $w \in -A_1^{\oplus 2}$ the moduli space $M_{\sigma}(\Ku(W), w)$ admits a natural antisymplectic involution $\iota$. As noted in \cite[Remark 5.8]{BP}, this involution is induced by an explicit involutive autoequivalence of $\Ku(W)$. 

\begin{theorem}
\label{theorem-fixed-locus}
Let $W$ be a GM fourfold. 
Let $w \in -A_1^{\oplus 2} $ be a primitive vector, and let $\sigma \in \Stab^{\circ}(\Ku(W))$ be a $w$-generic stability condition. 
Then the fixed locus of the involution $\iota$ on $M_{\sigma}(\Ku(W), w)$ is a nonempty smooth Lagrangian subvariety. 
\end{theorem} 

\begin{remark}
By the same reasoning as in Remark~\ref{remark-GM-5fold}, 
Theorem~\ref{theorem-fixed-locus} is also true verbatim for GM sixfolds. 
\end{remark}

The content of the theorem is the nonemptiness of the fixed locus; indeed, then by a general result on antisymplectic involutions \cite[Lemma 1]{Beauville}, it is automatically a smooth Lagrangian. 
As we explain below, our proof in fact gives a description of the fixed locus of $\iota$ in terms of a moduli space of stable objects in the Kuznetsov component of a GM threefold. 
Finally, we note that Theorem~\ref{theorem-fixed-locus} in particular gives a new proof of the nonemptiness of $M_{\sigma}(\Ku(W), w)$, which was one of the main results of \cite{PPZ}. 

\subsection{Enriques categories and their moduli spaces} 
Now we explain the main ideas behind the proofs of the above theorems, 
which involve some results of independent interest. 

The key to our analysis is that the Kuznetsov components of quartic double solids and GM threefolds are \emph{Enriques categories} in the sense that, like the derived category of an Enriques surface, they have Serre functor of the form 
$\tau \circ [2]$ where $\tau$ is the generator of a nontrivial $\bZ/2$-action. 
To any Enriques category $\cC$, there is an associated \emph{CY2 cover} $\cD = \cC^{\bZ/2}$ defined as the invariant category for the $\bZ/2$-action. 
When $\cC = \Db(S)$ for an Enriques surface $S$, then $\cD = \Db(T)$ where $T$ is the K3 double cover of $S$, while for the Kuznetsov components of interest, the results of \cite{KuzPer_cycliccovers} show that $\cD$ can be described as follows:  
\begin{enumerate}
\item If $Y \to \bP^3$ is a quartic double solid branched over a quartic K3 surface $S \subset \bP^3$, then the CY2 cover of $\Ku(Y)$ is $\Db(S)$. 
\item If $X$ is a GM threefold, then the CY2 cover of $\Ku(X)$ is $\Ku(W)$ where $W$ is the ``opposite'' GM variety (Example~\ref{ex_GM}), which is either a Fano fourfold or a K3 surface. 
\end{enumerate}
By the above descriptions, the CY2 cover is in fact a K3 category in these cases. 

An important feature of the CY2 cover $\cD$ is that it comes equipped with a \emph{residual $\bZ/2$-action} --- which should be thought of as an analog of the covering involution of a K3 surface over an Enriques surface --- such that the invariant category $\cD^{\bZ/2}$ recovers $\cC$. In fact, as we review in \S\ref{section-reconstruction}, this is an instance of a more general result of Elagin \cite{Elagin}, which shows that for a finite abelian group $G$, taking $G$-invariants gives a duality between categories equipped with a $G$-action and categories equipped with a $G^{\vee}$-action, where $G^{\vee}$ is the dual group. 

Strictly speaking, to have a well-behaved theory of $G$-invariants, 
we work with enhanced triangulated categories in the $\infty$-categorical sense. 
In Proposition~\ref{proposition-enhancement-CG} we prove that the operation of taking $G$-invariants commutes with taking the homotopy category 
(cf. \cite{Elagin} and \cite{sosna} for similar results in the DG setting). 
This essentially allows us to ignore higher-categorical phenomenon when studying $G$-equivariant stability conditions, since stability conditions are defined at the level of the homotopy category. 
By results of \cite{polishchuk, MMS}, we thus find that if a finite group $G$ acts on an enhanced triangulated category $\cC$, and $\sigma$ is a stability condition that is fixed by each $g \in G$, then there is an induced stability condition $\sigma^G$ on $\cC^G$ (Theorem~\ref{theorem-equivariant-stability}).  
Moreover, we observe in Lemma~\ref{lemma_backtosamestability} that $\sigma^G$ is fixed by the $G^{\vee}$-action on $\cC^{G}$, and that $(\sigma^G)^{G^{\vee}}$ recovers $\sigma$ under the equivalence $(\cC^{G})^{G^{\vee}} \simeq \cC$ mentioned above. 

In particular, in the case of an Enriques category $\cC$ with CY2 cover $\cD$, we obtain a correspondence between stability conditions $\sigma_{\cC}$ on $\cC$ and $\sigma_{\cD}$ on $\cD$ which are fixed by the $\bZ/2$-actions. 
We will consider the case where $\cC$ and $\cD$ are semiorthogonal components of derived categories of varieties and the stability conditions $\sigma_{\cC}$ and $\sigma_{\cD}$ are \emph{proper} (in the sense of Remark~\ref{remark-proper-stab}), which in particular means that for numerical classes $v \in \Knum(\cC)$ and $w \in \Knum(\cD)$, there exist proper algebraic stacks $\cM_{\sigma_{\cC}}(\cC, v)$ and $\cM_{\sigma_{\cD}}(\cD, w)$ of semistable objects of class $v$ and $w$.\footnote{By the results of \cite{BLMNPS}, all known constructions of stability conditions lead to proper ones.} 
Let $\pi_* \colon \cD \to \cC$ and $\pi^* \colon \cC \to \cD$ suggestively denote the forgetful and inflation functors from and to $\cD = \cC^{\bZ/2}$; when $\cC = \Db(S)$ for an Enriques surface $S$ with K3 cover $\pi \colon T \to S$, then $\pi_*$ and $\pi^*$ agree with the usual pushforward and pullback functors along $\pi$. 
For an object in $\cC$ or $\cD$, we say it is ``fixed by the $\bZ/2$-action'' if it is isomorphic to the object obtained by applying the generator of the action. 

\begin{theorem}
\label{theorem-moduli}
Let $\cC$ be an Enriques category with CY2 cover $\cD$, and assume that $\cC$ and $\cD$ arise a semiorthogonal components of the derived categories of smooth proper varieties. 
Let $\sigma_{\cC}$ be a stability condition on $\cC$ fixed by the $\bZ/2$-action, let $\sigma_{\cD}$ be the corresponding stability condition on $\cD$, and assume that $\sigma_{\cC}$ and $\sigma_{\cD}$ are proper. 
\begin{enumerate}
\item \label{theorem-enriques-moduli}
For any $v \in \Knum(\cC)$, the moduli stack $\cM_{\sigma_{\cC}}(\cC, v)$ is smooth at any point corresponding to a $\sigma_{\cC}$-stable object $E$ that is not fixed by the $\bZ/2$-action on $\cC$.  
If $v$ is fixed by the $\bZ/2$-action, then 
the functor $\pi_* \colon \cD \to \cC$ induces a surjective morphism 
\begin{equation*}
f \colon \bigsqcup_{\pi_*(w) = v} \cM_{\sigma_{\cD}}(\cD, w) \to \cM_{\sigma_{\cC}}(\cC, v)^{\bZ/2} ,
\end{equation*} 
where the target is the fixed locus of the $\bZ/2$-action on $\cM_{\sigma_{\cC}}(\cC, v)$ and the union in the source ranges over $w \in \Knum(\cD)$ such that $\pi_*(w) = v$. 
Moreover, $f$ is \'{e}tale of degree $2$ at $\sigma_{\cD}$-stable objects not fixed by the residual $\bZ/2$-action. 

\item \label{theorem-CY2-moduli} 
For any $w \in \Knum(\cD)$, the moduli stack $\cM_{\sigma_{\cD}}(\cD, w)$ is smooth at any point corresponding to a $\sigma_{\cD}$-stable object. 
If $w$ is fixed by the residual $\bZ/2$-action, 
the functor $\pi^* \colon \cC \to \cD$ induces a surjective morphism 
\begin{equation*}
g \colon \bigsqcup_{\pi^*(v) = w} \cM_{\sigma_{\cC}}(\cC, v) \to \cM_{\sigma_{\cD}}(\cD, w)^{\bZ/2} ,
\end{equation*} 
where the target is the fixed locus of the $\bZ/2$-action on $\cM_{\sigma_{\cD}}(\cD, w)$ and the union in the source ranges over $v \in \Knum(\cC)$ such that $\pi^*(v) = w$. 
Moreover, $g$ is \'{e}tale of degree $2$ at $\sigma_{\cC}$-stable objects not fixed by the $\bZ/2$-action. 
\end{enumerate} 
\end{theorem} 

Above, the fixed loci of the $\bZ/2$-actions are taken in the stack-theoretic sense, so that they parameterize linearized objects in the sense described in Example~\ref{example-classical-linear}. See \cite{Romagny} for background on fixed point stacks; in particular, \cite[Theorem 3.3]{Romagny} shows that the fixed point stacks for finite group actions considered in this paper are automatically algebraic. 

\begin{remark}
A stability condition $\sigma_{\cC}$ which is fixed by the $\bZ/2$-action is in particular Serre invariant. 
By the results of \cite{PY, FP, JLLZ, PR}, when $\cC$ is the Kuznetsov component of a quartic double solid or GM threefold, a $\bZ/2$-fixed $\sigma_{\cC}$ exists and is unique up to the action of $\widetilde{\GL}_2^{+}(\bR)$. 
Furthermore, the $\bZ/2$-action on $\Knum(\cC)$ is trivial in these examples (Lemmas~\ref{lemma_iotaisidentity_qds} and \ref{lemma_iotaisidentity_GM3}), so~\eqref{theorem-enriques-moduli} applies for arbitrary  $v \in \Knum(\cC)$. 
\end{remark} 

Theorem~\ref{theorem-moduli} combines several results: smoothness criteria for moduli of objects in Enriques and CY2 categories, proved in more general relative forms in Corollary~\ref{corollary-enriques-moduli-smooth} and \cite[Theorem 1.4]{IHC-CY2}, as well as Proposition~\ref{proposition-maps-bw-moduli} and Remark~\ref{remark-symmetrized-maps-bw-moduli} which describe the morphisms $f$ and $g$ in a more general setting of categories equipped with a $\bZ/2$-action. 
We note that in case $\cC = \Db(S)$ for an Enriques surface $S$ and $\cD = \Db(T)$ where $T$ is the associated K3 surface, the construction and properties of the morphism $g$ in~\eqref{theorem-CY2-moduli} go back to Nuer \cite{Nuer}. 

Theorem~\ref{theorem-moduli} provides a useful tool for studying moduli spaces of objects in $\cC$, since those of the CY2 cover $\cD$ are often well-understood in examples. 
For instance, consider the nonemptiness problem for a vector $v \in \Knum(\cC)$. 
By part~\eqref{theorem-enriques-moduli} of the theorem, it would be enough to find a vector $w \in \Knum(\cD)$ such that $\pi_*(w) = v$ and $\cM_{\sigma_{\cD}}(\cD, w)$ is nonempty. 
Unfortunately, in our examples, for a given $\cC$ and $v$ there often does not exist such a $w$. 
However, we prove the following more flexible deformation-theoretic criterion for nonemptiness. 
The formulation uses the notions of an Enriques category $\mathfrak{C}$ over a base $S$ (see Definition~\ref{definition-family-Enriques}) and of a stability condition $\underline{\sigma} = (\sigma_s)_{s \in S}$ on $\fC$ over $S$ (see \cite{BLMNPS} and \S\ref{subsec_moduliandfamilies}), which consists of proper numerical stability conditions $\sigma_s$ on the the fibers $\fC_s$ satisfying certain compatibility and tameness properties.

\begin{theorem}
\label{theorem-deformation} 
Let $\cC$ be an Enriques category arising as a semiorthogonal component of the derived category of a smooth proper variety,  
let $\sigma$ be a proper numerical stability condition on $\cC$, 
and let $v \in \Knum(\cC)$. 
Assume there exists: 
\begin{enumerate}
\item \label{1} An Enriques category $\fC$ over a complex variety $S$ such that $\fC$ is an $S$-linear semiorthogonal component of $\Dperf(\cX)$ for a smooth proper morphism $\cX \to S$. 
\item \label{2} A stability condition $\underline{\sigma} = (\sigma_s)_{s \in S}$ on $\fC$ over $S$ with respect to a lattice $\Lambda$ 
and homomorphisms $\bv_s \colon \Knum(\fC_s) \to \Lambda$.  
\item \label{defo-fiber-0}
A point $0 \in S(\bC)$ such that $\fC_0 \simeq \cC$ and under this equivalence $\sigma_0 = \sigma$ and the map $\bv_0 \colon \Knum(\cC) = \Knum(\fC_0) \to \Lambda$ is injective. 
\item \label{1-exists}
A point $1 \in S(\bC)$ such that: 
\begin{enumerate}
    \item \label{1-exists-a} $\sigma_1$ is fixed by the $\bZ/2$-action on $\fC_1$ and $\cM_{\sigma_1}(\fC_1, \bv_0(v))$ consists of stable objects. 
    \item \label{1-exists-b} If $\pi_{1*} \colon \fD_1 \to \fC_1$ is the CY2 cover of $\fC_1$ with corresponding stability condition $\sigma_{\fD_1}$, then 
    $\sigma_{\fD_1}$ is proper. 
    \item \label{conditions-w1} There exists a class $w_1 \in \Knum(\fD_1)$ such that if $v_1 = \pi_*(w_1)$ then $\bv_1(v_1) = \bv_0(v)$, 
    $\cM_{\sigma_{\fD_1}}(\fD_1, w_1)$ is nonempty, and $-\chi(w_1, w_1) + 2 < -\chi(v_1, v_1) + 1$.
\end{enumerate}
\end{enumerate}
Then the moduli stack $\cM_{\sigma}(\cC,v)$ is nonempty. 
\end{theorem} 

\begin{remark}
In our applications, by known results $\cM_{\sigma_{\fD_1}}(\fD_1, w_1)$ will be nonempty if and only if $-\chi(w_1, w_1) + 2 \geq 0$, so \eqref{conditions-w1} becomes a purely lattice theoretic condition. 
\end{remark} 

Theorem~\ref{theorem-deformation} provides an effective tool for proving nonemptiness of moduli spaces of semistable objects in Enriques categories. 
Indeed, we prove Theorems~\ref{theorem-M-quartic-double} and~\ref{theorem-Msigma-GM} by considering specializations of $\Ku(Y)$ or $\Ku(X)$ for which the K3 cover becomes the derived category of a K3 surface and the hypotheses of Theorem~\ref{theorem-deformation} can be verified. 
The smoothness statements in Theorems~\ref{theorem-M-quartic-double} and~\ref{theorem-Msigma-GM} are proved via the smoothness criterion of Theorem~\ref{theorem-moduli}\eqref{theorem-enriques-moduli}. 

Finally, Theorem~\ref{theorem-fixed-locus} is proved by combining the nonemptiness result of Theorem~\ref{theorem-Msigma-GM} with the description of the fixed locus as the image of the morphism $g$ in Theorem~\ref{theorem-moduli}\eqref{theorem-CY2-moduli}. 
As an important technical ingredient, we prove that the stability conditions $\sigma \in \Stab^{\circ}(\Ku(W))$ are fixed by the residual $\bZ/2$-action (Theorem~\ref{thm_invariantstab}), and deduce that they are all in the same $\widetilde{\GL}_2^{+}(\bR)$-orbit (Corollary~\ref{cor_samestabinStabcirc}). 

\subsection{Further directions} 
Now that their nonemptiness is known, 
it would be interesting to study other geometric properties of the moduli spaces from Theorems~\ref{theorem-M-quartic-double} and \ref{theorem-Msigma-GM}, 
like their singularities, Kodaira dimension, number of irreducible components, map to the intermediate Jacobian, and so on. 

Our approach to proving Theorems~\ref{theorem-M-quartic-double} and \ref{theorem-Msigma-GM} is quite general and should apply to other geometric examples of Enriques categories, like Examples~\ref{example-verra-3fold} and~\ref{example-other-enriques-cats}, but for brevity we have focused on just two interesting cases.
Even in the case of Enriques surfaces, it would be interesting to reprove nonemptiness results for moduli spaces using our method. 

Finally, we expect a similar nonemptiness result for moduli spaces of stable objects in the Kuznetsov component of a cubic threefold
with respect to Serre invariant stability conditions. 
Building on the method of this paper, one potential approach to the problem would be to use the relation between the Kuznetsov component of a cubic threefold and that of a cyclic cubic fourfold branched along it. 

\subsection{Related work} 

The nonemptiness result in Theorems~\ref{theorem-M-quartic-double} and \ref{theorem-Msigma-GM} has been applied in \cite{LPZ3} to show that the Kuznetsov components of quartic double solids and GM threefolds have a strongly unique enhancement.

In \cite{BP} the remaining case of Kuznetsov's Fano threefold conjecture, concerning the Kuznetsov components of quartic double solids and GM threefolds, was settled. 
Like ours, the paper~\cite{BP} is based on the relation between Enriques and CY2 categories. 

In a somewhat similar spirit to this paper, in \cite{BO} the authors studied finite group actions on bounded derived categories of smooth projective varieties and fixed loci in the moduli space of the equivariant derived category. 

Fixed loci of antisymplectic involutions on projective hyperk\"ahler manifolds have been extensively studied in \cite{FMOGS}. We also mention that the fixed loci studied in Theorem~\ref{theorem-fixed-locus} have a conjectural description in terms of the Lagrangian subspaces arising from smooth hyperplane sections in the GM fourfold by \cite{GLZ}.

Finally, we note that Lemma~\ref{lemma_backtosamestability} on the correspondence between $G$-fixed and $G^{\vee}$-fixed stability conditions was independently proved by Dell \cite{dell}. 

\subsection{Organization of the paper} 
In \S\ref{sec_groupactiononcat} we provide some preliminary material on group actions on categories and their (co)invariants. 
In \S\ref{sec_enriques} we give the definition of Enriques categories with a list of geometric examples, as well as recall some results on the derived categories of branched double covers. In  \S\ref{sec_stabcond} we provide a quick introduction to stability conditions, moduli spaces, stability conditions on equivariant categories, tilt stability, and the known results for the Kuznetsov components of GM varieties and quartic double solids. 
In \S\ref{sec_moduli} we prove our general structural results on moduli spaces: 
a smoothness criterion (Lemma~\ref{lemma-smoothness-Ext2} and Corollary~\ref{corollary-enriques-moduli-smooth}), the relation between moduli spaces in an Enriques category and its CY2 cover (Theorem~\ref{theorem-moduli}), and our deformational nonemptiness criterion (Theorem~\ref{theorem-deformation}). 
In the final three sections, we apply these results to our examples of interest:  \S\ref{section-quartic-double-solids} deals with quartic double solids, \S\ref{section-GM-3folds} with GM threefolds, and \S\ref{section-GM-4folds} with GM fourfolds. 

\subsection{Conventions} 
A variety over a field $k$ is an integral scheme which is separated and of finite type over $k$. 
For a scheme $X$, $\Dperf(X)$ denotes the category of perfect complexes, $\Dqc(X)$ denotes the unbounded derived category of quasi-coherent sheaves, and $\Db(X)$ denotes the bounded derived category of coherent sheaves. In fact, in all cases where we consider $\Db(X)$ in this paper, $X$ will be a smooth variety, so $\Db(X) = \Dperf(X)$. 
All functors are implictily derived; e.g. we write simply $f^*$ and $f_*$ for the derived pushforward and pullback functors along a morphism $f \colon X \to Y$, and $E \otimes F$ for the derived tensor product of objects $E, F \in \Dperf(X)$. 

For technical reasons, we work with linear categories, which are $\infty$-categorical enhancements of triangulated categories; see \S\ref{sec_groupactiononcat} and in particular Example~\ref{example-linear-categories} for our conventions and terminology. 
For a $k$-linear category $\cC$ and objects $E, F \in \cC$, we denote by $\cHom_k(E,F)$ the corresponding mapping complex, which coincides with $\mathrm{RHom}(E,F)$ in case $\cC = \Dperf(X)$ for a variety $X$. 

\subsection{Acknowledgements} 
We thank Arend Bayer, Soheyla Feyzbakhsh, Sasha Kuznetsov, Emanuele Macr\`{i}, and Shizhuo Zhang for interesting discussions related to this work.


\section{Group actions on categories} \label{sec_groupactiononcat}

In this section, we cover some preliminaries about group actions on categories and their (co)invariants. 
In particular, we compare the invariants of a group action on an enhanced category and its homotopy category, and we recall that for abelian groups the formation of invariants gives rise to a duality on categories equipped with group actions (subject to an assumption on the characteristic of the ground field). 

Throughout this section, we fix a finite group $G$ and a field $k$ (often assumed to have characteristic prime to $|G|$, but we state this explicitly when it is needed). 

\subsection{Higher-categorical actions}  
If $X$ is an object of a $1$-category $\cD$,  there is only one reasonable definition of a $G$-action on $X$: a group homomorphism $\phi \colon G \to \Aut(X)$. If $X$ is an object of a $2$-category $\cD$, then it is more natural to define a $G$-action to consist of an automorphism $\phi_g$ of $X$ for every $g \in G$, together with isomorphisms $\phi_{g,h} \colon \phi_g \circ \phi_h \cong \phi_{g \cdot h}$ for every $g, h \in G$ which satisfy a cocycle condition. Taking this to the limit, one arrives at the notion of a $G$-action on an object in an $\infty$-category, as well as the corresponding notions of invariants and coinvariants: 

\begin{definition}
\label{definition-G-action} 
Let $X$ be an object of an $\infty$-category $\cD$. 
A \emph{$G$-action} on $X$ is a functor $\phi \colon BG \to \cD$ such that $\phi(*) = X$, where $BG$ denotes the classifying space of $G$ regarded as an $\infty$-groupoid and $* \in BG$ is the unique object. 

Given a $G$-action $\phi$ on $X$, the  \emph{$G$-invariants} $X^G$ and \emph{$G$-coinvariants} $X_G$ are defined by 
\begin{equation*}
    X^G = \lim(\phi) \quad \text{and} \quad 
    X_G = \colim(\phi) 
\end{equation*}
provided the displayed limit and colimit exist. 
\end{definition}

\begin{remark}
We refer to \cite[\S3.1]{BP} which spells out how Definition~\ref{definition-G-action}  recovers the classical notions of actions on objects in $1$- or $2$-categories. 
\end{remark} 

\subsection{Linear categories} 
In this paper, we will be interested in $G$-actions on categories $\cC$ which are linear over a field $k$. 
There are several relevant settings: 

\begin{example}[Classical linear categories]
\label{example-classical-linear}
Suppose $\cC$ is a \emph{classical $k$-linear category}, i.e. a $1$-category enriched in $k$-vector spaces which admits finite products and coproducts. 
Equivalently, $\cC$ is a $1$-category equipped with a module structure over the symmetric monoidal category of finite-dimensional $k$-vector spaces $\Vect_k^{\mathrm{fd}}$. 
Then we can regard $\cC$ as an object of the $(2,1)$-category $\Cat^{\cl}_k$
with objects classical $k$-linear categories, $1$-morphisms $k$-linear functors, and $2$-morphisms isomorphisms of functors.  
The category $\Cat^{\cl}_k$ admits all limits and colimits (see e.g. \cite[\S4.2.3]{HA}), so the categories $\cC^G$ and $\cC_G$ exist for any $G$-action $\phi$ in this setting. 

For instance, let us recall the explicit description of $\cC^G$. The action $\phi$ consists of autoequivalences $\phi_g$ of $\cC$ for every $g \in G$ and compatible 
isomorphisms $\phi_{g,h} \colon \phi_g \circ \phi_h \cong \phi_{g \cdot h}$. Then the objects of $\cC^G$ are ``linearized objects'' of $\cC$, i.e. pairs $(E, \theta)$ where $E \in \cC$ and $\theta$ is a collection of isomorphisms $\theta_g \colon E \cong \phi_g(E)$ for $g \in G$ such that the diagram 
\begin{equation*}
    \xymatrix{
E \ar[rr]^{\theta_{gh}}\ar[d]_{\theta_g} && \phi_{gh}(E) \ar[d]^{\phi_{g,h}} \\ 
\phi_{g}(E) \ar[rr]^{\phi_{g}\theta_h} && \phi_g\phi_h(E) 
    }
\end{equation*}
commutes for all $g,h \in G$. 
\end{example} 

\begin{example}[Triangulated categories]
\label{example-triangulated}
Suppose $\cC$ is a $k$-linear triangulated category. Then we can regard $\cC$ as an object of the $(2,1)$-category $\Cat_k^{\tr}$ with objects $k$-linear triangulated categories, $1$-morphisms triangulated $k$-linear functors, and $2$-morphisms isomorphisms of functors. 
However, $\Cat_k^{\tr}$ does not admit all limits and colimits, and in particular it is not clear when $\cC^G$ exists in this setting. 

Some positive results were proved by Elagin \cite{Elagin}. 
Namely, we can form $\cC^G$ as a classical $k$-linear category, and ask when $\cC^G$ is triangulated. The shift functor on $\cC$ induces one on $\cC^G$, and we can consider the class of triangles in $\cC^G$ which become distinguished under the forgetful functor $\cC^G \to \cC$. 
Elagin proved that this gives $\cC^G$ the structure of a triangulated category, 
assuming the existence of a DG-enhancement of $\cC$ and invertibility of $|G|$ in $k$ \cite[Corollary 6.10]{Elagin}. 
For this and other reasons, it is natural to instead work with enhanced categories\footnote{Below for technical reasons we work with stable $\infty$-categories instead of DG categories, but in a precise sense these notions are equivalent \cite{dg-versus-stable}.} from the outset.
\end{example} 

\begin{example}[Linear categories]
\label{example-linear-categories} 
Suppose $\cC$ is a small idempotent-complete stable $\infty$-category equipped with a $\Dperf(k)$-module structure; for short, we simply refer to such a $\cC$ as a \emph{$k$-linear category}. 
The collection of all $k$-linear categories (with morphisms between them the exact $k$-linear functors) is organized into an $\infty$-category $\Cat_k^{\stable}$. 
The $\infty$-category $\Cat_k^{\stable}$ admits all limits and colimits (see \cite[\S2.1]{akhil-galois}), so the categories $\cC^G$ and $\cC_G$ exist for any $G$-action in this setting. 
\end{example} 

\begin{example}
\label{example-XmodG}
There is a simple description of $\cC^G$ in the case where $X$ is a scheme with a $G$-action and $\cC = \Dperf(X)$ with the induced action. 
Namely, descent along $X \to [X/G]$ shows that 
$\Dperf(X)^G \simeq \Dperf([X/G])$, where $[X/G]$ is the quotient stack. 
\end{example}

Taking the homotopy category of a $k$-linear category gives a functor $\Cat_k^{\stable} \to \Cat_k^{\tr}$. 
Essentially every $k$-linear triangulated category $\cT$ appearing in nature, and in particular any that appears in this paper, admits a natural \emph{enhancement}, i.e. a $k$-linear category $\cC$ with $\cT$ as its homotopy category. 
Moreover, 
given a $G$-action on a $k$-linear category $\cC$ there is an induced $G$-action on its homotopy category, and 
under mild hypotheses a $G$-action on a $k$-linear triangulated category $\cT$ can be lifted to a $G$-action on an enhancement $\cC$ \cite[Corollary 3.4]{BP}. 
In this situation, it is useful (both for psychological and practical reasons) to have a comparison between the homotopy category of $\cC^G$ and the $G$-invariants of the homotopy category computed in $\Cat^{\cl}_k$. 
We prove such a result below; see \cite[Theorem 8.7]{Elagin} and \cite{sosna} for related statements in the DG setting. 
Let $\rh \colon \Cat_k^{\stable} \to \Cat_k^{\cl}$ be the functor given by taking the homotopy category (so it is the composition of the functor $\Cat^{\stable}_k \to \Cat^{\tr}_k$ with the forgetful functor $\Cat^{\tr}_k \to \Cat^{\cl}_k$). 

\begin{proposition}
\label{proposition-enhancement-CG}
Assume that the order of the group $G$ is invertible in the field $k$. Then for any $k$-linear category $\cC$ with a $G$-action, there is a natural equivalence $\rh(\cC^G) \simeq (\rh \cC)^G$. 
\end{proposition}

\begin{remark}
Since $\cC^G$ is a $k$-linear category, its homotopy category $\rh( \cC^G)$ inherits a natural triangulated structure. 
The corresponding triangulated structure on $(\rh\cC)^G$ via the equivalence of the proposition can be described explicitly: the shift functor is induced by that on $\rh\cC$, and a triangle in $(\rh\cC)^G$ is distinguished if and only if it becomes distinguished under the forgetful functor $(\rh\cC)^G \to \rh \cC$ (cf. Example~\ref{example-triangulated}). 
\end{remark}

\begin{remark}
As we will see in the proof, there are also equivalences 
\begin{equation*} 
\cC^G \simeq \cC_G \qquad \text{and} \qquad (\rh \cC)^G \simeq (\rh \cC)_G,
\end{equation*} 
so the assertion of Proposition~\ref{proposition-enhancement-CG} also holds for coinvariants in place of invariants. 
\end{remark}

\begin{proof}
Note that there is a canonical morphism $\rh(\cC^G) \to (\rh \cC)^G$ induced by the universal property of $(\rh \cC)^G$ as a limit, which we  need to prove is an equivalence. 
In other words, if $\phi \colon BG \to \Cat_k^{\stable}$ is the given $G$-action on $\cC$, we want to prove that $\rh \colon \Cat_k^{\stable} \to \Cat_k^{\cl}$
commutes with taking the limit along $\phi$. 
To do so, we will consider a factorization of the functor $\rh$. 

By definition, 
\begin{equation*}
    \Cat_k^{\stable} = \Mod_{\Dperf(k)}(\Cat^{\stable})
\end{equation*}
is the category of modules over the commutative algebra object $\Dperf(k) \in \Cat^{\stable}$, where $\Cat^{\stable}$ is the $\infty$-category of small idempotent-complete stable $\infty$-categories. 
Similarly, 
\begin{equation*}
    \Cat_k^{\cl} =     \Mod_{\Vect_k^{\fd}}(\Cat^{\cl})
\end{equation*}
where $\Cat^{\cl}$ denotes the $(2,1)$-category of small $1$-categories. 
It will be convenient to consider two intermediate variants of these constructions. 
Namely, if $\Cat$ denotes the $\infty$-category of small  $\infty$-categories, 
then both $\Dperf(k)$ and $\Vect_k^{\fd}$ can be considered as commutative algebra objects in $\Cat$, so we can consider the corresponding categories of modules. 

Now consider the factorization 
\begin{equation} 
\label{factorization-h} 
    \rh \colon \Cat_k^{\stable} \xrightarrow{\, a \,} 
    \Mod_{\Dperf(k)}(\Cat) \xrightarrow{\, b \,} 
    \Mod_{\Vect_k^{\fd}}(\Cat) \xrightarrow{\, c \,}  
    \Cat_k^{\cl}
\end{equation}
where $a$ is the obvious inclusion, 
$b$ is given by restriction along the symmetric monoidal functor $\Vect_k^{\fd} \to \Dperf(k)$, and $c$  is given by taking homotopy categories. We will use the following observations about limits and colimits in the above situation: 
\begin{enumerate}
\item All of the categories appearing in the factorization~\eqref{factorization-h} admit all limits and colimits; see \cite[\S4.2.3]{HA}. 
In particular, for an object in one of these categories with a $G$-action, we may form the invariants and coinvariants. 
\item \label{norm-equivalence} Let $\cD$ be an idempotent-complete category contained in one of the classes of categories appearing in~\eqref{factorization-h}, and assume $\cD$ is equipped with a $G$-action. 
Then there is an equivalence $\cD_G \simeq \cD^G$ given by the norm functor; see \cite[Proposition 3.4]{perry-HH} where this is proved for $\cD \in \Cat_k^{\stable}$, but the same argument works for $\Mod_{\Dperf(k)}(\Cat)$, $\Mod_{\Vect_{k}^{\fd}}(\Cat)$, and $\Cat_k^{\cl}$. This is where we use the assumption that $|G|$ is invertible in $k$. 
\item \label{a-limits} The functor $a$ preserves limits; see \cite[Theorem 1.1.4.4]{HA} where this is proved in the non-$k$-linear setting, but the situation over $k$ is analogous. 
\item \label{b-limits} The functor $b$ preserves limits by \cite[Corollary 4.2.3.3]{HA}. 
\item \label{c-colimits} The functor $c$ preserves colimits. Indeed, $c$ is the left adjoint to the nerve functor. 
\end{enumerate} 

The proposition now follows formally from the above observations. 
Indeed, let $\cD = (b\circ a)(\cC)$ with its induced $G$-action. 
Then we have 
\begin{equation*}
    \rh(\cC^G) = c((b \circ a)(\cC^G))
    \underset{\eqref{a-limits}, \eqref{b-limits}}{\simeq} c( \cD^G ) 
    \underset{\eqref{norm-equivalence}}{\simeq} c( \cD_G ) 
    \underset{\eqref{c-colimits}}{\simeq} \rh(\cC)_G 
    \underset{\eqref{norm-equivalence}}{\simeq} \rh(\cC)^G, 
\end{equation*}
where the subscripts indicate the invoked observation from above. 
\end{proof}

\subsection{Reconstruction theorem} 
\label{section-reconstruction} 
Let $\cC$ be a $k$-linear category with $G$-action. 
By functoriality of the operation of taking $G$-invariants, the $k$-linear structure of $\cC$ induces the structure of a $\Dperf(k)^G$-module on $\cC^G$. 
Note that $\Dperf(k)^G$ is identified with the derived category of finite-dimensional $G$-representations. 
In particular, if $G^{\vee} = \Hom(G, k^{\times})$ is the dual group of characters of $G$, then $\cC^G$ inherits a \emph{residual $G^{\vee}$-action}. 
Concretely, for $\chi \in G^{\vee}$ the corresponding autoequivalence is 
\begin{equation*}
    - \otimes \chi \colon \cC^G \to \cC^G, \quad (E, \theta) \mapsto (E, \theta \otimes \chi),
\end{equation*}
where $\theta \otimes \chi$ is the linearization on $E$ given by $(\theta \otimes \chi)_g=\chi(g)\theta_g$ for every $g \in G$. 
For abelian groups, the category $\cC$ together with its $G$-action can be reconstructed from $\cC^G$ with its $G^{\vee}$-action: 

\begin{theorem}[{\cite[Theorem 4.2]{Elagin}}] \label{thm_elagin}
Assume that $G$ is abelian and $k$ is algebraically closed of characteristic prime to $|G|$. 
Let $\cC$ be a $k$-linear category with a $G$-action. 
Then there is an equivalence 
\begin{equation*}
 \rho \colon  (\cC^G)^{G^{\vee}} \xrightarrow{\sim} \cC
\end{equation*}
of $k$-linear categories, under which the $G$-action on $\cC$ is identified with the $G \cong (G^\vee)^{\vee}$-action on the left-hand side. 
\end{theorem}

\begin{proof}
In \cite[Theorem 4.2]{Elagin} this is proved for $\cC$ an idempotent-complete classical $k$-linear category, but the analogous arguments work in the setting of $k$-linear categories. 
\end{proof}

Recall that there are mutually left and right adjoint functors between $\cC^G$ and $\cC$: the \emph{forgetful functor}  
\begin{equation*}
    \Forg_G \colon \cC^G \to \cC 
\end{equation*}
given by forgetting the $G$-equivariant structure, 
and the \emph{inflation functor} 
\begin{equation*}
    \Inf_G \colon \cC \to \cC^G 
\end{equation*}
given on objects by $E \mapsto \bigoplus_{g \in G} \phi_g(E)$,  where $\phi_g$ are the autoequivalences of $\cC$ given by the $G$-action and the sum is given the natural $G$-linearization (see \cite[Lemma 3.8]{Elagin}, \cite[\S3.2]{perry-HH}). 
When it is clear from context, we will omit $G$ from the notation and write $\Forg$ and $\Inf$. 
For later use, we make an observation about the interaction of these functors with the equivalence of Theorem~\ref{thm_elagin}.  

\begin{lemma}
\label{lemma_generalnonsense}
Under the assumptions of Theorem \ref{thm_elagin}, there is an equivalence of functors
\begin{equation*}
   \Forg_{G^{\vee}} \simeq \Inf_G  \circ \rho \colon 
   (\cC^G)^{G^{\vee}} \to \cC^G. 
\end{equation*} 
\end{lemma}
\begin{proof}
The equivalence $\rho$ is the composition of two comparison functors $(\cC^G)^{G^\vee} \simeq (\cC^G)_{T(q^*,q_*)}$ and $\cC \simeq (\cC^G)_{T(p_*,p^*)}$, where 
$$p_*=\Forg_G, \quad p^*=\Inf_G, \quad q_*=\Forg_{G^{\vee}}, \quad q^*=\Inf_{G^{\vee}},$$
$(\cC^G)_{T(q^*,q_*)}$ is the category of comodules over the comonad $T(q^*,q_*)$ on $\cC^G$, $(\cC^G)_{T(p_*,p^*)}$ is the category of comodues over the comonad $T(p_*,p^*)$ on $\cC^G$, together with an equivalence $(\cC^G)_{T(q^*,q_*)} \simeq (\cC^G)_{T(p_*,p^*)}$. By the Comparison Theorem (see \cite[Proposition 2.6]{Elagin}), we have that $p_* \circ \rho \simeq q^*$, which in our notation means $\Inf_G \circ \rho \simeq \Forg_{G^{\vee}}$.
\end{proof}


\section{Enriques categories}  \label{sec_enriques}

In this section we define Enriques categories, give some important examples, and recall a related result about derived categories of branched double covers. 

\subsection{Semiorthogonal decompositions} \label{subsec_preliminarydef}
Since many examples of Enriques categories arise as components in semiorthogonal decompositions, 
we very briefly review some of the basic definitions; see for instance \cite{bondal,Bondal:representable-functors-serre-functor} for more details. 

Let $\cC$ be a $k$-linear category. A \emph{semiorthogonal decomposition} of $\cC$ consists of a collection $\cC_1, \dots, \cC_m$ of full triangulated subcategories of $\cC$ such that:
\begin{enumerate}
    \item $\Hom(F,G)=0$ for every $F \in \cC_i$, $G \in \cC_j$, $i>j$; 
    \item For every $F \in \cC$, there is a sequence of morphisms
		\begin{equation*}  
		0 = F_m \to F_{m-1} \to \cdots \to F_1 \to F_0 = F,
		\end{equation*}
		such that the cone of $F_i \to F_{i-1}$ is in $\cC_i$ for $1 \leq i \leq m$. 
\end{enumerate}
A semiorthogonal decomposition is denoted by $\cC=\langle \cC_1, \dots, \cC_m \rangle$. 

Examples of semiorthogonal decompositions arise from exceptional collections in $\cC$. An object $E \in\cC$ is \emph{exceptional} if $\cHom_k(E,E) \simeq k[0]$ is a copy of $k$ in degree $0$. 
An \emph{exceptional collection} is a sequence of exceptional objects $E_1,\dots,E_m$ in $\cC$ satisfying $\cHom_k(E_i,E_j)=0$ for all $i>j$. In this case, we have a semiorthogonal decomposition
$$\cC=\langle \cD, E_1, \dots, E_m \rangle$$
where $\cD:=\langle E_1, \dots, E_m \rangle^\perp=\set{ F \in \cC \st \cHom_k(E_i, F)=0 \text{ for all } 1 \leq i \leq m }$ is the right orthogonal to the exceptional collection.

If $E \in \cC$ is an exceptional object, then the \emph{left mutation functor} $\rL_{E} \colon \cC \to \langle E \rangle^{\perp}$ is defined by the exact triangle
\begin{equation*}
  \cHom_k(E,F) \otimes E \to F \to \rL_E(F).
\end{equation*} 
Analogously the \emph{right mutation functor} $\rR_E \colon \cC \to { }^\perp\langle E \rangle$ is defined by the exact triangle 
\begin{equation*}
    \rR_E(F) \to F \to \cHom(F, E)^\vee \otimes E .
\end{equation*}
More generally, given an admissible subcategory $\cA \subset \cC$, i.e. a full triangulated subcategory whose inclusion admits left and right adjoints, one can similarly define mutation functors $\rL_{\cA}$ and $\rR_{\cA}$ through $\cA$. 

\subsection{Definition and examples of Enriques categories} 
Below we consider smooth proper $k$-linear categories, where $k$ is a field. 
We refer to \cite[\S4]{NCHPD} for background on this notion, but note that if $X$ is a smooth proper $k$-variety and $\cC$ is a semiorthogonal component of $\Dperf(X)$, then $\cC$ is a smooth proper $k$-linear category \cite[Lemma 4.9]{NCHPD}.
Further, recall that if $\cC$ is a proper $k$-linear category, then a \emph{Serre functor} is an autoequivalence $\rS_{\cC}$ such that there are functorial equivalences 
\begin{equation*}
    \cHom_k(E, \rS_{\cC}(F)) 
    \simeq \cHom_k(F, E)^{\vee}
\end{equation*}
for $E, F \in \cC$. 
A Serre functor is unique when it exists, and it exists for any smooth proper $k$-linear category, e.g. any semiorthogonal component of a smooth proper variety \cite{Bondal:representable-functors-serre-functor}. 

\begin{definition}
\label{definition-CY2-Enriques}
Let $\cC$ be a smooth proper $k$-linear category. 
\begin{enumerate}
\item $\cC$ is an \emph{Enriques category} if it is equipped with a $\bZ/2$-action whose generator $\tau$ is a nontrivial autoequivalence of $\cC$ satisfying $\rS_{\cC} \simeq \tau \circ [2]$. 
\item $\cC$ is a \emph{$2$-Calabi--Yau (CY2) category} if $\rS_{\cC} \simeq [2]$. 
If $\cC$ further has Hochschild homology isomorphic to that of a K3 surface, then we 
say that $\cC$ is a \emph{K3 category}. 
\end{enumerate}
\end{definition}

Although we have made the definition of an Enriques category without any restrictions on the characteristic of $k$, 
in the paper we will often avoid pathologies by assuming $\mathrm{char}(k) \neq 2$. 

\begin{remark} 
In the literature, sometimes more assumptions are imposed in the definition of a CY2 category. 
For instance, in \cite[Definition~6.1]{IHC-CY2} it is required that: 
\begin{enumerate}
    \item $\cC$ is \emph{connected} as a $k$-linear category, in the sense that its Hochschild cohomology satisfies 
$\HH^0(\cC) = k$ and $\HH^{<0}(\cC) = 0$; and \item $\cC$ is \emph{geometric}, in the sense that it can be realized as a semiorthogonal component in $\Dperf(X)$ for a smooth proper variety $X$ over $k$. 
\end{enumerate}
We have taken the least restrictive definition as many arguments work in that generality.
\end{remark} 

\begin{remark}
Let us emphasize that in general it is an extra condition for an involutive autoequivalence to lift to a $\bZ/2$-action. 
Namely, suppose that $\cC$ is a connected $k$-linear category 
and let $\tau$ be an autoequivalence of $\cC$ such that $\tau \circ \tau \simeq \id_{\cC}$. 
Then there is an obstruction $\mathrm{ob}(\tau) \in \rH^3(\bZ/2, k^{\times})$ ($=\bZ/2$ when $k$ is algebraically closed of characteristic different from $2$) which vanishes if and only if there exists a $\bZ/2$-action on $\cC$ with generator $\tau$, in which case this action is unique up to equivalence; see \cite[Corollary 3.4 and Remark 4.2]{BP}. 
All of the examples of Enriques categories $\cC$ considered in this paper will be connected, so there will be a unique $\bZ/2$-action with generator the involutive autoequivalence $\tau = \rS_{\cC} \circ [-2]$. 
\end{remark}

Passing to $\bZ/2$-invariants leads to a correspondence between Enriques and CY2 categories, summarized in the following lemma; note that the second statement is just a special case of Theorem~\ref{thm_elagin}. 

\begin{lemma}[{\cite[Lemmas 4.5 and 4.6]{BP}}]
\label{lemma-Enriques-CY2}
Let $\cC$ be an Enriques category over an algebraically closed field $k$ of characteristic prime to $2$. 
Let $\cD = \cC^{\bZ/2}$ be the invariant category for the $\bZ/2$-action on $\cC$. 
\begin{enumerate}
    \item $\cD$ is a CY2 category, called the \emph{CY2 cover} of $\cC$. 
    \item \label{D-residual}
  $\cD$ is equipped with a natural $\bZ/2$-action, called the \emph{residual $\bZ/2$-action}, such that there is an equivalence $\cC \simeq \cD^{\bZ/2}$. 
\end{enumerate}
\end{lemma}

\begin{remark} 
In the definition of an Enriques category $\cC$, it would be natural to require that its CY2 cover is a K3 category. 
Indeed, Lemma~\ref{lemma-Enriques-CY2}\eqref{D-residual} would then say $\cC$ is a ``$\bZ/2$-quotient'' (i.e. $\bZ/2$-invariant category) of a K3 category, parallel to the situation for a geometric Enriques surface. 
We have not included this as part of the definition as it is not necessary for our arguments, 
but it will be satisfied in all of the examples of interest in the paper. 
\end{remark} 

We define an Enriques category over a base scheme by the requirement that its fibers are so. 
To formulate this, we use the formalism of categories linear over a base scheme $S$, of which the $k$-linear categories of Example~\ref{example-linear-categories} are the special case when $S = \Spec(k)$; see \cite[\S2]{IHC-CY2} for a quick summary of the key points of this theory. 
In particular, for any $S$-linear category $\cC$ and morphism $T \to S$, we can form the base changed $T$-linear category $\cC_T$, compatible with semiorthogonal decompositions. 
For a point $s \in S$, this gives a $\kappa(s)$-linear category $\cC_s$, called the \emph{fiber} of $\cC$ over $s$. 

\begin{definition}
\label{definition-family-Enriques}
Let $S$ be a scheme and let $\cC$ be a smooth proper $S$-linear category. 
We say that $\cC$ is an \emph{Enriques category over $S$} if for every point $s \in S$, the fiber $\cC_s$ is an Enriques category over $\kappa(s)$. 
\end{definition}

\begin{remark}
There are other possible notions of Enriques categories over a base. 
For instance, a more global definition would require the existence of an $S$-linear $\bZ/2$-action on $\cC$ such that the relative Serre functor is given by $\tau \circ [2]$ for the generator $\tau$ of the action. 
For our purposes in this paper, the fibral definition will suffice. 
\end{remark}

\subsection{Examples} \label{subsection_examples}
Below we identify some interesting Enriques categories and their CY2 covers, assuming $\mathrm{char}(k) \neq 2$ throughout. 

\begin{example}[Enriques surfaces]
\label{example-enriques-surface} 
Let $S$ be an Enriques surface. 
Then $\Db(S)$ is an Enriques category, with $\bZ/2$-action given by tensoring with the canonical bundle $\omega_S$. 
If $T \to S$ is the corresponding K3 double cover of $S$, then $\Db(T) \simeq \Db(S)^{\bZ/2}$ is the corresponding CY2 cover, with residual $\bZ/2$-action given by the covering involution. 
\end{example} 

\begin{example}[Quartic double solids] \label{example-quartic-ds}
Let $Y \to \bP^3$ be a quartic double solid. 
Then the Kuznetsov component $\Ku(Y):=\langle \cO_Y, \cO_Y(1) \rangle^\perp \subset \Db(Y)$ is a connected Enriques category, with $\bZ/2$-action given by the covering involution, see \cite[Proposition 4.14]{BP}. 
If $Y_{\br} \subset \bP^3$ is the branch locus of the double cover $Y \to \bP^3$, then 
$\Db(Y_{\br}) \simeq \Ku(Y)^{\bZ/2}$ is the CY2 cover, and the residual $\bZ/2$-action can be described in terms of the spherical twist around $\cO_{Y_{\br}}$, see \cite{KuzPer_cycliccovers} and \cite[Theorem 4.15]{BP}. 
\end{example}

\begin{example}[Gushel--Mukai varieties] \label{ex_GM}
A \emph{Gushel--Mukai (GM) variety} of dimension $n$, $2 \leq n \leq 6$, is a smooth intersection 
$$X= \text{CGr}(2,5) \cap Q \subset \bP^{10},$$
where $\text{CGr}(2,5)$ is the cone over the Grassmannian $\Gr(2, 5)$ embedded via the Pl\"ucker embedding in a $10$-dimensional projective space $\bP^{10}$, and $Q$ is a quadric hypersurface in a linear subspace $\bP^{n+4} \subset \bP^{10}$.
If $n=2$ then $X$ is a Brill-Noether general degree $10$ polarized K3 surface, 
while for $n \geq 3$
GM varieties coincide with the class of smooth Fano varieties of Picard number $1$, degree $10$ and coindex $3$ \cite{gushel, mukai, DebKuz:birGM}.

The linear projection from the vertex $\nu$ of $\text{CGr}(2,5)$ restricts to a well defined morphism $\gamma_X \colon X \to\Gr(2,5)$. We say that $X$ is an \emph{ordinary} GM variety if $\nu$ does not belong to the linear space $\bP^{n+4}$. In this case, $X$ is isomorphic to a quadric section of a linear section of $\Gr(2,5)$. If $\nu$ belongs to $\bP^{n+4}$, then $X$ is a \emph{special} GM variety. In this case, $\gamma_X$ is a double cover of a linear section $M \coloneqq \Gr(2,5) \cap \bP^{n+3}$ ramified in a quadric section $X'$ of $M$, which is an ordinary GM variety of dimension $n-1$. We say that $X$ and $X'$ are \emph{opposite} GM varieties, and use the notation $X^{\op} \coloneqq X'$ and $(X')^{\op} \coloneqq X$ for the operation of passing to opposites. 

For technical reasons and ease of reference to the literature, we will assume $\mathrm{char}(k) = 0$ in our discussion of derived categories of GM varieties below. 
The Kuznetsov component is defined by the semiorthogonal decomposition 
\begin{equation*}
\Db(X)= \langle \Ku(X), \cO_X, \cU_X^\vee, \dots, \cO_X((n-3)H), \cU^\vee_X((n-3)H)\rangle
\end{equation*}
where $\cU_X$ and $\cO_X(H)$ denote the pullbacks to $X$ of the tautological rank $2$ bundle and Pl\"{u}cker line bundle on $\Gr(2,5)$. 
Note that when $n = 2$ the exceptional collection is empty and $\Ku(X) = \Db(X)$. 
The categories $\Ku(X)$ were first introduced and studied in detail in \cite{KuzPerry:dercatGM}. 

If $X$ is an odd-dimensional GM variety, then the results of \cite{KuzPerry:dercatGM, KuzPer_cycliccovers} show that $\Ku(X)$ is a connected Enriques category, with $\bZ/2$-action given by the covering involution when $X$ is special. 
Moreover, the CY2 cover of $\Ku(X)$ is $\Ku(X^{\op})$, and the residual $\bZ/2$-action is given by the covering involution of $X^{\op}$ when $X$ is ordinary and in terms of certain mutation or spherical functors in general (see \cite[Theorem 4.15]{BP} for the explicit formula). 
\end{example} 

\begin{example}[Verra threefolds]
\label{example-verra-3fold}
A \emph{Verra threefold} is a smooth Fano threefold $V$ which is either a $(2,2)$-divisor $V \subset \bP^2 \times \bP^2$, or a double cover $V \to M$ of a $(1,1)$-divisor $M \subset \bP^{2} \times \bP^2$ branched along the K3 surface $Z \subset M$ given by the intersection with a $(2,2)$-divisor\footnote{Similar to the definition of a GM variety, we could also define $V$ more uniformly as a complete intersection in the cone over $\bP^2 \times \bP^2 \subset \bP^8$.}. 
We say $V$ is \emph{ordinary} in the first case, and \emph{special} in the second. 
In either case, if $\pi_1 \colon V \to \bP^2$ is the projection to the first $\bP^2$ factor, then 
there is a semiorthogonal decomposition 
\begin{equation*}
    \Db(V) = \langle \Ku(V), \pi_1^* \Db(\bP^2) \rangle , 
\end{equation*}
where $\Ku(V)$ is a connected Enriques category. 
Indeed, that $\Ku(V)$ is Enriques follows from the results of \cite{kuznetsov-CY}, while connectedness follows from the results of \cite{perry-HH} (cf. \cite[Proposition 4.14]{BP} and \cite[Proposition 2.12]{KuzPerry:dercatGM}). 

The results of \cite{KuzPer_cycliccovers} show that the CY2 cover of $\Ku(V)$ is $\Db(Z)$ if $V$ is special, or the Kuznetsov component of the Verra fourfold $W$ corresponding to $V$ if $V$ is ordinary. By definition, $W \to \bP^2 \times \bP^2$ is the double cover branched over $V$, with Kuznetsov component defined by 
\begin{equation*}
\Db(W) = \langle \Ku(W), \pi_1^* \Db(\bP^2) , \pi_1^*\Db(\bP^2) \otimes \cO(H_2) \rangle, 
\end{equation*}
where $\pi_1 \colon W \to \bP^2$ is the projection to the first $\bP^2$ factor and $H_2$ is the pullback of the hyperplane class on the second $\bP^2$ factor. 
In fact, $\Ku(W) \simeq \Db(S, \alpha)$ for a degree $2$ K3 surface $S$ with a $2$-torsion Brauer class $\alpha$. Indeed, this follows from Kuznetsov's results on the derived category of a quadric fibration \cite{kuznetsov08quadrics}; see \cite{rational-point-twisted} for details.  
\end{example} 

Using the results of \cite{KuzPer_cycliccovers} and \cite{kuznetsov-CY}, more examples of geometrically interesting Enriques categories can be constructed by taking a double cover branched along a variety with a CY2 Kuznetsov component. Let us give just one more example. 

\begin{example}
\label{example-other-enriques-cats}
Let $T \to \bP^2$ be a K3 surface of degree $2$, i.e. a double cover of $\bP^2$ branched along a smooth sextic curve $C$. 
We can also regard $T$ as a degree $6$ divisor $T \subset \bP(1,1,1,3)$; 
here and below, we use $\bP(w_0, w_1, \dots, w_n)$ to denote the weighted projective stack with weights $w_i$, i.e. the Deligne--Mumford stack given as the quotient of $\bA^{n+1}_k \setminus \{ 0 \}$ by $\bG_m$ acting with weight $w_i$ on the $i$-th coordinate. 
Let $Y \to \bP(1,1,1,3)$ be the double cover branched along $T$; as above, we can also regard $Y \subset \bP(1,1,1,3,3)$ as a degree $6$ divisor. 
Then $Y$ is a stacky Fano threefold with a semiorthogonal decomposition 
\begin{equation*}
    \Db(Y) = \langle \Ku(Y), \cO_Y, \cO_Y(1), \cO_Y(2) \rangle, 
\end{equation*}
and $\Ku(Y)$ is a connected Enriques category with CY2 cover $\Db(T)$. 

By \cite[Theorem 1.1]{KuzPer_cycliccovers} the K3 cover of $\Ku(Y)$ is $\Db(T)$, and by \cite[Proposition 1.1 and Theorem 7.7]{KuzPer_cycliccovers} the generator of the residual $\bZ/2$-action on $\Db(T)$ is $i_*$, where $i$ is the involution of the double cover $T \to \bP^2$. 
It follows that $\Ku(Y) \simeq \Db(T)^{\bZ/2} \simeq \Db([T/i])$, where $[T/i]$ is the quotient stack. 
Moreover, if $f \colon [T/i] \to \bP^2$ is the map induced by the double cover and $j \colon C \to [T/i]$ is the map obtained from the inclusion of $C$ in $T$ as the ramification divisor, then by \cite[Theorem 4.1]{KuzPer_cycliccovers} there is a semiorthogonal decomposition 
\begin{equation*}
\Ku(Y) \simeq \Db([T/i]) = \langle 
f^*\Db(\bP^2), j_*\Db(C) \rangle ,
\end{equation*}
where $f^*$ and $j_*$ are fully faithful. 
Altogether, we conclude that the Enriques category $\Ku(Y)$ admits a semiorthogonal decomposition with a copy of the derived category of a sextic plane curve and $3$ exceptional objects. 
\end{example} 

\begin{remark} \label{rmk_families_qdsGM}
It is easy to show that the above examples work in families. For instance, suppose that $f \colon \cY \to S$ is a smooth proper morphism equipped with a line bundle $\cO_{\cY}(1)$ such that the fibers $\cY_s$ for $s \in S$ are quartic double solids with  $\cO_{\cY_s}(1)$ the pullback of $\cO_{\bP^3}(1)$. 
Then there is a $S$-linear semiorthogonal decomposition 
\begin{equation*}
    \Dperf(\cY) = \langle \Ku(\cY), f^*\Dperf(S), f^*\Dperf(S) \otimes \cO_{\cY}(1) \rangle 
\end{equation*} 
where $\Ku(\cY)$ is an Enriques category over $S$ with fibers $\Ku(\cY_s)$. 
For families of GM varieties, see \cite[Lemma 5.9(1)]{BP}.  
\end{remark}

\subsection{Kuznetsov components of branched double covers} \label{subsec_generalsetting}

The descriptions of the CY2 covers of Enriques categories listed in Section \ref{subsection_examples} are special cases of a general result proved in \cite{KuzPer_cycliccovers} about semiorthogonal decompositions of certain ramified cyclic covers.  
In this section, we review this result in the case of double covers, fixing the notation we will need later in the paper. 

Let $M$ be a variety with a line bundle $\cO_M(1)$. Assume that the bounded derived category of $M$ admits a rectangular Lefschetz decomposition with respect to $\cO_M(1)$ of length $m$. Explicitly, there exists an admissible subcategory $\cB$ of $\Db(M)$ sitting in a semiorthogonal decomposition of the form
$$\Db(M)= \langle \cB, \cB(1), \dots, \cB(m-1)\rangle,$$
where $\cB(i) \coloneqq \cB \otimes \cO_M(1)$. 

Let $d$ be a positive integer such that $2d \leq m$. Let $f \colon Y \to M$ be a double cover ramified in a Cartier divisor $X$ in the linear system $|\cO_M(2d)|$. In particular, we are in the setting of \cite{KuzPer_cycliccovers} in the special case $n=2$. 

We denote by $j \colon X \hookrightarrow Y$ the inclusion of the ramification divisor and by $i \colon X \hookrightarrow M$ the inclusion of the branch divisor, which fit into the commutative diagram
$$
\xymatrix{
& Y \ar^{f}[d] \\
X \ar@{^{(}->}[r]_i \ar@{^{(}->}[ru]^j & M.
}
$$
By \cite[Lemma 5.1]{KuzPer_cycliccovers} the pullback $f^*$ is fully faithful on $\cB$, and if $\cB_Y$ denotes the essential image of $\cB$ via $f^*$, then there is a semiorthogonal decomposition of the form
$$\Db(Y)=\langle \Ku(Y), \cB_Y,\dots, \cB_Y(m-d-1) \rangle,$$
where $\Ku(Y):=\langle \cB_Y,\dots, \cB_Y(m-d-1)\rangle^{\perp}$. Analogously, by \cite[Lemma 5.5]{KuzPer_cycliccovers} there is a semiorthogonal decomposition
$$\Db(X)=\langle \Ku(X), \cB_X,\dots, \cB_X(m-2d-1) \rangle,$$
where $\cB_X$ is the essential image of $\cB$ via $i^*$ and $\Ku(X):=\langle \cB_X,\dots, \cB_X(m-2d-1)\rangle^{\perp}$ (if $m=2d$, then $\Ku(X)=\Db(X)$).
The group $\bZ/2$ acts on $Y$ over $X$ via the covering automorphism, and induces an action of $\bZ/2$ on $\Db(Y)$ which preserves $\Ku(Y)$. 

\begin{theorem}[\cite{KuzPer_cycliccovers}, Theorem 1.1]
\label{theorem-cyclic-cover}
In the above setting, there is an equivalence
$$\Ku(X) \simeq \Ku(Y)^{\bZ/2}.$$     
\end{theorem}

Let us recall the definition of the functor realizing this equivalence. 
We equip $X$ with the trivial $\bZ/2$-action, so that there is a completely orthogonal decomposition 
\begin{equation*}
\Db(X)^{\bZ/2} = \langle \Db(X) \otimes 1, \Db(X) \otimes \chi \rangle,
\end{equation*}
where $1$ and $\chi$ are the trivial and nontrivial characters of $\bZ/2$. 
The morphism $j \colon X \to Y$ is $\bZ/2$-equivariant, so there is a pushforward functor $j^{\bZ/2}_* \colon \Db(X)^{\bZ/2} \to \Db(Y)^{\bZ/2}$ on equivariant categories. 
Let $j_{0*} \colon \Db(X) \to \Db(Y)^{\bZ/2}$ be the restriction to the first component in the above orthogonal decomposition, given by $j_{0*}(E) = j^{\bZ/2}_*(E \otimes 1) = j_*(E) \otimes 1$. 

 Next, note that $\bZ/2$-action on $\Db(Y)$ induces a trivial action on $\cB_Y$; hence, as above, we have a completely orthogonal decomposition $\cB_Y^{\bZ/2} = \langle \cB_Y \otimes 1, \cB_Y \otimes \chi \rangle$. 
Moreover, by \cite[Lemma 5.4]{KuzPer_cycliccovers} there is a semiorthognal decomposition 
\begin{equation*}
\Db(Y)^{\bZ/2} = \langle \Ku(Y)^{\bZ/2}, \cB_Y^{\bZ/2}, \dots, \cB_Y^{\bZ/2}(m-d-1) \rangle. 
\end{equation*}

Then by \cite[Proposition 6.5]{KuzPer_cycliccovers} the functor
\begin{equation}
\label{eq_Phi0}    
\Phi_0:=\rL_{\langle \cB_Y \otimes 1, \dots, \cB_Y(d-1) \otimes 1 \rangle} \circ j_{0*} \circ (- \otimes \cO_X(d)) \colon \Ku(X) \to \Ku(Y)^{\bZ/2}
\end{equation}
gives the equivalence of Theorem~\ref{theorem-cyclic-cover}.
For later use, we note a simple description of the composition of this equivalence with the forgetful functor $\Forg \colon \Ku(Y)^{\bZ/2} \to \Ku(Y)$. 

\begin{lemma}
\label{lemma1}
The composition $\Forg \circ \Phi_0$ is isomorphic to the functor 
$$\rL_{\langle \cB_Y, \dots, \cB_Y(d-1) \rangle} \circ j_* \circ (- \otimes \cO_X(d)): \Ku(X) \to \Ku(Y).$$
\end{lemma}
\begin{proof}
For $E \in \Db(X)$, we have 
\begin{equation*}
\Phi_0(E) = \rL_{\langle \cB_Y \otimes 1, \dots, \cB_Y(d-1) \otimes 1 \rangle}(j_*(E(d))\otimes 1). 
\end{equation*} 
Therefore, the claim follows from the observation that for $F \in \cB_Y$ and $G \in \Db(X)$, we have 
$\Hom(F \otimes 1, j_*(G) \otimes 1) \cong \Hom(F, j_*(G))$. 
\end{proof}

Let us also recall the definition of the following autoequivalences introduced in \cite[Section 7.2]{KuzPer_cycliccovers}, called \emph{rotation functors}\footnote{The rotation functors in our notation are the shift by $(-1)$ of those in \cite{KuzPer_cycliccovers}.}: 
$$\sO_Y:= \rL_{\cB_Y}(- \otimes \cO_Y(1))[-1]\colon \Ku(Y) \simeq \Ku(Y),$$
$$\sO_Y^{\bZ/2}:= \rL_{\cB_Y^{\bZ/2}}(- \otimes \cO_Y(1))[-1]\colon \Ku(Y)^{\bZ/2} \simeq \Ku(Y)^{\bZ/2},$$
$$\sO_X:= \rL_{\cB_X}(- \otimes \cO_X(1))[-1]\colon \Ku(X) \simeq \Ku(X).$$
By \cite[Theorem 7.7, Proposition 7.9]{KuzPer_cycliccovers} the rotation functors satisfy the following compatibility conditions:
\begin{equation} \label{eq_rotfunct_iota}
\sO_Y^d \simeq \iota[1-d],    
\end{equation}
\begin{equation} \label{eq_rotfunct_chi}
(\sO_Y^{\bZ/2})^d \simeq (-\otimes \chi)[1-d],  
\end{equation}
\begin{equation} \label{eq_intertwine}
\sO_Y^{\bZ/2} \circ \Phi_0 \simeq \Phi_0 \circ \sO_X, 
\end{equation}
where $\iota$ denotes the involution of $\Ku(Y)$ induced by the involution of $Y$ over $X$ and $\chi$ denotes the nontrivial character of $\bZ/2$. 

On the other hand, the group $\bZ/2$ defines a dual action on $\Ku(X)$ given by
\begin{equation} \label{eq_tau}
\tau \simeq \sO_X^d[d-1]    
\end{equation}
by \cite[Proposition 7.10]{KuzPer_cycliccovers}. We denote by $\Ku(X)^{\bZ/2}$ the associated invariant category. As recalled in Theorem~\ref{thm_elagin}, there is an equivalence 
$$\rho \colon (\Ku(Y)^{\bZ/2})^{\bZ/2} \to \Ku(Y).$$
Then $\Phi_0$ induces an equivalence $(\Phi_0)^{\bZ/2} \colon \Ku(X)^{\bZ/2} \to (\Ku(Y)^{\bZ/2})^{\bZ/2}$ and
\begin{equation*} \label{eq_defpsi}
(\Phi_0^{-1})^{\bZ/2} \circ \rho^{-1}  \colon \Ku(Y) \cong \Ku(X)^{\bZ/2}.    
\end{equation*}


\section{Stability conditions}
\label{sec_stabcond}

In this section we review the definition of stability conditions, their associated moduli spaces,  their behavior in the equivariant setting, and their construction from tilt stability in the case of quartic double solids and GM varieties. 

\subsection{Stability conditions} 
\label{section-stability}
The notion of a stability condition, due to Bridgeland \cite{bridgeland}, gives a well-behaved notion of a stable object in a triangulated category. 
Below we recall the definition, as well as a ``weak'' variant which comes up in the known construction methods. 

\begin{definition}
\label{definition-stability}
Let $\cC$ be a triangulated category and let 
$\bv  \colon \rK(\cC) \to \Lambda$ be a homomorphism from the Grothendieck group of $\cC$ to a finite rank free abelian group. 
A \emph{(weak) stability condition} on $\cC$ with respect to $\bv$ is a pair $\sigma = (\cA, Z)$, where $\cA$ is the heart of a bounded t-structure on $\cC$ and $Z \colon \Lambda \to \bC$ is a group homomorphism called the \emph{central charge},  satisfying the following compatibility conditions: 
\begin{enumerate}
    \item The composition $\rK(\cA)= \rK(\cC) \xrightarrow{\bv } \Lambda \xrightarrow{Z} \bC$ is a \emph{(weak) stability function} on $\cA$: for any $0 \neq E \in \cA$ we have $\Im Z(E) \geq 0$, and in the case that $\Im Z(E) = 0$, we have $\Re Z(E) (\leq) < 0$, where for simplicity we write $Z(-)$ instead of $Z(\bv (-))$. 
    
    \noindent The \emph{slope} of $E \in \cA$ is 
     $$\mu_{Z}(E)= 
     \begin{cases}
     -\frac{\Re Z(E)}{\Im Z(E)} & \text{if } \Im Z(E) > 0, \\
     + \infty & \text{otherwise.}
     \end{cases}$$
     We say that an object $E \in \cC$ is $\sigma$-(semi)stable if $E[k] \in \cA$ for some $k \in \bZ$, and for every proper subobject $F \subset E[k]$ in $\cA$ we have $\mu_{Z}(F) (\leq) <  \mu_{Z}(E[k]/F)$.
    \item Every object $E \in \cA$ admits a  \emph{Harder--Narasimhan (HN) filtration}: there is a filtration 
    $$0=E_0 \hookrightarrow E_1 \hookrightarrow \dots E_{m-1} \hookrightarrow E_m=E$$
    in $\cA$, whose factors $A_i:=E_i/E_{i-1} \neq 0$ are $\sigma$-semistable and $\mu_Z(A_1) > \dots > \mu_Z(A_m)$. 
    \item The \emph{support property} holds: there exists a quadratic form $Q$ on $\Lambda \otimes \bR$ such that the restriction of $Q$ to $\ker Z_{\bR} \subset \Lambda \otimes \bR$ is negative definite and $Q(E) \geq 0$ for all $\sigma$-semistable objects $E$ in $\cA$.
    \end{enumerate}
\end{definition}

\begin{remark}
\label{remark-numerical-stability}
If $\cC \subset \Dperf(X)$ is a semiorthogonal component of the derived category of a smooth proper $k$-variety $X$, then there is a natural choice for the lattice $\Lambda$. 
Namely, consider the Euler pairing $\rK(\cC) \times \rK(\cC) \to \bZ$, given by 
$\chi(E,F) = \sum_i (-1)^i \dim_k \Ext^i(E,F)$. The left and right kernel of this pairing agree (by Serre duality on $\cC$), and the 
\emph{numerical Grothendieck group} $\rK_\num(\cC)$ is by definition the quotient of $\rK(\cC)$ by the kernel. 
It is a finite rank free abelian group \cite[Lemma 12.7]{BLMNPS}, equipped with a canonical homomorphism $\bv_{\mathrm{num}} \colon \rK(\cC) \to \rK_{\num}(\cC)$. 
A \emph{numerical} stability condition is one for which $\bv$ factors via $\bv_{\mathrm{num}}$, and a \emph{full numerical} stability condition is one for which $\bv = \bv_{\mathrm{num}}$. 

More generally, if $\cC$ is a proper $k$-linear category whose numerical Grothendieck group $\Knum(\cC)$ has finite rank, then we can similarly define the notion of a (full) numerical stability condition. 
\end{remark} 

Instead of slopes, it is often convenient to order objects by their phase. 
Given a (weak) stability condition $\sigma$ on $\cC$, the \emph{phase} of an object $E \in \cA$ is 
$$\phi(E)= 
\begin{cases}
\frac{\text{arg}(Z(E))}{\pi} & \text{if } \Im Z(E)>0, \\
\hfil 1 & \text{otherwise}. 
\end{cases}
$$
We can define a collection of full additive subcategories $\cP(\phi) \subset \cC$ for $\phi \in \bR$ as follows:
\begin{enumerate}
\item[(1)] for $\phi \in (0,1]$, the subcategory $\cP(\phi)$ is the union of the zero object and all $\sigma$-semistable objects with phase $\phi$;
\item[(2)] for $\phi+n$ with $\phi \in (0,1]$ and $n \in \bZ$, set $\cP(\phi+n):=\cP(\phi)[n]$.
\end{enumerate}
We write $\cP(I)$, where $I \subset \bR$ is an interval, to denote the extension-closed subcategory of $\cC$ generated by the subcategories $\cP(\phi)$ with $\phi \in I$. Note that $\cP((0, 1])= \cA$.

\begin{remark} 
The collection of subcategories $\cP = (\cP(\phi))_{\phi \in \bR}$ form what is known as a \emph{slicing}, and the definition of a stability condition can be rephrased as the data of a pair $(\cP, Z)$ of a slicing and a suitably compatible central charge \cite[Proposition 5.3]{bridgeland}.  
\end{remark} 

By Bridgeland's deformation theorem \cite{bridgeland}, the set of all stability conditions $\Stab_{\Lambda}(\cC)$ (with respect to $\bv \colon \rK(\cC) \to \Lambda$) has the structure of a complex manifold. 
We will be particularly interested in full numerical stability conditions in the setting of Remark~\ref{remark-numerical-stability}, 
whose stability manifold we denote simply by $\Stab(\cC)$. 
There are two natural group actions on $\Stab(\cC)$, corresponding to reparameterizations of the category $\cC$ or of the central charge: 

\begin{enumerate}
    \item Let $\Aut(\cC)$ be the group of triangulated autoequivalences of $\cC$, modulo isomorphisms of functors. 
    Then $\Aut(\cC)$ acts on $\Stab(\cC)$ on the left, via 
    \begin{equation}
    \label{Phi-action-sigma} 
        \Phi \cdot (\cA, Z) \coloneqq (\Phi(\cA), Z \circ \Phi_*^{-1})  
    \end{equation}
    where $\Phi_*$ denotes the induced automorphism of $\Knum(\cC)$. 
    
    \item Let $\widetilde{\GL}^+_2(\bR) \to \GL_2^+(\bR)$ be the universal cover of the group of $2 \times 2$ real matrices with positive determinant. 
    Explicitly, 
\begin{equation}
\label{GLplus}
 \widetilde{\GL}^+_2(\bR)  = \left \{ (M, f) ~  \middle\vert ~ 
 \begin{array}{l} 
 M \in \GL^+_2(\bR), f \colon \bR \to \bR \text{ is an increasing function} \\ 
 \text{such that for all $\phi \in \bR$ we have $f(\phi+1) = f(\phi) + 1$} \\
 \text{and $M \cdot e^{i\pi \phi} \in \bR_{>0} \cdot e^{i \pi f(\phi)}$} 
 \end{array}  \right \}.
\end{equation} 
The group $\widetilde{\GL}^+_2(\bR)$ acts on the right on $\Stab(\cC)$, via 
\begin{equation*} 
(\cA, Z) \cdot (M, f) \coloneqq (\cP((f(0), f(1)]), M^{-1} \circ Z). 
\end{equation*} 
\end{enumerate}

\subsection{Moduli spaces and stability in families} \label{subsec_moduliandfamilies}
One of the main upshots of stability conditions is that they allow for the formulation of (potentially) well-behaved moduli problems. 
Although the notion of a stability condition is defined for any triangulated category, to make sense of moduli spaces one needs to assume the existence of an enhancement\footnote{By definition, a stability condition on an enhanced category is just a stability condition on the underlying triangulated category.}, either via a realization as a semiorthogonal component of a variety or as the homotopy category of a linear category as in Example~\ref{example-linear-categories}. 
For simplicity, we take the first approach, which will suffice for our purposes. 
We will also work in a relative setting, needed for some of the applications later in the paper. 
Our discussion will proceed in steps, from moduli of complexes on a scheme, to moduli of objects in a semiorthogonal component, and finally to moduli of (semi)stable objects in such a component, with each step resulting in an open substack of the previous. 

Let $f \colon X \to S$ be a flat and finitely presented morphism of schemes. We need to recall some properties of objects in $\Dqc(X)$, the unbounded derived category of quasi-coherent sheaves. 
An object $E \in \Dqc(X)$ is 
\emph{relatively perfect} over $S$ if it is pseudo-coherent (in the noetherian setting, equivalently $E$ has bounded above coherent cohomology) and locally of finite Tor-dimension over $f^{-1}\cO_S$; 
in particular, perfect complexes are relatively perfect, and the converse holds when $f \colon X \to S$ is smooth. 
The base change of a relatively perfect object is relatively perfect \cite[\href{https://stacks.math.columbia.edu/tag/0DI5}{Tag 0DI5}]{stacks-project}, and when $S = \Spec(k)$ is the spectrum of a field, $E$ is relatively perfect if and only if it lies in $\Db(X)$; 
therefore, a relatively perfect object can be thought of as a family of bounded coherent complexes on the fibers of $f$. 
An object $E \in \Dqc(X)$ is \emph{universally gluable} over $S$ if $\Ext_{\kappa(s)}^{<0}(E_s, E_s) = 0$ for all $s \in S$. 
Lieblich proved the following; see also \cite[\href{https://stacks.math.columbia.edu/tag/0DLB}{Tag 0DLB}]{stacks-project} for a convenient exposition. 

\begin{theorem}[{\cite{lieblich-moduli}}]
Let $f \colon X \to S$ be a proper, flat, finitely presented morphism of schemes. Then the moduli functor 
\begin{equation*}
    \cM_{\pug}(X/S) \colon (\Sch/S)^{\op} \to \Gpd   
\end{equation*}
taking $T \in (\Sch/S)$ to the groupoid of all $T$-perfect, universally gluable objects in $\Dqc(X_T)$ is an algebraic stack locally of finite presentation over $S$. 
\end{theorem}

Now we consider the sub-moduli problem parameterizing objects in a fixed semiorthogonal component. 
For simplicity, for the rest of this subsection we work with a smooth proper morphism of varieties, but at the expense of extra technical details everything works in much greater generality (see \cite{BLMNPS}); note that due to the smoothness assumption, the notions of relatively perfect and perfect objects coincide. 

\begin{proposition}[{\cite[Proposition 9.2]{BLMNPS}}] \label{prop_BLMNPS}
\label{proposition-MC-open-MX}
Let $\cC \subset \Dperf(X)$ be an $S$-linear semiorthogonal component, where $X \to S$ is a smooth proper morphism of varieties. 
Then there is an open substack $\cM_{\pug}(\cC/S) \subset \cM_{\pug}(X/S)$ 
parameterizing objects whose fibers are contained in the fibers of $\cC$. 
\end{proposition} 

For $\cC$ as above (or again, in an even more general setting), in \cite{BLMNPS} the notion of a stability condition on $\cC$ over $S$ is developed. 
This consists of a collection $\underline{\sigma} = (\sigma_s)_{s \in S}$ of numerical stability conditions on the fibers $\cC_s$ whose central charges all factor via a fixed finite rank free abelian group $\Lambda$, satisfying certain compatibility and tameness conditions. 
For details see \cite{BLMNPS}, where among other things it is shown that a stability condition on $\cC$ over $S$ exists in all of the cases where stability conditions are known to exist on fibers.

\begin{theorem}[{\cite[Theorem 21.24]{BLMNPS}}] \label{thm_BLMNPS}
Let $\cC \subset \Dperf(X)$ be an $S$-linear semiorthogonal component, where $X \to S$ is a smooth proper morphism of varieties. 
Let $\underline{\sigma}$ be a stability condition on $\cC$ over $S$ with respect to a lattice $\Lambda$. 
Let $v \in \Lambda$. 
Then there are algebraic stacks $\cM_{\underline{\sigma}}^{\mathrm{st}}(\cC/S, v)$ and $\cM_{\underline{\sigma}}(\cC/S, v)$ parameterizing objects in $\cC$ that are fiberwise stable (in the first case) or semistable (in the second) with respect to $\underline{\sigma}$, such that: 
\begin{enumerate}
    \item \label{BLMNPS-open} 
    There are open immersions 
    $\cM_{\underline{\sigma}}^{\mathrm{st}}(\cC/S, v) \subset \cM_{\underline{\sigma}}(\cC/S, v) \subset \cM_{\pug}(\cC/S)$. 
    \item $\cM_{\underline{\sigma}}^{\mathrm{st}}(\cC/S, v)$ and $\cM_{\underline{\sigma}}(\cC/S, v)$ are finite type over $S$ and $\cM_{\underline{\sigma}}(\cC/S, v) \to S$ is universally closed. 
    \item \label{BLMNPS-gerbe} If $\cM_{\underline{\sigma}}^{\mathrm{st}}(\cC/S, v) = \cM_{\underline{\sigma}}(\cC/S, v)$, then this stack is a $\bG_m$-gerbe over an algebraic space $M_{\underline{\sigma}}(\cC/S, v)$ proper over $S$. 
    \item If $S$ is characteristic $0$, then $\cM_{\underline{\sigma}}(\cC/S, v)$ admits a good moduli space $M_{\underline{\sigma}}(\cC/S, v)$ proper over $S$. 
\end{enumerate} 
\end{theorem}

\begin{remark}
\label{remark-proper-stab}
We emphasize that even in the case where $S = \Spec(k)$ is the spectrum of a field, a stability condition on $\cC$ over $S$ is a stability condition $\sigma$ in the sense of Definition~\ref{definition-stability} satisfying additional properties --- essentially, the existence of proper moduli spaces of semistable objects --- which a priori may not be satisfied by an arbitrary $\sigma \in \Stab(\cC)$. 
To differentiate the terminology more clearly, in the special case where $S = \Spec(k)$ we say $\sigma$ is a \emph{proper stability condition} to mean that it is a stability condition on $\cC$ over $S$.  
\end{remark}

\subsection{Equivariant stability} \label{subsec_dcdcandstability}
Let $\cC$ be a $k$-linear category equipped with an action by a finite group $G$. 
Let $\Lambda$ be a finite rank free abelian group equipped with a $G$-action, and let $\bv \colon \rK(\cC) \to \Lambda$ be a $G$-equivariant homomorphism; for instance, in the situation of Remark~\ref{remark-numerical-stability} we can take $\Lambda = \Knum(\cC)$ to be the numerical Grothendieck group. 
There is an induced $G$-action on $\Stab_{\Lambda}(\cC)$ where $g \in G$ acts on $\sigma = (\cA, Z) \in \Stab_{\Lambda}(\cC)$ by 
\begin{equation*}
    g \cdot \sigma = (\phi_g(\cA), Z \circ g_*^{-1}) 
\end{equation*}
where $\phi_g$ is the  autoequivalence corresponding to $g$ and $g_*$ is the action of $g$ of $\Lambda$. 
We say that $\sigma \in \Stab_{\Lambda}(\cC)$ is \emph{$G$-fixed} if for every $g \in G$ we have 
$g \cdot \sigma = \sigma$. 

\begin{theorem}[{\cite{polishchuk, MMS}}]
\label{theorem-equivariant-stability} 
In the above setup, assume that the order of $G$ is invertible in the field $k$. 
Let $\sigma = (\cA, Z) \in \Stab_{\Lambda}(\cC)$ be a $G$-fixed stability condition. 
Define 
\begin{align*}
    \bv^G & = \bv \circ \Forg_* \colon \rK(\cC^G) \to \Lambda, \\ 
    \cA^G  &= \set{ E \in \cC^G \st \Forg(E) \in \cA } , \\ 
    Z^G & = Z \circ \bv^G, 
\end{align*}
where $\Forg_*$ is the map on Grothendieck groups induced by the forgetful functor. 
Then the pair $\sigma^G = (\cA^G, Z^G)$ is a stability condition on $\cC^G$ with respect to $\bv^G$. Moreover, for an object $E \in \cC^G$ we have: 
\begin{enumerate}
    \item \label{G-ss} $E$ is $\sigma^G$-semistable if and only if $\Forg(E)$ is $\sigma$-semistable. 
    \item \label{G-s} If $E$ is $\sigma^G$-stable, then $\Forg(E)$ is $\sigma$-polystable. 
    \item \label{Forg-G-s} If $\Forg(E)$ is $\sigma$-stable then $E$ is $\sigma^G$-stable.
\end{enumerate}
\end{theorem}

\begin{remark}
As the notation suggests, $\cA^G$ as defined above is indeed equivalent to the category of $G$-invariants for the induced action of $G$ on $\cA$. 
\end{remark}

\begin{proof}
That $\sigma^G$ is a stability condition is essentially \cite[Proposition 2.2.3]{polishchuk} and \cite[Theorem 1.1]{MMS}, with two caveats. 
First, the references consider the special situation where $\cC = \Db(X)$ and $G$ acts through a geometric action on $X$; however, the argument from \cite{MMS} goes through in our context, using the identification of the homotopy category of $\cC^G$ from Proposition~\ref{proposition-enhancement-CG}. 
(Note that a stability condition on $\cC^G$ is by definition one on its triangulated homotopy category.) Second, the notion of a stability condition in the references is slightly different than ours (which is now the standard one) and does not require the support property with respect to $\Lambda$; however, the arguments still go through in our context, cf. \cite[Theorem 10.1]{BMS:stabCY3s}. 
Finally, the claims~\eqref{G-ss}-\eqref{Forg-G-s} follow formally from the construction.  Explicitly, \eqref{G-ss} is observed in \cite[Theorem 1.1]{MMS}, and \eqref{G-s} and \eqref{Forg-G-s} in \cite[Lemma 4.3]{BO-equivariant}. 
\end{proof}

\begin{remark}
In \cite{MMS} it is also proved that the above construction gives a closed embedding 
$\Stab_{\Lambda}(\cC)^G \to \Stab_{\Lambda}(\cC^G)$ from the set of $G$-fixed stability conditions, with explicitly describable image which is a union of connected components. 
However, we will not need these finer statements. 
\end{remark}

Note that in the construction from Theorem~\ref{theorem-equivariant-stability}, 
the map $\bv^G \colon \rK(\cC^G) \to \Lambda$ is $G^{\vee}$-equivariant for the trivial $G^{\vee}$-action on $\Lambda$ and the 
stability condition $\sigma^G$ is $G^{\vee}$-fixed, so we may apply the theorem again to obtain a stability condition $(\sigma^G)^{G^{\vee}}$ on $(\cC^G)^{G^{\vee}}$ with respect to $(\bv^G)^{G^{\vee}} \colon \rK((\cC^G)^{G^{\vee}}) \to \Lambda$. 
We observe that when $G$ is abelian, this iterated construction is compatible with the duality between $k$-linear categories with $G$-action and $G^{\vee}$-action from Theorem~\ref{thm_elagin}, up to a factor of $|G|$. 

\begin{lemma} \label{lemma_backtosamestability}
In the setup of Theorem~\ref{theorem-equivariant-stability}, further assume that $G$ is abelian and $k$ is algebraically closed. 
Then under the equivalence $\rho \colon (\cC^G)^{G^{\vee}} \xrightarrow{\sim} \cC$ of Theorem~\ref{thm_elagin}, the stability condition $(\sigma^{G})^{G^{\vee}}$ corresponds to the stability condition 
$|G| \cdot \sigma = (\cA, |G| \cdot Z)$. 
\end{lemma} 

\begin{proof}
By definition we have 
\begin{align*}
    (\bv^G)^{G^{\vee}} & = \bv \circ (\Forg_G \circ \Forg_{G^{\vee}})_* \colon \rK((\cC^G)^{G^{\vee}}) \to \Lambda, \\ 
    (\cA^G)^{G^{\vee}}  &= \set{ E \in (\cC^G)^{G^{\vee}} \st \Forg_G \circ \Forg_{G^{\vee}}(E) \in \cA } , \\ 
    (Z^G)^{G^{\vee}} & = Z \circ (\bv^G)^{G^{\vee}},  
\end{align*}
where $\Forg_{G^{\vee}} \colon (\cC^G)^{G^{\vee}} \to \cC^G$ and 
$\Forg_{G} \colon \cC^G \to \cC$ are the forgetful functors. 
By Lemma~\ref{lemma_generalnonsense}, under the identification $\rho \colon (\cC^G)^{G^{\vee}} \xrightarrow{\sim} \cC$, the functor $\Forg_{G^{\vee}}$ corresponds to $\Inf_G$. 
Note that 
$\Forg_G \circ \Inf_G \simeq \bigoplus_{g \in G} \phi_g$, where as usual $\phi_g$ is the autoequivalence of $\cC$ corresponding to $g$.
Therefore, if $\sigma' = (\cA', Z')$ with respect to $\bv' \colon \rK(\cC) \to \Lambda$ denotes the stability condition on $\cC$ corresponding to $(\sigma^{G})^{G^{\vee}}$, then 
\begin{align*}
    \bv' & = \sum_{g \in G} g_* \circ \bv \colon \rK(\cC) \to \Lambda, \\ 
    \cA'  &= \set{ E \in \cC \st \bigoplus_{g \in G} \phi_g(E) \in \cA } , \\ 
    Z' & = |G| \cdot Z,   
\end{align*}
where the first line holds by $G$-equivariance of $\bv \colon \rK(\cC) \to \Lambda$ and the third holds by $G$-fixedness of $Z$. Moreover, by $G$-fixedness of $\cA$ we find that $\cA' \subset \cA$, and therefore $\cA' = \cA$ as both are hearts of bounded $t$-structures. 
The above shows that $\sigma'$ can also be regarded as a stability condition with respect to the original homomorphism $\bv \colon \rK(\cC) \to \Lambda$. 
\end{proof}

\subsection{Construction of stability conditions} 
 
Consider $\cC=\Db(X)$ for a smooth projective variety $X$ of dimension $n$ defined over an algebraically closed field of characteristic $0$. Fix an ample class $H$ on $X$ and denote by $\ch_i(F)$ the $i$-th Chern character of an object $F \in \Db(X)$. The slope stability 
$$\sigma_{\text{slope}}=(\Coh(X), Z_{\text{slope}}=- H^{n-1}\ch_1+ \sqrt{-1}H^n \ch_0)$$
is a weak stability condition on $\Db(X)$ with respect to 
$$\bv \colon \rK(\Db(X)) \to \bZ^2,\quad \bv(F)=(H^n\ch_0(F), H^{n-1}\ch_1(F)),$$
and $\Lambda:=\text{Im}(\bv)$ (see \cite[Example 2.8]{BLMS}). If $X$ is a curve, then $\sigma_{\text{slope}}$ is further a proper stability condition.

We now recall the tilting procedure in this special case, which allows to construct stability conditions on surfaces and weak stability conditions on higher dimensional varieties. Assume $n \geq 2$. Fix $\beta \in \bR$ and consider the following subcategories of $\Coh(X)$:
$$\cT^\beta:=\lbrace F \in \Coh(X) \colon \text{ all HN factors of $F$ have slope } \mu_{Z_{\text{slope}}}(F) > \beta \rbrace,$$
$$\cF^\beta:=\lbrace F \in \Coh(X) \colon \text{ all HN factors of $F$ have slope } \mu_{Z_{\text{slope}}}(F) \leq \beta \rbrace.$$
By \cite{HapReiSm} the extension-closure $\Coh^\beta(X):= \langle \cT^\beta, \cF^\beta[1] \rangle$ is the heart of a bounded t-structure. Define
$$\bv \colon \rK(\Db(X)) \to \bQ^3,\quad \bv(F)=(H^n\ch_0(F), H^{n-1}\ch_1(F), H^{n-2}\ch_2(F))$$    
and $\Lambda:=\text{Im}(\bv)$. For $\alpha >0$ in $\bR$, define $Z_{\alpha, \beta} \colon \Lambda \to \bC$ by
$$Z_{\alpha, \beta}=- (H^{n-2}\ch_2- \beta H^{n-1}\ch_1+\frac{\beta^2-\alpha^2}{2}H^n\ch_0)+ \sqrt{-1}(H^{n-1}\ch_1- \beta H^n\ch_0).$$
Then the pair 
$$\sigma_{\alpha, \beta}=(\Coh^\beta(X), Z_{\alpha, \beta})$$
is a weak stability condition on $\Db(X)$ with respect to $\Lambda$, and it is a proper stability condition if $n=2$ \cite{bridgeland-K3,BMT:3folds-BG,BMS:stabCY3s,BLMNPS}. We call $\sigma_{\alpha,\beta}$ \emph{tilt stability} and we denote its slope by $\mu_{\alpha,\beta}$.

For the construction of stability conditions on Kuznetsov components it is useful to consider the following rotation of $\sigma_{\alpha, \beta}$. Fix $\mu\in \bR$ and let $u$ be the unit vector in the upper half plane with $\mu=-\frac{\Re u}{\Im u}$. We set
$$\cT^\mu_{\alpha, \beta}:=\lbrace F \in \Coh^\beta(X) \colon \text{ all HN factors of $F$ have slope } \mu_{\alpha,\beta}(F) > \mu \rbrace,$$
$$\cF^\mu_{\alpha, \beta}:=\lbrace F \in \Coh^\beta(X) \colon \text{ all HN factors of $F$ have slope } \mu_{\alpha, \beta}(F) \leq \mu \rbrace.$$
Then the extension-closure $\Coh_{\alpha, \beta}^{\mu}(X):=\langle \cT^\mu_{\alpha, \beta}, \cF^\mu_{\alpha, \beta}[1] \rangle$ is the heart of a bounded t-structure on $\Db(X)$ and the pair
$$\sigma_{\alpha, \beta}^{\mu}:=(\Coh_{\alpha, \beta}^{\mu}(X), Z_{\alpha, \beta}^{\mu}:=\frac{1}{u} Z_{\alpha, \beta})$$
is a weak stability condition on $\Db(X)$ with respect to $\bv$ by \cite[Proposition 2.15]{BLMS}.

Now assume that a $k$-linear triangulated category $\cC$ admits an exceptional collection with right orthogonal category $\cD$. Then \cite[Proposition 5.1]{BLMS} provides a criterion to ensure that the restriction of a weak stability condition on $\cC$ induces a stability condition on $\cD$. This criterion has been applied in \cite{BLMS} to construct stability conditions on the Kuznetsov components of Fano threefolds of Picard rank $1$ and index $2$ or $1$ and cubic fourfolds, and in \cite{PPZ} for GM fourfolds. 

More precisely, let $Y$ be a quartic double solid (see Example~\ref{example-quartic-ds}). By \cite[Theorem 6.8]{BLMS} there exist values of $\alpha$ and $\beta$ such that the pair
\begin{equation} \label{eq_stabonKu}
\sigma(\alpha, \beta)=(\cA(\alpha, \beta):=\Coh^0_{\alpha, \beta}(Y) \cap \Ku(Y), Z(\alpha, \beta):=Z^0_{\alpha, \beta}|_{\Knum(\Ku(Y))})  
\end{equation}
is a numerical stability condition on $\Ku(Y)$. Further, by \cite[Proposition 25.3]{BLMNPS} we have that $\sigma(\alpha, \beta)$ is a proper stability condition. If $X$ is a GM threefold, the same construction defines numerical proper stability conditions on $\Ku(X)$ by \cite[Theorem 6.9]{BLMS}. 

If $X$ is a GM fourfold, the construction of proper stability conditions on $\Ku(X)$ is more involved. Similarly to the case of cubic fourfolds studied in \cite[Section 9]{BLMS}, it has been proved that $\Ku(X)$ embeds in the bounded derived category of a twisted quadric threefold $(Y, \Cl_0)$; then the restriction of a rotation of tilt stability induces a stability condition on the Kuznetsov component by the criterion in \cite{BLMS}. More details about this construction will be explained in Section~\ref{subsec_invstabGM}.

\begin{remark} \label{rmk_stabcondinfamily}
The above mentioned construction of stability conditions works more generally in families. More precisely, assume that $\cX \to S$ is a family of quartic double solids or GM threefolds over a complex variety $S$. Then combining \cite[Proposition 25.3 and Theorem 23.1]{BLMNPS} and Remark~\ref{rmk_families_qdsGM}, we have that there exists a stability condition $\underline{\sigma}=(\sigma_s)_{s \in S}$ on $\Ku(\cX)$ over $S$ such that $\sigma_s$ is the numerical stability condition on $\Ku(\cX_s)$ defined in~\eqref{eq_stabonKu}.
\end{remark}

\begin{remark}
By the duality conjecture proved in \cite[Theorem 1.6]{KuzPerry:cones}, the existence of stability conditions on the Kuznetsov components of GM threefolds and fourfolds implies the same for the Kuznetsov components of GM fivefolds and sixfolds.    
\end{remark}


\section{Structure of moduli spaces} 
\label{sec_moduli}
In this section we prove some general results about the structure of moduli spaces of objects, with a focus on the case of Enriques categories. 
First, we give a smoothness criterion, which is often easy to verify for Enriques categories. 
Second, we study the relation between moduli spaces in a category and its invariant category for a $\bZ/2$-action. 
Third, we prove the nonemptiness criterion for moduli spaces, Theorem~\ref{theorem-deformation}, from the introduction. 

\subsection{A smoothness criterion} 

The following smoothness criterion is certainly well-known in the classical situation of moduli of sheaves, but we prove it here for lack of a suitable reference in the generality we need. 

\begin{lemma}
\label{lemma-smoothness-Ext2}
Let $X \to S$ be a proper, flat, finitely presented morphism of schemes. 
Let $s \colon \Spec(k) \to S$ be a field-valued point of $S$, and let $E$ be a relatively perfect, universally gluable object on $X_k$ such that $\Ext^{\geq 2}_k(E,E) = 0$. 
Then the morphism $\cM_{\pug}(X/S) \to S$ is smooth at the $k$-point given by $E$. 
\end{lemma}

\begin{proof} 
Since $\cM_{\pug}(X/S) \to S$ is locally of finite presentation, it suffices to show that it is formally smooth at $E$. 
More precisely, given
\begin{equation*}
0 \to I \to A' \to A \to 0 
\end{equation*} 
a square-zero extension of artinian local rings over $S$ with residue field $k$ 
such that the composition 
\begin{equation*}
\Spec(k) \to \Spec(A) \to \Spec(A') \to S
\end{equation*}
equals the given $k$-point $s$ of $S$,
and given $F \in \Dqc(X_A)$ an $A$-perfect, universally gluable object, 
we must show that there exists an $A'$-perfect, universally gluable object $F' \in \Dqc(X_{A'})$ such that $F'_A \cong F$. 
By \cite[Theorem 3.1.1]{lieblich-moduli}, the obstruction to the existence of $F'$ is an element in $\Ext^2_A(F, F \otimes I)$, so it suffices to show that this group vanishes. 

In fact, we claim that $\Ext^{\geq 2}_A(F, F \otimes I) = 0$. 
Note that by base change and the identification $F_k \cong E$, we have 
\begin{equation*}
    s^*\cHom_A(F, F \otimes I) \simeq \cHom_k(E, E \otimes s^*I). 
\end{equation*}
Therefore, we have a spectral sequence 
\begin{equation*}
    E_2^{i,j} = \rH^j(s^*\Ext^i_A(F, F \otimes I)) \implies 
    \Ext^{i+j}_{k}(E, E \otimes s^*I). 
\end{equation*}
As $s^*I \in \mathrm{D}^{\leq 0}(k)$, our assumption $\Ext^{\geq 2}_k(E,E) = 0$ implies that the right-hand side vanishes for $i + j \geq 2$. 
If $i$ is maximal such that $\Ext^i_A(F, F\otimes I)$ is nonzero, then also $\rH^0(s^*\Ext_A^i(F, F \otimes I)$ must be nonzero, and it survives the spectral sequence; therefore, $i \leq 1$, which proves our claim. 
\end{proof}

\subsection{Invariant stability} 
To formulate an interesting situation in which Lemma~\ref{lemma-smoothness-Ext2} applies, 
we use the notion of invariant stability. Recall from \S\ref{section-stability} the actions of $\Aut(\cC)$ and $\widetilde{\GL}^+_2(\bR)$ on the stability manifold $\Stab(\cC)$ of full numerical stability conditions, defined when $\cC$ is proper with finite rank $\Knum(\cC)$. 

\begin{definition}
\label{def_serreinv}
Let $\cC$ be a proper $k$-linear category
whose numerical Grothendieck group $\Knum(\cC)$ has finite rank. 
Let $\Phi$ be an autoequivalence of $\cC$. 
A full numerical stability condition $\sigma$ on $\cC$ is called \emph{$\Phi$-invariant} if 
there exists $\wtilde{g} \in \widetilde{\GL}^+_2(\bR)$ such that $\Phi \cdot \sigma = \sigma \cdot \wtilde{g}$, 
i.e. the action of $\Phi$ preserves the $\widetilde{\GL}^+_2(\bR)$-orbit of $\sigma$. 
If $\cC$ admits a Serre functor $\rS_{\cC}$, we say $\sigma$ is \emph{Serre invariant} if it is $\rS_{\cC}$-invariant. 
\end{definition}

\begin{remark}
In the situation of Definition~\ref{def_serreinv}, we can also consider the notion of a 
\emph{$\Phi$-fixed} stability condition $\sigma$ appearing in Theorem~\ref{theorem-equivariant-stability}, namely one for which $\Phi \cdot \sigma = \sigma$. 
It is slightly weaker to require that $\sigma$ is $\Phi$-invariant. 
However, if $\sigma$ is $\Phi$-invariant, then $\sigma$ and $\Phi \cdot \sigma$ have the same collections of (semi)stable objects, so from the perspective moduli spaces they are essentially equivalent. 
In fact, as we observe below, when $\Phi$ has finite order, the collections of (semi)stable objects of any given phase also coincide. 
\end{remark} 

\begin{lemma}
\label{lemma-Phi-invt} 
In the situation of Definition~\ref{def_serreinv}, assume that $\Phi$ has finite order, i.e. $\Phi^n \simeq \id_{\cC}$ for some integer $n$, and $\sigma$ is $\Phi$-invariant. 
Then $\Phi$ preserves the collection of $\sigma$-(semi)stable objects of given phase~$\phi$, i.e. $\Phi(\cP(\phi)) = \cP(\phi)$ for every $\phi \in \bR$ where $\cP$ is the slicing for $\sigma$. 
\end{lemma}

\begin{proof}
Write $\wtilde{g} = (M, f)$ as in \eqref{GLplus}. 
Let $\cP$ be the slicing for $\sigma$. 
Together with $\Phi^n(E) \cong E$, it follows that the increasing function $f \colon \bR \to \bR$ satisfies $f^n(\phi(E)) = \phi(E)$, where $f^n$ is the $n$-fold composition of $f$. 
This implies that $f(\phi(E)) = \phi(E)$. 
\end{proof} 

\begin{lemma}
\label{lemma-M-Enriques-smooth}
Let $\cC$ be an Enriques category over a field $k$. 
Let $\sigma$ be a Serre invariant stability condition on $\cC$. 
If $E$ is a $\sigma$-stable object of $\cC$ such that $E$ is not fixed by the $\bZ/2$-action on $\cC$, i.e. 
$\tau(E) \ncong E$ where $\tau$ is the generating involution of the action, then $\Ext_k^{\geq 2}(E,E) = 0$. 
\end{lemma} 

\begin{proof}
For any object $E \in \cC$, Serre duality gives 
\begin{equation*}
\Ext_k^i(E,E) \cong \Ext_k^{2-i}(E, \tau(E))^{\vee}. 
\end{equation*} 
If $E$ is $\sigma$-(semi)stable, then by Lemma~\ref{lemma-Phi-invt}, 
$\tau(E)$ is $\sigma$-(semi)stable of the same phase $\phi(E)$. 
In particular, it follows that the above group vanishes for $i > 2$. 
If $E$ is $\sigma$-stable and $\tau(E) \ncong E$, 
then there can be no nonzero morphism $E \to \tau(E)$, so the 
above group also vanishes for $i = 2$. 
\end{proof} 

In summary, we obtain the following smoothness criterion for the moduli space of objects in a family of Enriques categories. 

\begin{corollary}
\label{corollary-enriques-moduli-smooth}
Let $\cC$ be an Enriques category over $S$ such that 
$\cC \subset \Dperf(X)$ is an $S$-linear semiorthogonal component,  
where $X \to S$ is a smooth proper morphism of varieties. 
Let $s \colon \Spec(k) \to S$ be a field-valued point of $S$, let 
$\sigma_{s}$ be a Serre invariant stability condition on the fiber $\cC_s$, and 
let $E$ be a $\sigma_s$-stable object which is not fixed by the $\bZ/2$-action on $\cC_s$. 
Then the morphism $\cM_{\pug}(\cC/S) \to S$ is smooth at the $k$-point given by $E$. 
\end{corollary} 

\begin{proof}
By Proposition~\ref{proposition-MC-open-MX}, $\cM_{\pug}(\cC/S) \to S$ factors through $\cM_{\pug}(X/S) \to S$ via an open immersion, so it suffices to show the latter is smooth at $E$, which 
follows from Lemmas~\ref{lemma-smoothness-Ext2} and~\ref{lemma-M-Enriques-smooth}. 
\end{proof} 

\subsection{Maps between moduli spaces} 
Next we explain the relation between moduli spaces of stable objects for categories related by a group action 
as in Theorem~\ref{theorem-equivariant-stability}, restricting for simplicity to the case where $G = \bZ/2$. 

\begin{proposition}
\label{proposition-maps-bw-moduli}
Let $\cC \subset \Dperf(X)$ be an $k$-linear semiorthogonal component of a smooth proper variety $X$ over a field $k$ such that $\characteristic(k) \neq 2$. 
Assume that $\cC$ is equipped with a $\bZ/2$-action such that 
$\cC^{\bZ/2}$ also embeds as a $k$-linear semiorthogonal component of the derived category of a smooth proper variety. 
Let $\sigma$ be a proper stability condition on $\cC$ (in the sense of Remark~\ref{remark-proper-stab}) 
with respect to a $\bZ/2$-equivariant homomorphism $\rK(\cC) \to \Lambda$ such that $\sigma$ is $\bZ/2$-fixed.  
Let $\sigma^{\bZ/2}$ be the stability condition on $\cC^{\bZ/2}$ given by Theorem~\ref{theorem-equivariant-stability} 
and assume that $\sigma^{\bZ/2}$ is proper. 
Then for any $v \in \Lambda$, the forgetful functor $\Forg \colon \cC^{\bZ/2} \to \cC$ induces a morphism of moduli stacks  
\begin{equation*}
\Forg \colon \cM_{\sigma^{\bZ/2}}(\cC^{\bZ/2}, v) \to \cM_{\sigma}(\cC, v)
\end{equation*} 
such that: 
\begin{enumerate}
\item \label{def-F}
If $v \in \Lambda$ is fixed by the $\bZ/2$-action, then 
$\Forg$ factors via a surjective morphism
\begin{equation*}
f \colon \cM_{\sigma^{\bZ/2}}(\cC^{\bZ/2}, v) \to \cM_{\sigma}(\cC, v)^{\bZ/2}
\end{equation*} 
onto the fixed locus of the induced 
$\bZ/2$-action on $\cM_{\sigma}(\cC, v)$. 
\item \label{F-etale} 
The restriction of $f$ to the stable locus $\cM_{\sigma^{\bZ/2}}^{\mathrm{st}}(\cC^{\bZ/2}, v)$ is degree $2$ onto its image and \'{e}tale away from points corresponding to objects $F \in \cC^{\bZ/2}$ that are fixed by the generating involution of the residual $\bZ/2 = (\bZ/2)^{\vee}$-action from \S\ref{section-reconstruction}. 
\end{enumerate} 
\end{proposition}

\begin{proof}
The fixed stack $\cM_{\sigma}(\cC, v)^{\bZ/2}$ parameterizes $\bZ/2$-linearized (in the sense of Example~\ref{example-classical-linear}) objects in $\cC$ of class $v$; therefore,~\eqref{def-F} follows from Proposition~\ref{proposition-enhancement-CG} and Theorem~\ref{theorem-equivariant-stability}. 

About \eqref{F-etale}, the fact that the restriction of $f$ to the stable locus has degree $2$ follows from \cite[Lemma 1]{Ploog}. For the second part, we argue similarly to \cite[Proposition 6.2]{Nuer}. Set $E=(F, \theta)$ where $F=\Forg(E)$. Up to shifting, we may assume $E$ is in the heart $\cA^{\bZ/2}$ of $\sigma^{\bZ/2}$. By \cite[Theorem 3.1.1 and Proposition 3.5.1]{lieblich-moduli}, 
the tangent space of $\cM_{\sigma^{\bZ/2}}(\cC^{\bZ/2}, v)$ at the class of $E$ is identified with $\Ext^1(E, E)$. An analogous identification holds for the tangent space of $\cM_{\sigma}(\cC, v)$ at the class of $F$, and the differential of $f$ at the class of $E$ is given by the map
$$\mathsf{d}f \colon \Ext^1(E, E) \to \Ext^1(F, F)$$ which operates as follows. An element of $\Ext^1(E, E)$ is by definition an extension in $\cA^{\bZ/2}$ of the form
\begin{equation}
\label{SES-EE'} 
0 \to E \to E' \to E \to 0.
\end{equation} 
Indeed, $E'$ is semistable in $\cA^{\bZ/2}$ with the same slope of $E$. Writing $E'=(F', \theta')$, we have $\mathsf{d}f(E')=F'$, where $F'$ is an extension in $\Ext^1(F,F)$ and is $\sigma$-semistable in the heart $\cA$ of $\sigma$ with the same slope as $F$.

Now if $\mathsf{d}\Phi(E')=0$, then the sequence
$$0 \to F \to F' \to F \to 0$$
in $\cA$ defining $F'$ splits. Applying the inflation functor $\Inf \colon \cC \to \cC^{\bZ /2\bZ}$ we get a triangle
\begin{equation}
\label{eq_seqInfF}    
\Inf(F) \to \Inf(F') \to \Inf(F).
\end{equation}
By Lemma~\ref{lemma_infvschi} below, we have $\Inf(F) \cong (F, \theta) \oplus (F, \theta \otimes \chi)$. Moreover, as $\Forg(\Inf(F))\cong F \oplus F$ is $\sigma$-semistable in $\cA$, we have that $\Inf(F)$
is $\sigma^{\bZ/2}$-semistable in $\cA^{\bZ/2}$. The same observation holds for $\Inf(F')$ and thus \eqref{eq_seqInfF} is a short exact sequence in $\cA^{\bZ/2}$. 
The sequence \eqref{eq_seqInfF} also splits, because the one defining $F'$ does. Thus we get a split sequence 
\begin{equation}
\label{SES-EE'-sum}
0 \to E \oplus E \otimes \chi \to E' \oplus E' \otimes \chi \to E \oplus E \otimes \chi \to 0.
\end{equation}
Now assume that $E$ is not fixed by the residual action, i.e.\ $E \ncong E \otimes \chi$. Since $E$ is stable, we deduce that $\Hom(E, E \otimes \chi)=0=\Hom(E \otimes \chi, E)$, and hence by the exact sequence~\eqref{SES-EE'} we have $\Hom(E', E \otimes \chi)=0=\Hom(E' \otimes \chi, E)$. 
It follows that a splitting of~\eqref{SES-EE'-sum} must also induce a splitting of~\eqref{SES-EE'}. 
\end{proof}

\begin{lemma} \label{lemma_infvschi}
With the same assumptions as in Proposition~\ref{proposition-maps-bw-moduli}, for every $B \in \cC$, we have $\Inf(B) \cong (B, \theta) \oplus (B, \theta \otimes \chi)$ for every linearization $\theta$ on $B$.
\end{lemma}
\begin{proof}
By adjunction we have $\Hom_{\cC}(\Forg((A, \psi)), B) \cong \Hom_{\cC^{\bZ/2}}((A, \psi), \Inf(B))$ for every $(A, \psi) \in \cC^{\bZ/2}$. So the statement is equivalent to showing 
\begin{equation*}
    \Hom_{\cC}(\Forg((A, \psi)), B) \cong \Hom_{\cC^{\bZ/2}}((A, \psi), (B, \theta) \oplus (B, \theta \otimes \chi)),
\end{equation*}
which is the same as proving
$$\Hom_{\cC}(A, B) \cong \Hom_{\cC^{\bZ/2}}((A, \psi), (B, \theta)) \oplus \Hom_{\cC^{\bZ/2}}((A, \psi), (B, \theta \otimes \chi).$$
Denote by $\iota$ the generating involution of the $\bZ/2$-action on $\cC$. First, note that if 
\begin{equation*}
    f \in \Hom_{\cC^{\bZ/2}}((A, \psi), (B, \theta))  \cap \Hom_{\cC^{\bZ/2}}((A, \psi), (B, \theta \otimes \chi),\end{equation*}
then $f=0$. Indeed, by definition $f$ satisfies $\theta_\iota \circ f = \iota(f) \circ \psi_\iota$ and $-\theta_\iota \circ f = \iota(f) \circ \psi_\iota$, which is possible only if $f=0$. Thus we have an injective map
$$\Hom_{\cC^{\bZ/2}}((A, \psi), (B, \theta)) \oplus \Hom_{\cC^{\bZ/2}}((A, \psi), (B, \theta \otimes \chi)) \hookrightarrow \Hom_{\cC}(A, B), \, f_1 \oplus f_2 \mapsto f_1+f_2.$$
We claim that this map is an isomorphism. To define an inverse, for $f \in \Hom_{\cC}(A,B)$, we set
$$f_1=\frac{1}{2}(f + \theta_\iota^{-1} \circ \iota(f) \circ \psi_\iota),$$
$$f_2=\frac{1}{2}(f - \theta_\iota^{-1} \circ \iota(f) \circ \psi_\iota).$$
To check that $f_1 \in \Hom_{\cC^{\bZ/2}}((A, \psi), (B, \theta))$, $f_2 \in \Hom_{\cC^{\bZ/2}}((A, \psi), (B, \theta \otimes \chi))$, it is enough to observe that $\theta_\iota \circ f_1= \iota(f_1) \circ \psi_\iota$, $-\theta_\iota \circ f_2= \iota(f_2) \circ \psi_\iota$, which follows from $\iota(\psi_\iota)^{-1}=\psi_\iota$, $\iota(\theta_\iota)^{-1}=\theta_\iota$.
\end{proof}

\begin{remark}[Symmetrized version of Proposition~\ref{proposition-maps-bw-moduli}] 
\label{remark-symmetrized-maps-bw-moduli}
Via the duality of categories equipped with a $\bZ/2$-action and a $\bZ/2$-fixed stability condition, explained in \S\ref{section-reconstruction} and \S\ref{subsec_dcdcandstability}, we can reverse the roles of $\cC$ and $\cC^{\bZ/2}$ in Proposition~\ref{proposition-maps-bw-moduli}. 
To explain this, for notational clarity we 
let $\cD = \cC^{\bZ/2}$ equipped with the residual $\bZ/2 = (\bZ/2)^{\vee}$-action, 
and write $\sigma_{\cC} = \sigma$ and $\sigma_{\cD} = \sigma^{\bZ/2}$. 
Then by Lemma~\ref{lemma_backtosamestability}, the pairs $(\cC, \sigma_{\cC})$ and $(\cD, \sigma_{\cD})$ (up to scaling by $2$ the central charge of the stability conditions) correspond to each other under the above-mentioned duality. 
Therefore, in addition to the morphism $f \colon \cM_{\sigma_{\cD}}(\cD, v) \to \cM_{\sigma_{\cC}}(\cC, v)^{\bZ/2}$, we also get a morphism 
$g \colon \cM_{\sigma_{\cC}}(\cC, v) \to \cM_{\sigma_{\cD}}(\cD, v)^{\bZ/2}$ with similar properties. 

Note that above, both $\sigma_{\cD}$ and $\sigma_{\cC}$ are both regarded as stability conditions with respect to a fixed lattice $\Lambda$. It is perhaps more natural, however, to use the numerical Grothendieck groups to fix the class of objects, as follows. 
Let us suggestively denote by $\pi_* \colon \cD \to \cC$ and $\pi^* \colon \cC \to \cD$ the canonical functors, namely, the forgetful and inflation functors between $\cD = \cC^{\bZ/2}$ and $\cC$. 
Then for any $v \in \Knum(\cC)$ fixed by the $\bZ/2$-action, we obtain a surjective morphism  
\begin{equation*}
f \colon \bigsqcup_{\pi_*(w) = v} \cM_{\sigma_{\cD}}(\cD, w) \to \cM_{\sigma_{\cC}}(\cC, v)^{\bZ/2} ,
\end{equation*} 
where the unions ranges over $w \in \Knum(\cD)$ such that $\pi_*(w) = v$, 
with the property that $f$ is \'{e}tale of degree $2$ at $\sigma_{\cD}$-stable objects not fixed by the $\bZ/2$-action. 
Similarly, in the other direction, for any $w \in \Knum(\cD)$ fixed by the $\bZ/2$-action, 
we obtain a surjective morphism  
\begin{equation*}
g \colon \bigsqcup_{\pi^*(v) = w} \cM_{\sigma_{\cC}}(\cC, v) \to \cM_{\sigma_{\cD}}(\cD, w)^{\bZ/2} ,
\end{equation*} 
where the union ranges over $v \in \Knum(\cC)$ such that $\pi^*(v) = w$, 
with the property that $g$ is \'{e}tale of degree $2$ at $\sigma_{\cC}$-stable objects not fixed by the $\bZ/2$-action. 
\end{remark} 

\begin{proof}[Proof of Theorem~\ref{theorem-moduli}]
The result follows from Corollary~\ref{corollary-enriques-moduli-smooth}, Proposition~\ref{proposition-maps-bw-moduli}, Remark~\ref{remark-symmetrized-maps-bw-moduli}, and \cite[Theorem 1.4]{IHC-CY2}.   
\end{proof}

\subsection{Nonemptiness criterion} 
In this section, we prove Theorem~\ref{theorem-deformation} from the introduction. 
We will use the following lemma. 

\begin{lemma}
\label{lemma-todo-enriques} 
In the setup of Remark~\ref{remark-symmetrized-maps-bw-moduli}, assume that $\cC$ is an Enriques category equipped with the corresponding $\bZ/2$-action. 
Let $v \in \Knum(\cC)$ be class such that $\cM_{\sigma_{\cC}}(\cC, v)$ consists of stable objects. 
Let $w \in \Knum(\cD)$ be a class such that $\pi_*(w) = v$,  
$\cM_{\sigma_{\cD}}(\cD, w)$ is nonempty, and the inequality $-\chi(w,w)+2 < -\chi(v,v)+1$ holds. 
Then there exist objects in $\cM_{\sigma_{\cC}}(\cC, v)$ which are not fixed by the $\bZ/2$-action. 
\end{lemma}

\begin{proof}
By Theorem~\ref{theorem-equivariant-stability}\eqref{Forg-G-s},  $\cM_{\sigma_{\cD}}(\cD, w)$ also consists of stable objects. 
Therefore there are $\bG_m$-gerbes $\cM_{\sigma_{\cC}}(\cC, v) \to M_{\sigma_{\cC}}(\cC, v)$ and 
$\cM_{\sigma_{\cD}}(\cD, w) \to M_{\sigma_{\cD}}(\cD, w)$ over proper algebraic spaces. 
By \cite{IHC-CY2} the space $M_{\sigma_{\cD}}(\cD, w)$ is smooth of dimension $-\chi(w,w)+2$. 
By deformation theory, the dimension of $M_{\sigma_{\cC}}(\cC, v)$ at any point $E$ satisfies 
\begin{equation*}
    \dim_E M_{\sigma_{\cC}}(\cC, v) \geq \dim \Ext^1(E,E) - \dim \Ext^2(E,E) = - \chi(v,v) + 1, 
\end{equation*}
where the last equality uses that $E$ is stable, and in particular simple; therefore, the irreducible components of $M_{\sigma_{\cC}}(\cC, v)$ all have dimension at least $-\chi(v,v) + 1$.  
Another consequence of the simplicity of $E$ which we will use below is that there exists an object $F \in \cD$ such that $\pi_*(F) \cong E$, and the isomorphism classes of such $F$ form a torsor under the $\bZ/2$-action on $\cD$; this follows from \cite[Lemma 1]{Ploog}, since $\rH^2(\bZ/2, k^{\times}) = 0$ by our assumption that $\characteristic(k) \neq 2$. 

Let $M_0$ be an irreducible component of $M_{\sigma_{\cC}}(\cC, v)$ which meets the image of $M_{\sigma_{\cD}}(\cD,w)$. 
We claim that an object in $\cC$ corresponding to a generic point of $M_0$ is not fixed by the $\bZ/2$-action. 
Assume on the contrary that every point of $M_0$ is fixed by the $\bZ/2$-action. 
Then by Remark~\ref{remark-symmetrized-maps-bw-moduli}, there exists $w'\in \Knum(\cD)$ such that $\pi_*(w') = v$ and an irreducible component of $M_{\sigma_{\cD}}(\cD, w')$ which surjects onto $M_0$. 
Let $E$ be an object in $M_0$ contained in the image of map $M_{\sigma_{\cD}}(\cD,w) \to M_{\sigma_{\cC}}(\cC, v)$, say $\pi_*(F) \cong E$ for $F$ in $M_{\sigma_{\cD}}(\cD,w)$. 
By our choice of $w'$, there also exists $F'$ in $M_{\sigma_{\cD}}(\cD, w')$ such that $\pi_*(F') \cong E$. 
By the remarks in the previous paragraph, we conclude that $F$ and $F'$ are related by the $\bZ/2$-action on $\cD$, and in particular, it follows that $\chi(w,w) = \chi(w',w')$. 
But then  
\begin{equation*}
\dim M_{\sigma_{\cD}}(\cD, w') = - \chi(w',w') + 2 < -\chi(v,v)+1 \leq \dim M_0
\end{equation*} 
which contradicts that a component of $M_{\sigma_{\cD}}(\cD, w')$ surjects onto $M_0$. 
\end{proof}

\begin{proof}[Proof of Theorem~\ref{theorem-deformation}]
We consider the relative moduli space $\cM_{\underline{\sigma}}(\fC/S, \bv_0(v)) \to S$. 
Note that by~\eqref{defo-fiber-0} its fiber over $0 \in S(\bC)$ is given by $\cM_{\sigma_{0}}(\fC_0, \bv_0(v)) \cong \cM_{\sigma}(\cC, v)$, and its fiber over $1 \in S(\bC)$ is 
$\cM_{\sigma_{1}}(\fC_1, \bv_0(v))$. 
By~\eqref{1-exists} and Lemma~\ref{lemma-todo-enriques}, 
there exists a $\bC$-point of $\cM_{\sigma_{1}}(\fC_1, \bv_0(v))$ for which the corresponding object $E \in \fC_1$ is $\sigma_1$-stable and not fixed by the $\bZ/2$-action on $\fC_1$. 
Thus by Corollary~\ref{corollary-enriques-moduli-smooth} the morphism $\cM_{\pug}(\fC/S) \to S$ is smooth at $E$; as by Theorem~\ref{thm_BLMNPS}\eqref{BLMNPS-open} there is an open immersion $\cM_{\underline{\sigma}}(\fC/S, \bv_0(v)) \subset \cM_{\pug}(\fC/S)$, it follows that 
$\cM_{\underline{\sigma}}(\fC/S, \bv_0(v)) \to S$ is also smooth at $E$, and in particular dominant. 
But by Theorem~\ref{thm_BLMNPS} 
the morphism $\cM_{\underline{\sigma}}(\fC/S, \bv_0(v)) \to S$ is also universally closed, so it must be surjective. 
In particular, its fiber $\cM_{\sigma}(\cC, v)$ over $0$ is nonempty.  
\end{proof}


\section{Quartic double solids} 
\label{section-quartic-double-solids}

This section is devoted to the proof of Theorem~\ref{theorem-M-quartic-double}, which is obtained as an application of Theorem~\ref{theorem-deformation}.

\subsection{Preparation} \label{sec_qds_preparation}

Let $Y$ be a quartic double solid. As seen in Example~\ref{example-quartic-ds}, its Kuznetsov component $\Ku(Y)$ is an Enriques category. 

We denote by $i \colon X \hookrightarrow \bP^3$ the branch locus of the double covering $f \colon Y \to \bP^3$, which is a quartic K3 surface, and by $j \colon X \hookrightarrow Y$ its inclusion as the branch divisor. Let $H$ be the hyperlane class on $\bP^3$, whose pullbacks to $Y$ and $X$ we also denote by the same symbol. 

The category $\Db(X)$ is the CY2 cover of $\Ku(Y)$. Indeed, we are in the setting of \S\ref{subsec_generalsetting} with $d=2$, $m=4$, $\cB=\langle \cO \rangle$. In particular, we have the equivalence 
$$\Phi_0:=\rL_{\langle \cO_Y \otimes 1, \cO_Y(H) \otimes 1 \rangle} \circ (j_0)_* \circ (- \otimes \cO_X(2H)) \colon \Db(X) \simeq \Ku(Y)^{\bZ/2}.$$
The $\bZ/2$-action on $\Ku(Y)$ is induced by the involutive autoequivalence $\iota \colon \Ku(Y) \to \Ku(Y)$ given by pullback along the double cover involution. The Serre functor $S_{\Ku(Y)}$ of $\Ku(Y)$ satisfies $S_{\Ku(Y)}=\iota[2]$ by \cite[Corollary 4.6]{kuznetsov-CY}. 

By \cite[Proposition 3.9]{Kuz_Fano} the numerical Grothendieck group $\Knum(\Ku(Y))$ has rank $2$, with a basis $\mu_1, \mu_2$ such that 
\begin{equation} \label{eq_basisKnumqds}
\ch(\mu_1) =1 - \frac{1}{2}H^2, \quad \ch(\mu_2) =H-\frac{1}{2}H^2-\frac{2}{3}p  , 
\end{equation}
where $p$ denotes the class of a point. 
Moreover, the Euler form is given in this basis by 
$$\langle \mu_1,\mu_2 \rangle= 
\begin{pmatrix}
-1 & -1 \\ 
-1 & -2
\end{pmatrix}.$$
Since $\iota$ is induced by the double covering involution, and the latter preserves $H$, we have the following lemma. 
\begin{lemma} \label{lemma_iotaisidentity_qds}
The involution $\iota_*$ induced by $\iota$ on $\Knum(\Ku(Y))$ is trivial.    
\end{lemma}

In the next lemma, we show that the stability conditions constructed on $\Ku(Y)$ in \cite{BLMS} are fixed by the action of $\iota$.

\begin{lemma}[{\cite[Lemma 6.1]{PY}}] \label{lemma_BLMSstabiotainv_qds}
Let $\sigma(\alpha, \beta)$ be a stability condition defined in \eqref{eq_stabonKu} as constructed in \cite{BLMS}. Then $\iota \cdot \sigma(\alpha, \beta)=\sigma(\alpha, \beta)$.    
\end{lemma}
\begin{proof}
Note that $\iota$ preserves $\Coh(Y)$ and fixes the Chern character of every object in $\Db(Y)$. Thus $\iota$ preserves the tilted heart $\Coh^\beta(Y)$ and acts as the identity on the weak stability function $Z_{\alpha,\beta}$. Analogously $\iota$ preserves $\Coh^0_{\alpha, \beta}(Y)$ and fixes $Z^0_{\alpha, \beta}$. Since $\iota$ restricts to an involutive autoequivalence of $\Ku(Y)$, we conclude that $\sigma(\alpha, \beta)$ is $\iota$-fixed.
\end{proof}

To simplify the notation, we set 
\begin{equation} \label{eq_sigmaY}
\sigma_Y:=\sigma(\alpha, \beta).    
\end{equation}
By Theorem~\ref{theorem-equivariant-stability} it follows that $\sigma_Y$ induces a stability condition $\sigma_Y^{\bZ/2}$
on $\Ku(Y)^{\bZ/2}$ and thus a stability condition 
$$\sigma_X:=\Phi_0^{-1}(\sigma_Y^{\bZ/2})$$ on $\Db(X)$. We also introduce the notation 
\begin{equation}
    \Phi \coloneqq \Forg \circ \Phi_0 \colon \Db(X) \to \Ku(Y), 
\end{equation}
as this functor will occur in many of our computations below. 
Note that by Lemma~\ref{lemma1}, we have the formula 
\begin{equation}
\label{formula-forgphi}
    \Phi = \rL_{\langle \cO_Y, \cO_Y(H) \rangle} \circ j_* \circ (- \otimes \cO_X(2H)) .
\end{equation}

\begin{lemma} \label{lemma_geometricsigma_qds}
The stability condition $\sigma_X$ on $\Db(X)$ is geometric, i.e. the skyscraper sheaf $\cO_x$ at any point $x \in X$ is $\sigma_X$-stable.
\end{lemma}
\begin{proof}
By Theorem~\ref{theorem-equivariant-stability} it is enough to show that $\Phi(\cO_x) = \Forg(\Phi_0(\cO_x))$ is $\sigma_Y$-stable. 
By~\eqref{formula-forgphi}, we thus need to compute $\rL_{\langle \cO_Y, \cO_Y(H) \rangle}(\cO_x)$.
The first mutation is 
$$\rL_{\cO_Y(H)}(\cO_x)=\cI_{x/Y}(H)[1],$$ 
where $\cI_{x/Y}$ denotes the ideal sheaf of the point $x$ in $Y$. 
Applying $\rL_{\cO_Y}$ we find   \begin{equation*}
  \Forg(\Phi_0(\cO_x))=K_x[2],  
\end{equation*} 
where $K_x$ fits into an exact triangle 
\begin{equation*}
    K_x \to \cO_Y^{\oplus 3} \to \cI_{x/Y}(H). 
\end{equation*}
In particular, we have that 
\begin{equation*}
\ch(K_x)=2-H-\frac{1}{2}H^2+\frac{2}{3}p, 
\end{equation*} 
from which it follows that that 
\begin{equation}
\label{class-Kx}
    [K_x] = 2\mu_1-\mu_2 \in \Knum(\Ku(Y))
\end{equation}
We need to show the $\sigma_Y$-stability of $K_x$. 
Note that since $\Knum(\Ku(Y))$ has rank $2$ and $[K_x]$ is primitive, this is equivalent to the $\sigma_Y$-semistability of $K_x$. 

If $f \colon Y \to \bP^3$ denotes the double covering map, then $f^*\cI_{x/\bP^3}$ is a subsheaf of $\cI_{x/Y}$, since $x$ is in the ramification locus and the pullback via $f$ of a function vanishing at $x$ will vanish at $x$ to order $2$. Since $\cO_Y(H)$ is pullback from the base, by the above computation $f^*\cI_{x/\bP^3}(H)$ is a subsheaf of $\cI_{x/Y}(H)$. Computing the Chern character, we observe that the cokernel of this map is $\cO_x$ and we get the commutative diagram of exact triangles
\[
\xymatrix{
F \ar[r] \ar[d] & f^*\cO_{\bP^3}^{\oplus 3} \ar[r] \ar[d]& f^*\cI_{x/\bP^3}(H) \ar[d]\\ 
K_x \ar[r] \ar[d] & \cO_Y^{\oplus 3} \ar[r] \ar[d] & \cI_{x/Y}(H) \ar[d] \\
\cO_x[-1] \ar[r] & 0 \ar[r] & \cO_x
}
\]
where the first row and the third column are short exact sequences of sheaves. In particular, this implies that the nonzero cohomology sheaves of $K_x$ are $\cH^0(K_x) = F$ and $\cH^1(K_x)=\cO_x$. 

Next we show that 
\begin{equation*}
    \hom^i(K_x, K_x) = 
    \begin{cases}
1 & \text{ if $i = 0,2$} \\ 
4 & \text{ if $i = 1$} \\ 
0 & \text{ else}
    \end{cases}
\end{equation*}
By adjunction we have $$\Hom_{\Ku(Y)}(K_x,K_x[i])=\Hom_Y(\cO_x,K_x[i+2]),$$
so the claim for $i \geq 2$ follows directly from the spectral sequence 
\begin{equation*}
    E_2^{p,q}=\Hom(\cO_x,\cH^q(K_x)[p]) \implies  \Hom(\cO_x,K_x[p+q]). 
\end{equation*}
Since $K_x$ is $\iota$-invariant, we have
$$\Hom(K_x, K_x[j])=\Hom(K_x, \iota(K_x)[2-j])=\Hom(K_x,K_x[2-j]),$$ 
so the claim for $i \leq 0$ follows. Finally, using~\eqref{class-Kx} we compute $\chi(K_x, K_x) = -2$, from which the claim for $i = 1$ follows. 

If $K_x$ is not $\sigma_Y$-semistable, then consider the destabilizing sequence $A \to K_x \to B$ where $B$ is the last semistable factor in the HN filtration. Applying $\iota$ we get the sequence 
\begin{equation*} 
\iota(A) \to \iota (K_x) \cong K_x \to \iota(B). 
\end{equation*} 
Since $\sigma$ is $\iota$-invariant, we have that $\iota(B)$ is semistable of the same phase as $B$. But then $\iota(B) \cong B$ by uniqueness of the HN filtration. It follows that 
$$\hom(A, B)=\hom(A, \iota(B))=0.$$
By \cite[Lemma 6.4]{PY} it follows that
$$\hom^1(A, A)+ \hom^1(B, B) \leq \hom^1(K_x, K_x)=4.$$
By \cite[Proposition 3.3(c)]{FP} the only possibility is $\hom^1(A,A)=\hom^1(B,B)=2$. However, $B$ would then be a $\iota$-invariant semistable object with $\hom^1(B,B)=2$, and hence 
\begin{equation*}
    \chi(B, B) = 2\hom(B,B) - \hom^1(B,B) \geq 0. 
\end{equation*}
This contradicts the inequality $\chi(v,v) \leq -1$ for every $0 \neq v \in \Knum(\Ku(Y))$. We conclude that $K_x$ is $\sigma_Y$-stable.
\end{proof}

The above computations show that $\Ku(Y)$ and $\Db(X)$ together with the stability conditions $\sigma_Y$ and $\sigma_X$ satisfy the assumptions of Theorem~\ref{theorem-moduli}. Moreover, since $\sigma_X$ is a geometric stability condition, by \cite{bayer-macri-projectivity} the moduli space $M_{\sigma_X}(\Db(X), w)$ is nonempty if and only if the inequality $w^2+2 \geq 0$ holds, where $w^2=-\chi(w,w)$. Thus the missing ingredient to apply Theorem~\ref{theorem-moduli} for proving Theorem~\ref{theorem-M-quartic-double} is to find for every $v \in \Knum(\Ku(Y))$, a class $w \in \Knum(X)$ such that $\Phi_*(w)=v$ and $w^2 \geq -2$. We will see in the next section that this property does not hold for every $v \in \Knum(\Ku(Y))$, so instead we will use the deformational criterion of Theorem~\ref{theorem-deformation}. 

\subsection{Numerical Grothendieck groups}

Let us keep the notation of \S\ref{sec_qds_preparation}. 

\begin{lemma} \label{lemma_forgofO-H}
We have the following equalities in $\Knum(\Ku(Y))$: 
\begin{align*}
\Phi_*([\cO_x]) & = 2\mu_1 - \mu_2  \quad \textup{for all~~}  x \in X, \\ 
    \Phi_*([\cO_X(-H)]) & =2 \mu_1 , \\ 
    \Phi_*([\cO_X]) & =2 \mu_1-2\mu_2 ,  \\ 
    \Phi_*([\cO_H]) & =-2 \mu_2. 
\intertext{Further, if $X$ contains a line $L$, then}
    \Phi_*([\cO_L]) & =3\mu_1-2\mu_2. 
\end{align*} 
\end{lemma}

\begin{proof}
We already computed $\Phi_*([\cO_x])$ in the proof of Lemma~\ref{lemma_geometricsigma_qds}. 

For $\Phi_*([\cO_X(-H)])$, 
by the formula~\eqref{formula-forgphi} for $\Phi$, we need to compute the numerical class of $\rL_{\langle \cO_Y, \cO_Y(H) \rangle}j_*(\cO_X(H))$. 
A straightforward computation gives 
\begin{equation*}
    \chi(\cO_Y(H), j_*\cO_X(H)) = 2, \quad 
    \chi(\cO_Y, j_*\cO_X(H)) = 4, \quad 
    \chi(\cO_Y, \cO_Y(H)) = 4. 
\end{equation*}
Therefore, in $\Knum(\Ku(Y))$ we have 
\begin{equation*}
   [ \rL_{\cO_Y(H)}j_*(\cO_X(H)) = [j_*(\cO_X(H))] - 2 [\cO_Y(H)],  
\end{equation*}
from which we find 
\begin{align*}
    [ \rL_{\langle \cO_Y, \cO_Y(H) \rangle}j_*(\cO_X(H))] &= 
    [j_*(\cO_X(H))] - 2 [\cO_Y(H)] + 4[\cO_Y]  \\
    & = -[\cO_Y(H)] - [\cO_Y(-H)] + 4[\cO_Y], 
\end{align*}
where in the last line we used that the class of $X \subset Y$ is $2H$. 
Now computing the Chern character of this class and comparing to~\eqref{eq_basisKnumqds}, we conclude that it equals $2\mu_1$. 

The formulas for $\Phi_*([\cO_X])$ and $\Phi_*([\cO_L])$ can be proved similarly, while the one for $\Phi_*([\cO_H])$ follows from $[\cO_H] = [\cO_X] - [\cO_X(-H)]$. 
\end{proof}

\begin{remark} \label{rmk_aodd_noimage}
Since $[\cO_X], [\cO_H], [\cO_x]$ form a basis for the numerical Grothendieck group of a general K3 surface, in this case it follows that $\Phi_*$ is not surjective. Indeed, if $X$ is general, i.e.\ $X$ has Picard rank $1$, then classes of the form $a \mu_1+b \mu_2$, where $a$ is an odd integer, are not in the image.   
\end{remark}

Theorem~\ref{theorem-M-quartic-double} will be obtained in virtue of Theorem~\ref{theorem-deformation}, by specializing to quartic double solids whose ramification K3 surface contains a line, using the following lemma.

\begin{lemma} \label{lemma_deformationtoquartiwithline}
Assume that $Y$ is a quartic double solid whose ramification K3 surface $X$ contains a line $L$. Then for every $v \in \Knum(\Ku(Y))$, there exists $w \in \Knum(X)$ such that $\Phi_*(w)=v$ and $w^2 \geq -2$.
\end{lemma}
\begin{proof}
Consider $v=a\mu_1+b\mu_2 \in \Knum(\Ku(Y))$. Note that it is enough to prove the statement for $a \geq 0$. Indeed, if $a<0$, we have $\Phi_*(-w)=v$, where $w \in \Knum(X)$ satisfies $\Phi_*(w)=-v$. So assume $a \geq 0$. We distinguish several case.

Assume $a$ is even, so $a=2a'$. If $a'+b \geq 0$, we may write  
\begin{align*}
    v=2a'\mu_1+b\mu_2 &=a'(2\mu_1-\mu_2)+(a'+b)\mu_2  \\ 
    &=a'\Phi_*([\cO_x])+(a'+b)\mu_2 \\ 
    & =a'\Phi_*([\cO_x])-(a'+b)\Phi_*([\cI_x]), 
\end{align*}
since $\Phi_*([\cI_x])=\Phi_*([\cO_X])-\Phi_*([\cO_x])=-\mu_2$. Then the statement follows from 
$$(a'[\cO_x]-(a'+b)[\cI_x])^2=-2a'(a'+b) \langle [\cO_x], [\cI_x] \rangle=2a'(a'+b) \geq 0$$
by our assumption. Instead if $a'+b < 0$, write $b=-a'-k$ for $k >0$. If $k=2h$ is even, then write \begin{equation*}
    v=2a'\mu_1-(a'+2h)\mu_2=a'\Phi_*([\cO_x])-h(2\mu_2)=a'\Phi_*([\cO_x])+h\Phi_*([\cO_H]).
\end{equation*}
Then the statement follows from
$$(a'[\cO_x]+h[\cO_H])^2=4h^2 \geq 0.$$
Otherwise, we can write $k=2h-1$ for $h \geq 1$. Then write 
\begin{equation*}
    v=2a'\mu_1-(a'+2h-1)\mu_2=a'(2\mu_1-\mu_2)+ \mu_2 -h(2\mu_2)=a'\Phi_*([\cO_x])- \Phi_*([\cI_x])+ h\Phi_*([\cO_H]).
\end{equation*}
    Then 
$$(a'[\cO_x]-[\cI_x]+h[\cO_H])^2 = 4h^2 + 2a' -4h \geq -2.$$
Indeed, the above inequality is equivalent to
$$(h-1)^2+h^2+a' \geq 0$$
which is satisfied since $a' \geq 0$. 

Now if $a$ is odd, write $a=2a'+1$. If $a'+b+1 \geq 0$, then 
\begin{align*}
    v & =(2a'+1)\mu_1+b\mu_2 \\ 
    & =a'(2\mu_1-\mu_2)+(a'+b)\mu_2+\mu_1  \\ 
    & =a'\Phi_*([\cO_x])-(a'+b)\Phi_*([\cI_x])+\Phi_*([\cO_L]-[\cO_X]). 
\end{align*}
Then 
$$(a'[\cO_x]-(a'+b)[\cI_x]+[\cO_L]-[\cO_X])^2 =  -2+2a'(a'+b+1) \geq -2,$$
as we want. If $a'+b+1< 0$, write $b=-1-a'-k$, for $k>0$. If $k=2h-1$ for $h \geq 1$, then $b=-a'-2h$ and we can write 
\begin{equation*}
    v=(2a'+1)\mu_1-(a'+2h)\mu_2=a'\Phi_*([\cO_x])+h\Phi_*([\cO_H])+\Phi_*([\cO_L]-[\cO_X]). 
\end{equation*}
Then 
$$(a'[\cO_x]+h[\cO_H]+[\cO_L]-[\cO_X])^2=4h^2-2h-2+2a' \geq -2,$$
since $h \geq 1$. Finally, if $k=2h$ for $h \geq 0$, then \begin{align*}
    v& =(2a'+1)\mu_1-(1+a'+2h)\mu_2 \\ 
    & =a'(2\mu_1-\mu_2)+(\mu_1-\mu_2)+h(2\mu_2) \\ 
    & =a'\Phi_*([\cO_x])+\Phi_*([\cO_L]-[\cO_x])+h\Phi_*([\cO_H]). 
\end{align*}    
The element $a'[\cO_x]+[\cO_L]-[\cO_x]+h[\cO_H]$ has square
$$-2+4h^2+2h \geq -2$$
as $h\geq 0$. This implies the statement.
\end{proof}

\subsection{Proof of Theorem~\ref{theorem-M-quartic-double}}
\label{section-proof-quartic-double}

We first prove part~\eqref{non-empty}. Since Serre invariant stability conditions on $\Ku(Y)$ are in the same $\wtilde{\text{GL}}^+(2, \bR)$-orbit by \cite[Theorem 3.2 and Remark 3.8]{FP}, and the stability conditions on $\Ku(Y)$ constructed by \cite{BLMS} are Serre invariant by \cite[Proposition 5.7]{PY}, it is enough to prove the result for $\sigma_Y \in \Stab(\Ku(Y))$ defined in \eqref{eq_sigmaY} (see also \eqref{eq_stabonKu}). Note also that we can reduce to proving the statement in the case $v$ primitive. Indeed, if $v=k(a \mu_1 + b \mu_2)$ and we know there exists $F \in M_{\sigma_Y}(\Ku(Y), a \mu_1 + b \mu_2)$, then $F^{\oplus k}$ will be an object in $M_{\sigma_Y}(\Ku(Y), v)$. 

Assume $v=a \mu_1+ b \mu_2$, where $a, b$ are coprime integers. 
Since all quartic double solids are deformation equivalent, we may choose a family 
$\cY \to S$ of quartic double solids over a smooth connected curve $S$ with two points $0, 1 \in S(\bC)$ such that $\cY_0 \cong Y$ and $\cY_1 \cong Y'$, where $Y'$ is a quartic double solid whose branch locus is a K3 surface $X'$ containing a line. Let $\fC:=\Ku(\cY)$ be the associated Enriques category over $S$ such that $\fC_0\cong \Ku(Y)$ and $\fC_1 \cong \Ku(Y')$ (see Remark~\ref{rmk_families_qdsGM}).
Moreover, by Remark~\ref{rmk_stabcondinfamily} there exists a stability condition $\underline{\sigma} = (\sigma_s)_{s \in S}$ on $\Ku(\cY)$ over $S$ such that $\sigma_s$ is a stability condition as defined in~\eqref{eq_stabonKu}. 
By construction, $\fC$ and $\underline{\sigma}$ satisfy the conditions~\eqref{1}-\eqref{defo-fiber-0} in Theorem~\ref{theorem-deformation}, with $\cC = \Ku(Y)$ and $\sigma = \sigma_Y$.

We claim that condition~\eqref{1-exists} also holds. 
First, $\sigma_{Y'} = \sigma_1$ satisfies \eqref{1-exists-a} due to Lemma~\ref{lemma_BLMSstabiotainv_qds} and the primitivity of $v$. The induced stability condition on the CY2 cover $\fD_1:=\Db(X')$, denoted $\sigma_{X'}$, satisfies \eqref{1-exists-b} by Lemma~\ref{lemma_geometricsigma_qds}. Finally, by Lemma~\ref{lemma_deformationtoquartiwithline} there exists a class $w \in \Knum(X')$ such that $\pi_*(w):=\Phi_*(w)=v$ and $w^2 \geq -2$. Thus by \cite{bayer-macri-projectivity} the moduli stack $\cM_{\sigma_{X'}}(\Db(X'), w)$ is nonempty. It remains to show 
$$-\chi(w,w)+2=w^2+2 < -\chi(v, v)+1= a^2+2ab+2b^2+1.$$
By a straightforward computation it is possible to check that this condition holds for all $w$ computed in Lemma~\ref{lemma_deformationtoquartiwithline}, thus \eqref{conditions-w1} is satisfied.

We conclude by Theorem~\ref{theorem-deformation} that the moduli stack $\cM_{\sigma_Y}(\Ku(Y), v)$ is nonempty, and so is the moduli space $M_{\sigma_Y}(\Ku(Y), v)$.

We now prove part~\eqref{v-smooth-M-KuY} of Theorem~\ref{theorem-M-quartic-double}. As noted in Remark~\ref{rmk_aodd_noimage} the class $v$ is not in the image of $\Phi_*$. By Theorem~\ref{theorem-moduli}\eqref{theorem-enriques-moduli} it follows that the fixed locus $\cM_{\sigma_Y}(\Ku(Y), v)^{\bZ/2}$ of the action induced by $\iota$ on $\cM_{\sigma_Y}(\Ku(Y), v)$ is empty, so $\cM_{\sigma_Y}(\Ku(Y), v)$ is smooth and there is a $\bG_m$-gerbe $\cM_{\sigma_Y}(\Ku(Y), v) \to M_{\sigma_Y}(\Ku(Y), v)$. By deformation theory the dimension of $M_{\sigma_Y}(\Ku(Y), v)$ at any point $E$ is equal to $\dim \Ext^1(E,E) = - \chi(v,v) + 1$. Finally, $M_{\sigma_Y}(\Ku(Y), v)$ is projective by \cite[Corollary 3.4]{Villa}. \qed 

\begin{remark}
\label{remark-projectivity}
We expect that the moduli space $M_{\sigma}(\Ku(Y), v)$ is projective for every $Y$, $v$, and $\sigma$ Serre invariant. The strategy to show this would be to make use of Theorem~\ref{theorem-moduli}\eqref{theorem-CY2-moduli} and deduce the projectivity of $M_{\sigma_Y}(\Ku(Y), v)$ from the projectivity of the moduli space of semistable objects in the associated CY2 category $\Db(X)$. However, the stability condition on $\Db(X)$ may fail to be $w$-generic and, at this moment, structure results for moduli spaces of semistable objects in K3 categories with respect to non-generic stability conditions are unavailable.
\end{remark}

\begin{remark}
Note that it is possible to show a stronger version of \eqref{conditions-w1} in the case of the deformation used in the proof of Theorem~\ref{theorem-M-quartic-double}. More precisely, a straightforward but long computation implies that for every lifting $w$ of $v$ the inequality 
\begin{equation*}
    -\chi(w,w)+2=w^2+2 < -\chi(v, v)+1
\end{equation*} 
holds. As a consequence, no irreducible component of $M_{\sigma}(\Ku(Y), v)$ is fixed by the $\bZ/2$-action.
\end{remark}


\section{Gushel--Mukai threefolds} 
\label{section-GM-3folds}

In this section we prove Theorem~\ref{theorem-Msigma-GM}\eqref{nonempty-GM}, as an application of Theorem~\ref{theorem-deformation}. The strategy consists in deforming to a special GM threefold and arguing similarly as in the case of the quartic double solid. The difficulty in this case compared to the previous one is that a different deformation is needed for each numerical class $v$. 
We postpone the proof of Theorem~\ref{theorem-Msigma-GM}\eqref{Msigma-GM-smooth} until \S\ref{subsection-Msigma-GM-smooth}, as it uses the relation to GM fourfolds discussed later. 

\subsection{Special GM threefolds}

A special GM threefold is a double cover $f \colon Y \to M$ of a codimension three linear section $M$ of $\Gr(2,5)$ ramified over a smooth divisor $S \in |\cO_M(2H)|$, where $H$ denotes the (restriction to $M$ of the) ample generator of $\Pic(\Gr(2,5))$. The branch locus $S$ is a GM surface, or equivalently a Brill-Noether general polarized K3 surface of degree $10$ by \cite[Proposition 2.13]{DebKuz:birGM}. As mentioned in Example~\ref{ex_GM}, the Kuznetsov component $\Ku(Y)$ is an Enriques category with CY2 cover given by $\Db(S)$.

In fact, we are in the setting of \S\ref{subsec_generalsetting} with $d=1$, $m=2$, $\cB=\langle \cO_M, \cU_M^\vee \rangle$. In particular, we have the equivalence
$$\Phi_0:=\rL_{\langle \cO_Y \otimes 1, \cU_Y^\vee \otimes 1 \rangle} \circ (j_0)_* \circ (- \otimes \cO_S(H)) \colon \Db(S) \simeq \Ku(Y)^{\bZ/2}.$$
The $\bZ/2$-action on $\Ku(Y)$ is induced by the involutive autoequivalence $\iota \colon \Ku(Y) \to \Ku(Y)$ defined by the double cover involution. 

In the next lemmas, we show that the known stability conditions on $\Ku(Y)$ are $\iota$-fixed, and the induced stability conditions on $\Db(S)$ are geometric. 

\begin{lemma} \label{lemma_BLMSiotainv_GM3}
Let $\sigma_Y$ be a stability condition constructed in \cite{BLMS}. Then $\sigma_Y$ is $\iota$-invariant.
\end{lemma}
\begin{proof}
Recall that the Serre functor of $\Ku(Y)$ satisfies $S_{\Ku(Y)}=\iota[2]$. Then the result follows from \cite[Theorem 1.1, Corollary 4.3]{PR}. Alternatively, one can argue as in Lemma~\ref{lemma_BLMSstabiotainv_qds}.
\end{proof}

By Theorem~\ref{theorem-equivariant-stability} it follows that $\sigma_Y$ induces a stability condition $\sigma_Y^{\bZ/2}$
on $\Ku(Y)^{\bZ/2}$ and thus a stability condition 
$$\sigma_S:=\Phi_0^{-1}(\sigma_Y^{\bZ/2})$$ on $\Db(S)$.
Similarly to \S\ref{section-quartic-double-solids}, we introduce the notation  
\begin{equation}
    \Phi \coloneqq \Forg \circ \Phi_0 \colon \Db(S) \to \Ku(Y), 
\end{equation}
which Lemma~\ref{lemma1} can be described as 
\begin{equation}
\label{formula-forgphi-GM3}
    \Phi = \rL_{\langle \cO_Y, \cU_Y^{\vee} \rangle} \circ j_* \circ (- \otimes \cO_X(H)) .
\end{equation}

\begin{lemma} \label{lemma_geometricstability_GM3}
The stability condition $\sigma_S$ on $\Db(S)$ is geometric.
\end{lemma}

\begin{proof}
We need to show that the skyscraper sheaf $\cO_x$ at any point $x \in S$ is $\sigma_S$-stable. 
By Theorem~\ref{theorem-equivariant-stability} it is enough to show that $\Phi(\cO_x) = \rL_{\langle \cO_Y, \cU_Y^\vee \rangle}(\cO_x)$ is $\sigma_Y$-stable.  
By \cite[Theorem 5.9]{JLLZ} it is known that $\rL_{\langle \cU_Y, \cO_Y \rangle}(\cO_x)$ is stable with respect to any Serre invariant stability condition on $\langle \cU_Y, \cO_Y \rangle^{\perp}$. By Serre duality it is straightforward to check that $\langle \cU_Y, \cO_Y \rangle^{\perp}= \rL_{\cU_Y}\Ku(Y)$, and by \cite[Corollary 3.16]{PR} the $\wtilde{\text{GL}}^+(2, \bR)$-orbit of the stability conditions constructed in \cite{BLMS} is preserved by $\rL_{\cU_X}$. Combining these observations, we deduce the $\sigma_Y$-stability of $\rL_{\langle \cO_Y, \cU_Y^\vee \rangle}(\cO_x)$, as we wanted.
\end{proof}

By \cite[Proposition 3.9]{Kuz_Fano} the numerical Grothendieck group $\Knum(\Ku(Y))$ has rank $2$, with a basis $\kappa_1, \kappa_2$ such that 
\begin{equation} \label{eq_basisKnumGM3}
\ch(\kappa_1) =1 - \frac{1}{5}H^2, \quad \ch(\kappa_2) =2-H +\frac{5}{6}p,     
\end{equation}
where $p$ denotes the class of a point. 
Moreover, the Euler form is given in this basis by 
$$\langle \kappa_1,\kappa_2 \rangle= 
\begin{pmatrix}
-1 & 0 \\ 
0 & -1
\end{pmatrix}. 
$$
In fact, this description of $\Knum(\Ku(Y))$ applies to any GM threefold $Y$, whether special or ordinary. 
As in Lemma~\ref{lemma_iotaisidentity_qds}, we have the following.
\begin{lemma} \label{lemma_iotaisidentity_GM3}
The involution $\iota_*$ induced by $\iota$ on $\Knum(\Ku(Y))$ is trivial.    
\end{lemma}

As in the case of quartic double solids, the next lemma implies that Theorem~\ref{theorem-moduli} is not enough to prove the desired nonemptiness (see Remark~\ref{rmk_aodd_noimage}). 

\begin{lemma} \label{lemma_phifromtheK3}
We have the following equalities in $\Knum(\Ku(Y))$: 
\begin{align*}
    \Phi_*([\cO_X(-H)]) & =5 \kappa_1+ 4\kappa_2, \\
    \Phi_*([\cO_X]) & =4\kappa_2, \\
    \Phi_*([\cO_x]) & =\kappa_1+2\kappa_2.
\end{align*}
\end{lemma}
\begin{proof}
We start with $\Phi_*([\cO_X(-H)])$ which is equal to the class of $\rL_{\langle\cO_Y, \cU_Y^\vee \rangle}(j_*\cO_X)$. Consider the short exact sequence
\begin{equation} \label{eq_sesforX}
0 \to \cO_Y(-H) \to \cO_Y \to j_*\cO_X \to 0.    
\end{equation}
Tensoring it by $\cU_Y$ and since $\rH^i(Y, \cU_Y)=0$ for every $i$, $\rH^i(Y, \cU_Y(-H))=\rH^{3-i}(Y, \cU_Y^\vee)$ which is $0$ for $i \neq 3$ and $\bC^5$ for $i=3$, we get the triangle
$$\cU_Y^{\vee \oplus 5}[-2] \to j_*\cO_X \to \rL_{\cU_Y^\vee }j_*\cO_X.$$
Since $\chi(\cO_Y, \cO_X)=2$, we have $\chi(\cO_Y, \rL_{\cU_Y^\vee }j_*\cO_X)=-23$. It follows that
$$\ch(\Phi(\cO_X(-H)))= \ch(j_*\cO_X)- 5 \ch(\cU_Y^\vee)+ 23 \ch(\cO_Y)=- \ch(\cO_Y(-H)) - 5 \ch(\cU_Y^\vee)+ 24 \ch(\cO_Y),$$
which is equal to $13-4H-H^2+\frac{10}{3}p$. Thus $\Phi_*([\cO_X(-H)])=5 \kappa_1+ 4\kappa_2$.

Similarly, to compute $\Phi_*([\cO_X])$ we use the sequence~\eqref{eq_sesforX} and we get the triangle
$$\cU_Y^{\vee \oplus 5}\to j_*\cO_X(H) \to \rL_{\cU_Y^\vee }j_*\cO_X(H).$$
Since $\chi(\cO_Y, \rL_{\cU_Y^\vee }j_*\cO_X(H))= -18$, we conclude that
$$\ch(\Phi(\cO_X))=\ch(\cO_Y(H))- 5 \ch(\cU_Y^\vee)+ 17 \ch(\cO_Y)= 8-4H+\frac{10}{3}p,$$
as we wanted.

Finally, we compute that $\Phi_*([\cO_x])=[\cO_x]-2[\cU_Y^\vee]+9[\cO_Y]=\kappa_1+2\kappa_2$.
\end{proof}

\subsection{Construction of special GM threefolds} 

The goal of this subsection is to show that 
for any class $v \in \Knum(\Ku(X))$, where $X$ is a GM threefold, there exists a special GM threefold $Y$ for which $v$ lies in the image of $\Phi_* \colon \Knum(S) \to \Knum(\Ku(Y))$ (under the identification $\Knum(\Ku(Y)) = \Knum(\Ku(X) = \langle \kappa_1, \kappa_2 \rangle$). 

The surjectivity of the period map for K3 surfaces ensures a lot of freedom in finding a suitable lifting for $v$, but one issue is to make sure that the K3 surface is Brill-Noether general. We deal with this in the next lemma, making use of \cite[Lemma 2.8]{GLT}. 

\begin{lemma} \label{lemma_K3BNgeneral}
\begin{enumerate}
\item \label{x2} For every integer $x\geq 6$ there exists a special GM threefold $Y$ whose branch locus is a Brill-Noether general K3 surface $S$ of degree $10$ with Picard group generated by two classes $H$ and $L$ with intersetion matrix
\begin{equation} \label{eq_PicS}
\Pic(S)= \langle H, L \rangle=
\begin{pmatrix}
10 & x\\
x & 2
\end{pmatrix}.    
\end{equation}
\item \label{50} There exists a special GM threefold $Y$ whose branch locus is a Brill-Noether general K3 surface $S$ of degree $10$ with Picard group generated by two classes $H$ and $L$ with intersetion matrix
\begin{equation} \label{eq_PicS_50}
\Pic(S)= \langle H, L \rangle=
\begin{pmatrix}
10 & 5\\
5 & 0
\end{pmatrix}.    
\end{equation}
\item \label{9and74} For $x=7,9$, there exists a special GM threefold $Y$ whose branch locus is a Brill-Noether general K3 surface $S$ of degree $10$ with Picard group generated by two classes $H$ and $L$ with intersetion matrix
\begin{equation} \label{eq_PicS_9and74}
\Pic(S)= \langle H, L \rangle=
\begin{pmatrix}
10 & x\\
x & 4
\end{pmatrix}.    
\end{equation}
\end{enumerate}
\end{lemma}
\begin{proof}
We show \eqref{x2}; the proof of \eqref{50}-\eqref{9and74} is similar. By the surjectivity of the period map for K3 surfaces, there exists a polarized K3 surface $(S, H)$ with Picard group as in \eqref{eq_PicS}, since $20 -x^2< 0$ and $\Pic(S)$ does not contain $(-2)$-classes orthogonal to $H$ \cite[Remark 6.3.7]{Huy:K3}. 
Since $x \neq 5$, by \cite[Lemma 2.8]{GLT} the pair $(S,H)$ is a polarized Brill-Noether general K3 surface of degree $10$. In particular, $S$ is a GM surface by \cite[Proposition 2.13]{DebKuz:birGM}. Moreover, by \cite[Lemma 2.7]{GLT} we have that $S$ is contained in a prime Fano threefold $M$ of index $2$ and degree $5$, which is a codimension $3$ linear section of the Grassmannian $\Gr(2, 5)$. Then the double cover $Y$ of $M$ ramified in $S$ is a special GM threefold satisfying the required properties. 
\end{proof}

The following lemmas provide for any $v \in \langle \kappa_1, \kappa_2 \rangle$ a special GM threefold such that there is a class $w$ in the numerical Grothendieck group of the branch K3 surface which lifts $v$ along $\Phi_*$ and satisfies $w^2 \geq -2$.

\begin{lemma} \label{lemma_phiL}
The following equalities hold in $\Knum(\Ku(Y))$: 
\begin{enumerate}
\item \label{L-in-x2} If $x$, $Y$ and $S$ are as in Lemma~\ref{lemma_K3BNgeneral}\eqref{x2}, then 
    $\Phi_*([\cO_S(L)])=\kappa_1+(6+x)\kappa_2$. 
\item \label{L-in-50} If $Y$ and $S$ are as in Lemma~\ref{lemma_K3BNgeneral}\eqref{50}, then $\Phi_*([\cO_S(L)])=9\kappa_2$.  
\item \label{L-in-9and74} If $Y$ and $S$ are as in Lemma~\ref{lemma_K3BNgeneral}\eqref{9and74}, then $\Phi_*([\cO_S(L)])=2\kappa_1+17\kappa_2$ when $x=9$, $\Phi_*([\cO_S(L)])=2\kappa_1+15\kappa_2$ when $x=7$.  
\end{enumerate}
\end{lemma}
\begin{proof}
We show \eqref{L-in-x2}; the proof of \eqref{L-in-50}-\eqref{L-in-9and74} is similar. By~\eqref{formula-forgphi-GM3} we need to compute the class of $\rL_{\langle\cO_Y, \cU_Y^\vee \rangle}j_*(\cO_S(L+H))$. Note that
$$\ch(\cU_S(L+H))=2 + (2L+H)+ (3+x).$$ 
Thus $\chi(\cU_Y^\vee, j_*\cO_S(L+H))=\chi(S, \cU_S(L+H))=7+x$. In particular, we have
$$[\rL_{\cU_Y^\vee}(j_*\cO_S(L+H))]=[j_*\cO_S(L+H)]-(7+x)[\cU_Y^\vee].$$
Using the above equation, we get $$\chi(\cO_Y, \rL_{\cU_Y^\vee}(j_*\cO_S(L+H)))=\chi(\cO_Y, j_*\cO_S(L+H)))-5(7+x)=-27-4x.$$
We conclude that
$$[\Phi(\cO_S(L))]= [j_*\cO_S(L+H)]-(7+x)[\cU_X^\vee]+(27+4x)[\cO_Y]=\kappa_1+(6+x) \kappa_2. \qedhere $$
\end{proof}

\begin{lemma} \label{lemma_podd}
Consider $v=p \kappa_1+q \kappa_2$ where either $q$ is a positive integer and $p$ an odd integer coprime with $q$, or $q=0$ and $p=1$. Then there is a special GM threefold $Y$ with branch K3 surface $S$ such that there exists $w \in \Knum(S)$ satisfying $\Phi_*(w)=v$ and $w^2 \geq -2$.
\end{lemma}
\begin{proof}
Set $x=q+6$. By Lemma~\ref{lemma_K3BNgeneral}\eqref{x2} there exists a special GM threefold $Y$ whose associated K3 surface $S$ has Picard group as in \eqref{eq_PicS}. Consider $w \in \Knum(S)$ defined by
$$w:=-[\cO_S(L)]+b[\cO_S]-q[\cO_S(-H)]+(p+5q+1)[\cO_x],$$
where $b:=-q+\frac{1}{2}(5-p) \in \bZ$, as $p$ is odd. Then $w^2+2 =\frac{1}{2}p^2-\frac{1}{2} \geq 0$. By a straightforward computation using Lemmas~\ref{lemma_phifromtheK3} and  \ref{lemma_phiL}\eqref{L-in-x2} one can check that $\Phi_*(w)=v$. 
\end{proof}

\begin{lemma} \label{lemma_peven}
Consider $v=p \kappa_1+q \kappa_2$ where $q$ is a positive integer and $p$ an even integer coprime with $q$. Then there is a special GM threefold $Y$ with branch K3 surface $S$ such that there exists $w \in \Knum(S)$ satisfying $\Phi_*(w)=v$ and $w^2 \geq -2$.
\end{lemma}
\begin{proof}
Assume $q \neq 1$. Set $x=q+4$. The same argument as in Proposition~\ref{lemma_podd} with 
$$w:=-[\cO_S(L)]+b[\cO_S]-q[\cO_S(-H)]+(p+5q+1)[\cO_x],$$
where $b:=-q+ 2-\frac{1}{2}p\in \bZ$, as $p$ is even, implies the statement in this case, where $w^2= \frac{1}{2}p^2$.

Analogously, if $q=1$ and $p \neq 0, \pm 2$, set $x=9$. Then 
$$w:=-[\cO_S(L)]+(-\frac{1}{2}p+2)[\cO_S]-q[\cO_S(-H)]+(p+6)[\cO_x]$$
satisfies $w^2+2=\frac{1}{2}p^2-4 \geq 0$ and $\Phi_*(w)=v$.

In the case $q=1, p=0$, consider a special GM threefold whose branch $S$ is as in Lemma~\ref{lemma_K3BNgeneral}\eqref{50}. Then $w=[\cO_S(L)]-2[\cO_S]$ has square $-2$ and satisfies $\Phi_*(w)=v$ by Lemmas~\ref{lemma_phifromtheK3} and  \ref{lemma_phiL}\eqref{L-in-50}.

If $q=1, p=2$, set $x=9$ and consider a special GM threefold whose branch $S$ is as in Lemma~\ref{lemma_K3BNgeneral}\eqref{9and74}. Then $w=[\cO_S(L)]-4[\cO_S]$ has square $-2$ and satisfies $\Phi_*(w)=v$ by Lemmas~\ref{lemma_phifromtheK3} and  \ref{lemma_phiL}\eqref{L-in-9and74}.

If $q=1, p=-2$, set $x=7$ and consider a special GM threefold whose branch $S$ is as in Lemma~\ref{lemma_K3BNgeneral}\eqref{9and74}. Then $w=-[\cO_S(L)]+4[\cO_S]$ has square $-2$ and satisfies $\Phi_*(w)=v$ by Lemmas~\ref{lemma_phifromtheK3} and  \ref{lemma_phiL}\eqref{L-in-9and74}.
\end{proof}

\begin{lemma} \label{lemma_qnegative}
Consider $v=p \kappa_1+q \kappa_2$ where $q$ is a negative integer and $p$ an integer coprime with $q$. Then there is a special GM threefold $Y$ with branch K3 surface $S$ such that there exists $w \in \Knum(S)$ satisfying $\Phi_*(w)=v$ and $w^2 \geq -2$.
\end{lemma}
\begin{proof}
Consider $v':=-v$. Then we can apply either Lemma~\ref{lemma_peven} or Lemma~\ref{lemma_podd} which imply there exists a special GM threefold $Y$ whose branch K3 surface $S$ has a class $w' \in \Knum(S)$ such that $\Phi_*(w')=v'$ and $(w')^2 \geq -2$. Then $w:=-w'$ satisfies the required properties.
\end{proof}

\subsection{Proof of Theorem~\ref{theorem-Msigma-GM}\eqref{nonempty-GM}}
Recall that Serre invariant stability conditions on $\Ku(X)$ are in the same $\wtilde{\text{GL}}^+(2, \bR)$-orbit by \cite[Corollary 4.5]{PR}, and the stability conditions on $\Ku(X)$ constructed by \cite{BLMS} are Serre invariant by \cite[Theorem 1.1]{PR}. Thus it is enough to prove the statement for a stability condition $\sigma_X \in \Stab(\Ku(X))$ as constructed by \cite{BLMS}. Note also that we can reduce to proving the result in the case $v$ primitive, as explained in the proof of Theorem~\ref{theorem-M-quartic-double}.

Set $v=p \kappa_1+ q \kappa_2$, where $p, q$ are coprime integers. 
Since all GM threefolds are  deformation equivalent, 
by Lemmas~\ref{lemma_podd}, \ref{lemma_peven}, and \ref{lemma_qnegative}
we may choose a family 
$\cX \to S$ of GM threefolds over a smooth connected curve $S$ with two points $0, 1 \in S(\bC)$ such that 
$\cX_0 \cong X$ and $\cX_1 \cong Y$, where $Y$ is a special GM threefold whose branch K3 surface $S$ has a class $w \in \Knum(S)$ satisfying $\Phi_*(w) = v$ and $w^2 \geq -2$. 
Now using the associated Enriques category $\fC \coloneqq \Ku(\cX)$ and arguing exactly as in the proof of Theorem~\ref{theorem-M-quartic-double} in \S\ref{section-proof-quartic-double}, we find that $M_{\sigma}(\Ku(X), v)$ is nonempty via the criterion of Theorem~\ref{theorem-deformation}.  \qed 

\begin{remark}
As for quartic double solids, it is possible to show a stronger version of \eqref{conditions-w1} in the case of the deformations used in the proof of Theorem~\ref{theorem-Msigma-GM}\eqref{nonempty-GM}. More precisely, it is possible to show that for every lifting $w$ of $v$ under $\Phi_*$ we have the inequality 
\begin{equation*}
    -\chi(w,w)+2=w^2+2 < -\chi(v, v)+1.
\end{equation*}
In particular, no irreducible component of $M_{\sigma}(\Ku(X), v)$ is fixed by the $\bZ/2$-action.
\end{remark}


\section{Gushel--Mukai fourfolds}
\label{section-GM-4folds}

In this section, we first prove Theorem~\ref{theorem-Msigma-GM}\eqref{Msigma-GM-smooth}, after some preliminaries on special GM fourfolds. 
Then we show that for such a fourfold $W$, every stability condition $\sigma \in \Stab^\circ(\Ku(W))$ constructed in \cite{PPZ} is fixed by natural $\bZ/2$-action (Theorem \ref{thm_invariantstab}). 
Finally, we prove Theorem~\ref{theorem-fixed-locus} by combining this with Theorem~\ref{theorem-Msigma-GM}\eqref{nonempty-GM} and Theorem~\ref{theorem-moduli}.

\subsection{Preparation}

Let $W$ be a special GM fourfold, which is defined as the double cover $f \colon W \to M$ of a codimension $2$ linear section $M$ of the Grassmannian $\Gr(2,5)$ ramified in a quadric section $X \subset M$. Note that $X$ is an ordinary GM threefold. 
We denote by $i \colon X \hookrightarrow M$ and $j \colon X \hookrightarrow W$ the natural inclusions. 

We are in the setting of \S\ref{subsec_generalsetting} with $d=1$, $m=3$, $\cB=\langle \cO_M, \cU^\vee_M \rangle$, where $\cU_M$ denotes the restriction to $M$ of the rank $2$ tautological bundle on $\Gr(2,5)$ (see \cite[\S8.2]{KuzPer_cycliccovers}). More precisely, we have the semiorthogonal decompositions
$$\Db(X)= \llangle \Ku(X), \cB_X \rrangle \quad \text{and} \quad \Db(W)= \llangle \Ku(W), \cB_W, \cB_W(H) \rrangle$$
where
$$\cB_X= \llangle \cO_X, \cU_X^\vee \rrangle \quad \text{and} \quad  \cB_W=\llangle \cO_W, \cU_W^\vee \rrangle, $$
and the equivalence in \eqref{eq_Phi0} is 
\begin{equation}
\label{eq_Phi0_oGM}    
\Phi_0:=\rL_{\cB_W \otimes 1} \circ (j_0)_* \circ (- \otimes \cO_X(H)) \colon \Ku(X) \to \Ku(W)^{\bZ/2}.
\end{equation}
Then $\Ku(W)$ is the CY2 cover of the Enriques category $\Ku(X)$ (see Example~\ref{ex_GM}).

Recall that the numerical Grothendieck group of $\Ku(W)$ contains a canonical rank $2$ sublattice 
\begin{equation}\label{eq_lambda12}
\llangle \lambda_1,\lambda_2 \rrangle \cong -A_1^{\oplus 2}= 
\begin{pmatrix}
-2 & 0 \\ 
0 & -2
\end{pmatrix}    
\end{equation} 
by \cite[Proposition 2.25]{KuzPerry:dercatGM}, \cite[Eq.(4)]{Pert}.

We consider the functor 
\begin{equation*}
    \Forg \circ \Phi_0 \colon \Ku(X) \to \Ku(W),
\end{equation*}
which by Lemma~\ref{lemma1} can be given explicitly by the formula 
\begin{equation}
\label{formula-forgphi-GM4}
    \Forg \circ \Phi_0 = \rL_{\langle \cO_W, \cU_W^{\vee} \rangle} \circ j_* \circ (- \otimes \cO_X(H)) .
\end{equation}
In the next lemma, we compute the image via $(\Forg \circ \Phi_0)_*$ 
of the basis $\kappa_1, \kappa_2$ defined in \eqref{eq_basisKnumGM3}.

\begin{lemma} \label{lemma_Forgkappa}
For $i=1,2$ we have
$$(\Forg \circ \Phi_0)_*(\kappa_i)=\lambda_i. $$
\end{lemma}
\begin{proof}
The Todd classes of $X$ and $W$ can be computed to be 
$$\td(X)=1+\frac{1}{2}H+\frac{17}{60}H^2 + 1,$$
$$\td(W)=1+H+ (\frac{2}{3}H^2-\frac{1}{12}\gamma_W^*\sigma_2)+\frac{17}{60}H^3+1.$$
Here $\gamma_W$ is the double covering map composed with the inclusion of $M$ into $\Gr(2,5)$, and $\sigma_2$ is a Schubert cycle on $\Gr(2,5)$. By Grothendieck--Riemann--Roch we have
$$\ch(j_* \kappa_i(H))=j_*(\ch(\kappa_i(H)) \cdot \td(T_j))$$
where 
$$\td(T_j)=1 -\frac{1}{2}H+\frac{1}{6}H^2-\frac{5}{12}.$$
Now a computation using the description~\eqref{formula-forgphi-GM4} gives 
\begin{align*}
\ch((\Forg \circ \Phi_0)_*(\kappa_1)) & =-2+ \gamma_W^*\sigma_{1,1} -\frac{1}{2}, \\ 
\ch((\Forg \circ \Phi_0)_*(\kappa_2)) & =-4 +2H -\frac{1}{6}H^3.
\end{align*}
Using the description of $\lambda_1$ and $\lambda_2$ in \cite[Eq.(4)]{Pert}, we find that $ \ch(\lambda_1)$ and $\ch(\lambda_2)$ are given by the same formulas, so we conclude $(\Forg \circ \Phi_0)_*(\kappa_1) = \lambda_1$ and $(\Forg \circ \Phi_0)_*(\kappa_2) = \lambda_2$. 
\end{proof}

\subsection{Proof of Theorem~\ref{theorem-Msigma-GM}\eqref{Msigma-GM-smooth}}
\label{subsection-Msigma-GM-smooth}

Let $X$ be a generic GM threefold, in the sense that $X$ is the branch locus of a very general special GM fourfold $W$. 
By Theorem~\ref{theorem-moduli}\eqref{theorem-enriques-moduli} the moduli stack $\cM_{\sigma}(\Ku(X), v)$ can be singular only at points corresponding to objects $E$ such that $E \cong \tau(E)$, where $\tau$ is the generator of the $\bZ/2$-action. 

If $E \in \Ku(X)$ is such an object, then we claim there exists an object $F \in \Ku(X)^{\bZ/2}$ such that $\Forg(F) = E$. 
Indeed, $E$ is necessarily $\sigma$-stable, as its class $v$ is primitive and 
$\Knum(\Ku(Y))$ has rank $2$. 
In particular, $E$ is simple, so by \cite[Lemma 1]{Ploog} it admits a $\bZ/2$-linearization, i.e. a lift to $(\rh \Ku(X))^{\bZ/2}$, the $\bZ/2$-invariants of the homotopy category of $\Ku(X)$; by Proposition~\ref{proposition-enhancement-CG}, this means that $E$ also lifts to $\Ku(X)^{\bZ/2}$. 

As recalled above, there is an equivalence $\Phi_0 \colon \Ku(X) \to \Ku(W)^{\bZ/2}$. 
Set $E' \coloneqq \Phi_0(E)$. 
Then by the previous paragraph, there exists $F' \in (\Ku(W)^{\bZ/2})^{\bZ/2}$ such that $\Forg'(F') = E'$, where $\Forg' \colon (\Ku(W)^{\bZ/2})^{\bZ/2} \to \Ku(W)^{\bZ/2}$ is the forgetful functor for the residual $\bZ/2$-action on $\Ku(W)^{\bZ/2}$. 
If $\Forg \colon \Ku(W)^{\bZ/2} \to \Ku(W)$ denotes the forgetful functor, then by Lemma~\ref{lemma_generalnonsense} it follows that 
$$\Forg(E')= \Forg \Forg'(F')= \Forg \Inf \rho(F')=\rho(F') \oplus \iota(\rho(F')),$$
where $\rho \colon (\Ku(W)^{\bZ/2})^{\bZ/2} \to \Ku(W)$ is the equivalence of Theorem~\ref{thm_elagin} and $\iota$ is the generator of the $\bZ/2$-action on $\Ku(W)$.

Since $\Knum(\Ku(W)) \cong -A_1^{\oplus 2}$ by the assumption that $W$ is very general, and $\iota$ acts as the identity on $-A_1^{\oplus 2}$  by \cite[Proposition 5.7]{BP}, it follows that $[\iota(\rho(F'))]=[\rho(F')]$ in $\Knum(\Ku(W))$. Hence we conclude that $[\Forg(E')]=2[\rho(F')] \in \Knum(\Ku(W))$. However, this contradicts Lemma~\ref{lemma_Forgkappa} as $v$ is primitive. We conclude that $\cM_{\sigma}(\Ku(X), v)$ is smooth. We have a $\bG_m$-gerbe $\cM_{\sigma}(\Ku(X), v) \to M_{\sigma}(\Ku(X), v)$ and the moduli space $M_{\sigma}(\Ku(X), v)$ is smooth. By deformation theory the dimension of $M_{\sigma}(\Ku(X), v)$ at a point $E$ is equal to 
\begin{equation*} 
\dim \Ext^1(E,E)= -\chi(E,E) + \dim \Hom(E,E) =-\chi(v,v)+1, 
\end{equation*} 
where for the first equality we used Lemma~\ref{lemma-M-Enriques-smooth}. 

The projectivity of $M_{\sigma}(\Ku(X), v)$ follows from the smoothness and \cite[Corollary 3.4]{Villa}.

\subsection{Invariant stability conditions} \label{subsec_invstabGM}
Let $W$ be a special GM fourfold. 
By \cite{PPZ} there is a collection of proper full numerical stability conditions which we denote by $\Stab^{\circ}(\Ku(W))$. In this subsection, we show that every $\sigma_W \in \Stab^{\circ}(\Ku(W))$ is fixed by the $\bZ/2$-action induced by the covering involution. 

\begin{theorem}
\label{thm_invariantstab}
Let $W$ be a special GM fourfold, and $\sigma_W \in \Stab^{\circ}(\Ku(W))$. Then $\sigma_W$ is $\bZ/2$-fixed.
\end{theorem}

Recall that by construction, $\sigma_W$ is induced from a stability condition $\sigma_{W'}$ on $\Ku(W')$, where $W'$ is an ordinary GM fourfold with smooth canonical quadric which is a period partner of $W$ (see \cite[Theorem 4.18]{PPZ}). Indeed, the duality conjecture proved in \cite[Theorem 1.6]{KuzPerry:cones} implies $\Ku(W) \simeq \Ku(W')$. The stability condition $\sigma_W$ is obtained from the constructed $\sigma_{W'}$ through this equivalence. We have the following lemma.

\begin{lemma}
\label{lemma_fromspecialtoordinary}
Given $W$ and $W'$ as above, there exists an equivalence $\Ku(W) \simeq \Ku(W')$ which is $\bZ/2$-equivariant with respect to the following $\bZ/2$-actions: 
\begin{itemize}
\item on $\Ku(W)$ the group $\bZ/2$ acts via the covering involution,
\item on $\Ku(W')$ the group $\bZ/2$ acts via the rotation functor
$$\sO_{W'}:=\rL_{\cB_{W'}} \circ (- \otimes \cO_{W'}(H))[-1]$$
\end{itemize}
\end{lemma}
\begin{proof}
We denote by 
$$\iota \colon \Ku(W) \to \Ku(W)$$
the involutive autoequivalence induced by the covering involution. By \eqref{eq_rotfunct_iota} we have that $\iota \simeq \sO_W$. On the other hand, let $X'$ be the special GM fivefold obtained as a double cover of a linear section of the Grassmannian $\Gr(2,5)$ ramified in $W'$, and let $X$ be the branch ordinary GM threefold of $W$. In the terminology of Example~\ref{ex_GM}, 
the GM varieties $X$ and $X'$ are opposite to $W$ and $W'$. Note that $\Ku(X)$, $\Ku(X')$ are Enriques categories and $\Ku(W)$, $\Ku(W')$ are their CY2 covers. Since a GM variety and its opposite have the same period point (see \cite[\S3.1]{KuzPerry:dercatGM}), we have that $X$ and $X'$ are generalized period partners in the sense of \cite[Definition 3.5]{KuzPerry:dercatGM}. By \cite[Theorem 1.6]{KuzPerry:cones} there is an equivalence $\Ku(X) \simeq \Ku(X')$. By \cite[Lemma 4.9]{BP} this equivalence is equivariant with respect to the action induced by the $(-2)$-shifted Serre functors and induces an equivalence between the covers $\Ku(W)$ and $\Ku(W')$ which is equivariant with respect to the action of the rotation functors.
\end{proof}

As a consequence, in order to prove Theorem~\ref{thm_invariantstab}, it is enough to show that  for an ordinary GM fourfold $X$ with smooth canonical quadric, a stability condition $\sigma_X \in \Stab^\circ(\Ku(X))$ is invariant with respect to the action of the rotation functor $\sO_{X}$.\footnote{In the rest of this subsection we change the notation from $W'$ to $X$ to denote an ordinary GM fourfold with smooth canonical quadric in order to be compatible with \cite{PPZ}.} 

In order to do this, we need to go through the construction of $\sigma_X$. Denote by $T$ the smooth canonical quadric of $X$. By \cite[Lemma 2.1]{PPZ} the blowup $\widetilde{X}$ of $X$ in $T$ resolves the linear projection from $T$, giving a commutative diagram
$$
\xymatrix{
& \ar_-b[ld] \widetilde{X} \ar[rd]^-\pi & \\
X & & Y
}$$
where $Y$ is a quadric threefold and $\pi$ is a conic fibration. We denote by $h$ the hyperplane class on $Y$. Denote by $\Cl_0$ and $\Cl_1$ the even and odd parts of the sheaf of Clifford algebras associated to $\pi$, and let $\Db(Y, \Cl_0)$ be the bounded derived category of coherent $\Cl_0$-modules on $Y$. By \cite{kuznetsov08quadrics} there is a semiorthogonal decomposition of the form
$$\Db(\widetilde{X})= \llangle \Phi(\Db(Y, \Cl_0)), \pi^*\Db(Y) \rrangle,$$
where $\Phi \colon \Db(Y, \Cl_0) \to \Db(\widetilde{X})$ is a fully faithful functor. By \cite[Proposition 2.3]{PPZ} there is a fully faithful embedding
$$\xi:= \rL_{\cO_{\tX}(-h)}\circ \rL_{\cU_{\tX}} \circ b^* \colon \Ku(X) \hookrightarrow \Db(\tX)$$
and if $\Ku(X)'$ denotes the essential image of $\xi$, then we have the semiorthogonal decomposition
$$\Phi(\Db(Y, \Cl_0))= \llangle \Ku(X)', \cU_{\tX}, \cF_a, \cF_b, \cG \rrangle,$$
where $\cF_a, \cF_b, \cG$ are rank $2$ vector bundles on $\tX$. In particular, $\Ku(X)'$ is a full subcategory of $\Phi(\Db(Y, \Cl_0))$ which is equivalent to $\Ku(X)$. Moreover, by \cite[Lemma 2.15]{PPZ} there is an equivalence 
$$\Xi \colon \Phi(\Db(Y, \Cl_0)) \simeq \Db(Y, \Cl_0);$$
the restriction of $\Xi$ to $\Ku(X)'$ (which we still denote by $\Xi$ for simplicity) induces a fully faithful embedding
$$\Xi \colon \Ku(X)' \hookrightarrow \Db(Y, \Cl_0)$$
whose essential image is denoted by $\Ku(Y, \Cl_0)$. A stability condition $\sigma$ on $\Ku(Y, \Cl_0)$ is constructed in \cite[Theorem 4.12]{PPZ} by restricting a double tilting of the slope stability on $\Db(Y, \Cl_0)$, and by definition $\sigma_X=\xi^{-1} \cdot \Xi^{-1} \cdot \sigma$. 

Now, consider the rotation functor
$$\sO_{\widetilde{X}}:=\rL_{\pi^*\Db(Y)} \circ (- \otimes \cO_{\tX}(H))[-1] \colon \Db(\tX) \to \Db(\tX).$$
On the other hand, the functor
$$- \otimes_{\Cl_0} \Cl_1 \colon \Db(Y, \Cl_0) \to \Db(Y, \Cl_0)$$
defines a $\bZ/2$-action on $\Db(Y, \Cl_0)$, since $\Cl_1 \otimes_{\Cl_0} \Cl_1 \cong \Cl_0$ by \cite[Corollary 3.9]{kuznetsov08quadrics}. We call it the rotation functor on $\Db(Y, \Cl_0)$. In the next lemmas, we show that the functors $\Xi$ and $\xi$ are equivariant with respect to the actions of the rotations functors.

\begin{lemma} \label{lemma_Xiisequivariant}
The rotation functor $\sO_{\widetilde{X}}$ induces a  $\bZ/2$-action on $\Phi(\Db(Y, \Cl_0))$, and the equivalence $\Xi \colon \Phi(\Db(Y, \Cl_0)) \simeq \Db(Y, \Cl_0)$ is $\bZ/2$-equivariant with respect to the actions of $\sO_{\widetilde{X}}$ and $- \otimes_{\Cl_0} \Cl_1$.
\end{lemma}
\begin{proof}
First note that the restriction of $\sO_{\tX}$ to $\Phi(\Db(Y, \Cl_0))$ has image in $\Phi(\Db(Y, \Cl_0))$, thus $\sO_{\tX}$ defines an endofunctor $\sO_{\tX} \colon \Phi(\Db(Y, \Cl_0)) \to \Phi(\Db(Y, \Cl_0))$ which is an autoequivalence by \cite[Proposition 3.8]{kuznetsov-cubic}. 

Set $F:= - \otimes_{\Cl_0} \Cl_1$. We are going to show that there is an isomorphism of functors $\Xi \circ \sO_{\tX} \simeq F \circ \Xi$. As a consequence, we have that $\sO_{\tX}$ induces a $\bZ/2$-action on $\Phi(\Db(Y, \Cl_0))$, as $F$ does on $\Db(Y, \Cl_0)$. 

To prove the claim, note that the Serre functor of $\Phi(\Db(Y, \Cl_0))$ satisfies 
$$\rS^{-1}_{\Phi}:= \rL_{\pi^*\Db(Y)} \circ (- \otimes \cO_{\tX}(H+h))[-4]$$
 by \cite[Proposition 3.8]{kuznetsov-cubic}, since the canonical divisor of $\tX$ is $K_{\tX}=-H-h$. On the other hand, the Serre functor of $\Db(Y, \Cl_0)$ satisfies 
$$\rS_{\Cl_0}^{-1}=- \otimes_{\Cl_0} \Cl_1(h)[-3]$$
by \cite[Lemma 4.5]{PPZ}. Since every equivalence commutes with the Serre functor, we have
$$\Xi \circ \rS^{-1}_{\Phi} \simeq \rS^{-1}_{\Cl_0} \circ \Xi.$$
Spelling this out and using that 
$$\rL_{\pi^*\Db(Y)}(A(H+h))\cong \rL_{\pi^*\Db(Y)(-h)}(A(H)) \otimes \cO_{\tX}(h) \cong \rL_{\pi^*\Db(Y)}(A(H)) \otimes \cO_{\tX}(h)$$
for every $A \in \Phi(\Db(Y, \Cl_0))$, we get 
$$\Xi (\sO_{\tX}(-) \otimes \cO_{\tX}(h)) \simeq \Xi(-) \otimes_{\Cl_0} \Cl_1(h).$$
Recall that 
$$\Xi= \pi_*(\cG^\vee \otimes -)$$
by \cite[Eq.(2.24)]{PPZ}, where $\cG$ is a rank $2$ vector bundle on $\tX$. By projection formula and the fact that $\cO_{\tX}(h)=\pi^*\cO_Y(h)$, we see that
$$\Xi (\sO_{\tX}(-)) \otimes \cO_Y(h) \simeq \Xi(-) \otimes_{\Cl_0} \Cl_1(h),$$
which proves the claim.
\end{proof}

\begin{lemma}
\label{lemma_xiisequivariant}
The functor $\xi$ is $\bZ/2$-equivariant, i.e.\ there is an isomorphism of functors 
\begin{equation}
\xi \circ \sO_X \simeq \sO_{\tX} \circ \xi.
\label{eq_xivsrotationfunctors}    
\end{equation}
\end{lemma}
\begin{proof}
Note that the left-hand-side of \eqref{eq_xivsrotationfunctors} shifted by $1$ is
$$\rL_{\cO(-h)}\circ \rL_{\cU} \circ b^* \circ \rL_{\langle \cO, \cU^\vee \rangle} \circ (- \otimes \cO_X(H)) \simeq (- \otimes \cO_{\tX}(H)) \circ \rL_{\langle \cO(-h-H), \cU(-H), \cO(-H) \rangle} \circ \rL_{\cU} \circ b^*$$
using $\cU^\vee(-H) \cong \cU$.\footnote{In this proof, we often simply write $\cU$ instead of $\cU_{\tX}$ for the pullback of $\cU$ to $\tX$, and similarly for other pullback bundles.} Similarly, the right-hand-side shifted by $1$ is
$$\rL_{\pi^*\Db(Y)} \circ (- \otimes \cO_{\tX}(H)) \circ \rL_{\cO(-h)}\circ \rL_{\cU} \circ b^* \simeq (- \otimes \cO_{\tX}(H)) \circ \rL_{\pi^*\Db(Y)(-H)} \circ \rL_{\cO(-h)} \circ \rL_{\cU} \circ b^*.$$
We claim that
\begin{equation} \label{eq_relationonleftmutations}
\rL_{\langle \cO(-h-H), \cU(-H), \cO(-H) \rangle} \simeq \rL_{\pi^*\Db(Y)(-H)} \circ \rL_{\cO(-h)}
\end{equation}
on the subcategory $\Ku(X)^o:=\rL_{\cU}b^*(\Ku(X))$. 

First we consider the right-hand-side. We may write $\Db(Y)=\langle \cO(-h), \cO, \cS^\vee, \cO(h) \rangle$, where $\cS$ is the spinor bundle on $Y$ (equal to the restriction of the tautological subbundle $\cS$ on the Grassmannian $\Gr(2,V_4)$ of which $Y$ is a hyerplane section). 
Then the right-hand-side of \eqref{eq_relationonleftmutations} is identified with
$$\rL_{\langle  \cO(-h-H), \cO(-H), \cS^\vee(-H), \cO(h-H) \rangle} \circ \rL_{\cO(-h)}.$$
Note that $\Ku(X)^o$ and $\cO_{\tX}(-h)$ are right orthogonal to $\cO(h-H)$ by \cite[(2.17)]{PPZ}. Also note that 
$$(\langle  \cO(-h-H), \cO(-H), \cS^\vee(-H) \rangle, \cO(-h))$$ 
are semiorthogonal. Indeed,  by Serre duality
$$\Hom^\bullet(\cO(-h),\cS^\vee(-H))=\Hom^\bullet(\cS^\vee, \cO(H-h-h-H))=\Hom^\bullet(\cS^\vee, \cO(-2h))=0.$$
Similarly, if $E$ denotes the exceptional divisor of $b \colon \tX \to X$, then 
$$\Hom^{\bullet}(\cO(-h), \cO(-H)) = \rH^{\bullet}(\cO(-E)) = 0,$$ 
and $\Hom^{\bullet}(\cO(-h), \cO(-h-H)) = \rH^{\bullet}(\cO(-H)) = 0$. 
As a consequence, we can write the right-hand-side of \eqref{eq_relationonleftmutations} as
\begin{equation} 
\label{RHS-mutation-proof}
\rL_{\langle  \cO(-h-H), \cO(-H), \cS^\vee(-H), \cO(-h) \rangle}
\end{equation} 
on the subcategory $\Ku(X)^{o}$. 

On the other hand, if we right mutate $\cU(-H)$ through $\cO(-H)$, we obtain $\cQ(-H)$ (up to shift), where $\cQ$ denotes the tautological quotient bundle on $\Gr(2,5)$. So the left-hand-side of \eqref{eq_relationonleftmutations} can be written as
\begin{equation}
\label{LHS-mutation-proof}
\rL_{\langle \cO(-h-H), \cO(-H), \cQ(-H) \rangle}.
\end{equation}
Our goal is to manipulate \eqref{RHS-mutation-proof} into this form. 

Note that the tautological bundle $\cS$ and the quotient bundle $\cW$ on $\Gr(2,V_4)$ satisfy the following relation:
$$\cS \cong \cW^\vee.$$
Consider the following commutative diagram of sheaves on $\tX$(see \cite[Lemma 2.6]{PPZ}):
$$
\xymatrix{
\cS^\vee \ar[r] & \cU^\vee \ar[r] & \cO_E \\
\cO^{\oplus 4} \ar[u] \ar[r]& \cO^{\oplus 5} \ar[u] \ar[r] & \cO \ar[u] \\
\cW^\vee \ar[u] \ar[r] & \cQ^\vee \ar[u] \ar[r] & \cO(-E) \ar[u].
}
$$
Since $\cS^\vee = \rR_{\cO_E} \cU^\vee$, we deduce that $\cW^\vee = \rR_{\cO(-E)} \cQ^\vee$. Dualizing, we get that
$\cW = \rL_{\cO(E)} \cQ$, thus $$\cW(-H) \cong \rL_{\cO(E-H)} \cQ(-H)=\rL_{\cO(-h)} \cQ(-H).$$
Putting everything together, we get
\begin{align*}
\rL_{\langle \cS^\vee(-H), \cO(-h) \rangle}& = \rL_{\langle \cW(-H), \cO(-h) \rangle}= \rL_{\langle \rL_{\cO(-h)}(\cQ(-H)), \cO(-h) \rangle} \\
& = \rL_{\rL_{\cO(-h)}(\cQ(-H))} \circ \rL_{\cO(-h)} \cong \rL_{\cO(-h)} \circ \rL_{\cQ(-H)}.
\end{align*}
Thus~\eqref{RHS-mutation-proof} can be written as
$$\rL_{\langle  \cO(-h-H), \cO(-H), \cO(-h), \cQ(-H) \rangle}.$$
Using the sequence
$$0 \to \cO(E) \to \cQ \to \cW \to 0,$$
we deduce that 
$$\rR_{\cQ(-H)}\cO(-h)=\cW(-H)[-1]=\cS^\vee(-H)[-1].$$
So~\eqref{RHS-mutation-proof} is equal to 
\begin{equation*}
\rL_{\langle  \cO(-h-H), \cO(-H), \cQ(-H), \cS^\vee(-H) \rangle}.
\end{equation*}

Now set
$$M:=\rL_{\langle \cO(-H),\cQ(-H) \rangle} \cS^\vee(-H).$$
Thus~\eqref{RHS-mutation-proof} can be written as
\begin{equation}
\label{eq_rhs}
\rL_{\cO(-h-H)} \circ \rL_{\langle M, \cO(-H), \cQ(-H) \rangle}.
\end{equation} 

\begin{sublemma}
\label{sublemma-mutation}
If $F \in \Ku(X)^o$, then
$$\Hom^\bullet(M, \rL_{\langle \cO(-H), \cQ(-H) \rangle}F)=0.$$
\end{sublemma}
\begin{proof}
By definition there exists an exact triangle
$$A \to \cS^\vee(-H) \to M,$$
where $A \in \langle \cO(-H), \cQ(-H) \rangle$. Thus 
$$\Hom^\bullet(M, \rL_{\langle \cO(-H), \cQ(-H) \rangle}F)=\Hom^\bullet(\cS^\vee(-H), \rL_{\langle \cO(-H), \cQ(-H) \rangle}F).$$
Recall the sequence
$$0 \to \cS^\vee(-H) \to \cU^\vee(-H) \to \cO_E(-H) \to 0,$$
where $\cU^\vee(-H) \cong \cU$. Note that
$$\Hom^\bullet(\cU, \rL_{\langle \cO(-H), \cQ(-H) \rangle}F)=0,$$
since $\Ku(X)^o$ is right orthogonal to $\cU$ (by definition) and
$$\Hom^\bullet(\cU, \cO(-H))=0= \Hom^\bullet(\cU, \cQ(-H)).$$
It follows that
$$\Hom^\bullet(\cS^\vee(-H), \rL_{\langle \cO(-H), \cQ(-H) \rangle}F)=\Hom^{\bullet+1}(\cO_E(-H), \rL_{\langle \cO(-H), \cQ(-H) \rangle}F).$$

Now consider the sequence
$$0 \to \cO(-E-H) \to \cO(-H) \to i_*\cO_E(-H).$$
We see that
$$\Hom^{\bullet}(i_*\cO_E(-H), \rL_{\langle \cO(-H), \cQ(-H) \rangle}F)=\Hom^{\bullet-1}(\cO(-E-H), \rL_{\langle \cO(-H), \cQ(-H) \rangle}F).$$
By Serre duality on $\tX$, this can be computed through
$$\Hom^{\bullet}(\rL_{\langle \cO(-H), \cQ(-H) \rangle}F, \cO(-E-H-H-h))$$
which is
\begin{equation}
\label{eq_sdontX}
\Hom^{\bullet}(\rL_{\langle \cO(-H), \cQ(-H) \rangle}F, \cO(-3H)).
\end{equation}

If $F= \rL_{\cU}b^*G$ for $G \in \Ku(X)$, then we have
$$\rL_{\langle \cO(-H), \cQ(-H) \rangle}F=\rL_{\langle \cO(-H), \cQ(-H) \rangle}\rL_{\cU}b^*G \cong b^* \rL_{\langle \cO(-H), \cQ(-H) \rangle}\rL_{\cU} G.$$
So by Serre duality on $X$, we can compute \eqref{eq_sdontX} via
\begin{align*}
\Hom^{\bullet}(b^* \rL_{\langle \cO(-H), \cQ(-H) \rangle}\rL_{\cU} G, \cO(-3H)) & \cong \Hom^{\bullet}(\rL_{\langle \cO(-H), \cQ(-H) \rangle}\rL_{\cU} G, \cO(-3H))\\
& \cong \Hom^{\bullet}(\cO(-H),\rL_{\langle \cO(-H), \cQ(-H) \rangle}\rL_{\cU} G)=0.
\end{align*}
This vanishing implies the statement.
\end{proof}

By Sublemma~\ref{sublemma-mutation} we see that
$$\rL_{\langle M, \cO(-H), \cQ(-H) \rangle} \simeq \rL_{\langle \cO(-H), \cQ(-H) \rangle}$$
on $\Ku(X)^o$. This proves that \eqref{eq_rhs} is equal to $\rL_{\langle \cO(-h-H), \cO(-H), \cQ(-H) \rangle}$, i.e. to~\eqref{LHS-mutation-proof}, on $\Ku(X)^o$, as we wanted.
\end{proof}

Now consider the restriction of $F:= - \otimes_{\Cl_0} \Cl_1$ to $\Ku(Y, \Cl_0)$. Note that $F$ is an autoequivalence of $\Ku(Y, \Cl_0)$. Indeed, by \cite[Theorem 2.11]{PPZ} we have the semiorthogonal decomposition of the form
$$\Db(Y, \Cl_0)=\langle \Ku(Y, \Cl_0), \Cl_0, \Cl_1, \cR_a, \cR_b \rangle.$$
For every $A \in \Ku(Y, \Cl_0)$, using that $F(\Cl_1)=\Cl_0$ and $F$ either fixes $\cR_a$ and $\cR_b$ or swaps them (see the proof of \cite[Lemma 4.9]{PPZ}), we have
$$\Hom^\bullet(\Cl_0, A \otimes_{\Cl_0} \Cl_1)=\Hom^\bullet(\Cl_1, A)=0= \Hom^\bullet(\Cl_1, A \otimes_{\Cl_0} \Cl_1)=\Hom^\bullet(\Cl_0, A)=0,$$
$$\Hom^\bullet(\cR_c, A \otimes_{\Cl_0} \Cl_1)=\Hom^\bullet(\cR_c \otimes_{\Cl_0} \Cl_1, A)=0$$
for $c=a,b$, which implies $F(A) \in \Ku(Y, \Cl_0)$. Since $F$ is an equivalence, the same holds for its restriction. Thus by abuse of notation we set
$$F:=- \otimes \Cl_1 \colon \Ku(Y, \Cl_0) \xrightarrow{\sim} \Ku(Y, \Cl_0).$$
By Lemma~\ref{lemma_xiisequivariant} we also have the autoequivalence
$$\sO_{\tX} \colon \Ku(X)' \xrightarrow{\sim} \Ku(X)'.$$
\begin{proof}[Proof of Theorem~\ref{thm_invariantstab}]
By Lemma~\ref{lemma_fromspecialtoordinary}, it is enough to show that $\sO_X \cdot \sigma_X= \sigma_X$ in the case of an ordinary GM fourfold $X$ with smooth canonical quadric. Recall that then $\sigma_X=\xi^{-1} \cdot \Xi^{-1} \cdot \sigma$ where $\sigma$ is a stability condition $\sigma$ constructed on $\Ku(Y, \Cl_0)$. 

First, note that $F \cdot \sigma=\sigma$. Indeed, recall that $\sigma$ is the restriction of the weak stability condition $\sigma_\alpha=(\Coh^0_{\alpha, \beta}(Y, \Cl_0), Z^0_{\alpha, \beta})$, where $\beta=-\frac{5}{4}$ and $0< \alpha < \frac{1}{4}$. By definition 
\begin{equation*}
    F \cdot \sigma=(F(\Coh^0_{\alpha, \beta}(Y, \Cl_0)|_{\Ku(Y, \Cl_0)}), Z^0_{\alpha, \beta}|_{\Knum(\Ku(Y, \Cl_0))} \circ F_*^{-1}),
\end{equation*}
where $F_*$ denotes the automorphism of $\Knum(\Ku(Y, \Cl_0))$ induced by $F$. Since tensoring by $\Cl_1$ does not change the Chern character of objects in $\Db(Y, \Cl_0)$ in degrees $\leq 2$ by \cite[Lemma 3.9(3)]{PPZ}, we have that $Z^0_{\alpha, \beta}|_{\Knum(\Ku(Y, \Cl_0))} \circ F_*^{-1}=Z^0_{\alpha, \beta}|_{\Knum(\Ku(Y, \Cl_0))}$. Moreover, tensoring by $\Cl_1$ preserves $\Coh(Y, \Cl_0)$ and slope stability, thus $\Coh^\beta(Y, \Cl_0) \otimes_{\Cl_0} \Cl_1=\Coh^\beta(Y, \Cl_0)$. We also observe that if $E \in \Coh^\beta(Y, \Cl_0)$ is $\sigma_{\alpha, \beta}$-semistable, then $E \otimes_{\Cl_0} \Cl_1$ is $\sigma_{\alpha, \beta}$-semistable with the same $\mu_{\alpha,\beta}$-slope, since $\ch(E \otimes_{\Cl_0} \Cl_1)_{\leq 2}=\ch(E)_{\leq 2}$. Using the definition of $\Coh^0_{\alpha,\beta}(Y, \Cl_0)$ as the tilting at $\mu_{\alpha, \beta}=0$ of $\Coh^\beta(Y, \Cl_0)$ we conclude that $\Coh^0_{\alpha, \beta}(Y, \Cl_0) \otimes_{\Cl_0} \Cl_1=\Coh^0_{\alpha, \beta}(Y, \Cl_0)$ and so $F(\Coh^0_{\alpha, \beta}(Y, \Cl_0)|_{\Ku(Y, \Cl_0)})=\Coh^0_{\alpha, \beta}(Y, \Cl_0)|_{\Ku(Y, \Cl_0)}$, as we wanted.

Now note that the stability condition $\sigma':=\Xi^{-1} \cdot \sigma$ on $\Ku(X)'$ satisfies $\sO_{\tX}^{-1} \cdot \sigma'=\sigma'$. Indeed, by Lemma \ref{lemma_Xiisequivariant} and the previous computation, we have
$$\sO_{\tX}^{-1} \cdot \sigma'= \sO_{\tX}^{-1} \cdot \Xi^{-1} \cdot \sigma = \Xi^{-1} \cdot F^{-1} \cdot \sigma= \Xi^{-1} \cdot \sigma=\sigma'.$$
As a consequence, by Lemma~\ref{lemma_xiisequivariant} it follows that
$$\sO_X^{-1} \cdot \sigma_X = \sO_X^{-1} \cdot \xi^{-1} \cdot \sigma' = \xi^{-1} \cdot \sO_{\tX}^{-1} \cdot \sigma'= \xi^{-1} \cdot \sigma'=\sigma_X.$$
Equivalently, $\sO_X \cdot \sigma_X= \sigma_X$, proving the statement.
\end{proof}

\subsection{Serre invariant stability conditions on $\Ku(X)$}

Let us now denote by $W$ a special GM fourfold and by $X$ its associated ordinary GM threefold.

By Theorem~\ref{thm_invariantstab} and Theorem~\ref{theorem-equivariant-stability} a stability condition $\sigma_W \in \Stab^\circ(\Ku(W))$ induces a stability condition $\sigma_W^{\bZ/2}$ on the invariant category $\Ku(W)^{\bZ/2}$ and thus a full numerical stability condition 
\begin{equation} \label{eq_stabinducedonordinaryGM}
\sigma_X:=\Phi_0^{-1}(\sigma_W^{\bZ/2})    
\end{equation}
on $\Ku(X)$. This provides an alternative construction to \cite{BLMS} of stability conditions on $\Ku(X)$.

Recall that in \cite{BLMS} stability conditions on $\Ku(X)$ are constructed by restriction of double tilting of slope stability on $\Db(X)$ (see \eqref{eq_stabonKu}). We denote one of these stability conditions by $\tau_X$.

\begin{corollary} \label{cor_sigmaassigmaX}
The stability condition $\sigma_X$ is in the same $\widetilde{\emph{GL}}^+(2, \bR)$-orbit of every stability condition $\tau_X$ constructed in \cite{BLMS}. In particular, $\sigma_X$ is proper.
\end{corollary}
\begin{proof}
Since $\sigma_W^{\bZ/2}$ is
fixed by the $\bZ/2$-action on $\Ku(W)^{\bZ/2}$ 
as remarked in \S\ref{subsec_dcdcandstability}, by \eqref{eq_tau} it follows that $\sigma_X$ is Serre invariant. By \cite[Theorem 1.1]{PR} the stability condition $\tau_X$ is Serre invariant. By \cite[Lemmas 4.27, 4.28, 4.29]{JLLZ} (see also \cite[Corollary 4.5]{PR}) there is a unique $\widetilde{\text{GL}}^+(2, \bR)$-orbit of Serre invariant stability conditions on $\Ku(X)$. This implies the statement.
\end{proof}

\begin{remark}
Up to now, the known stability conditions on $\Ku(X)$ are Serre invariant. It would be interesting to understand whether every stability condition on $\Ku(X)$ is Serre invariant.    
\end{remark}

As a consequence of Theorem~\ref{thm_invariantstab}, we remark the following identification of the stability conditions in $\Stab^\circ(\Ku(W))$ for a GM fourfold, which is of independent interest.

\begin{corollary} \label{cor_samestabinStabcirc}
Let $W$ be a GM fourfold. Then for every pair $\sigma^1_W$, $\sigma^2_W \in \Stab^\circ(\Ku(W))$, there exists $(M,g) \in \widetilde{\GL}_2^{+}(\bR)$ such that $\sigma^2_W= \sigma^1_W \cdot (M,g)$. 
\end{corollary}
\begin{proof}
Let $X$ be the opposite GM variety to $W$ (see Example~\ref{ex_GM}), so $\Ku(W)$ is the CY2 cover of $\Ku(X)$. By Theorem~\ref{thm_invariantstab} and Lemma~\ref{lemma_fromspecialtoordinary} the stability conditions $\sigma^1_W$ and $\sigma^2_W$ are $\bZ/2$-fixed, thus by Theorem~\ref{theorem-equivariant-stability} they induce stability conditions $\sigma^1_X, \sigma^2_X$ on $\Ku(X)$, which are fixed by the $\bZ/2$-action. The generator for this action is the $(-2)$-shifted Serre functor of $\Ku(X)$. By \cite[Corollary 4.5]{PR} there exists $(g, M) \in \widetilde{\GL}_2^{+}(\bR)$ such that $\sigma^2_X= \sigma^1_X \cdot (M,g)$. 

As a consequence, if $\cP_{\sigma_W^i}$ and $\cP_{\sigma_X^i}$ denote the slicings of $\sigma_W^i$ and $\sigma_X^i$, respectively, then 
$$\cP_{\sigma_W^2}(\phi)=\Forg(\cP_{\sigma_X^2}(\phi))=\Forg(\cP_{\sigma_X^1}(g(\phi)))=\cP_{\sigma_W^1}(g(\phi))$$
for every phase $\phi \in \bR$. Moreover, since the central charges of the $\sigma_X^i$'s satisfy $Z^2_X=M^{-1}\circ Z^1_X$, and $Z^i_X=Z^i_W \circ \bv \circ \Forg_*$, where $\bv \colon \rK(\Ku(W)) \to \Knum(\Ku(W)))$, then 
$$Z^2_W \circ \bv \circ \Forg_*=M^{-1} \circ Z^1_W \circ \bv \circ \Forg_*.$$
This implies that $Z^2_W$ and $M^{-1} \circ Z^1_W$ coincide on the image of the forgetful functor in $\Knum(\Ku(W))$, which is equal to the invariant lattice with respect to the $\bZ/2$-action. By \cite[Proposition 5.7]{BP} this lattice is the canonical sublattice $-A_1^{\oplus 2}$. By \cite[Lemma 4.17]{PPZ} the central charges $Z_W^i$'s factors through the $-A_1^{\oplus 2}$-lattice. We conclude  $Z^2_W=M^{-1} \circ Z^1_W$. 
\end{proof}

\subsection{Proof of Theorem~\ref{theorem-fixed-locus}}

Assume first that $W$ is ordinary. Then there is a special GM fourfold $W'$ and an equivalence $\Ku(W') \cong \Ku(W)$ which intertwines $\iota$ and the rotation functor $\sO_{W}$ by Lemma \ref{lemma_fromspecialtoordinary}. Thus if $\sigma'$ denotes the stability condition on $\Ku(W')$ corresponding to $\sigma \in \Stab^\circ(\Ku(W))$ under this equivalence and $w' \in \Knum(\Ku(W'))$ denotes the class corresponding to $w \in \Knum(\Ku(W))$, then the fixed loci $M_\sigma(\Ku(W), w)^{\langle \sO_{W} \rangle}$ and $M_{\sigma'}(\Ku(W'), w')^{\langle \iota \rangle}$ are identified by the isomorphism $M_\sigma(\Ku(W), w)\cong M_{\sigma'}(\Ku(W'), w')$ induced by $\Ku(W') \simeq \Ku(W)$. Thus it is enough to prove the result for special GM fourfolds.

So assume that $W$ is special, and we denote by $X$ its branch GM threefold. By Theorem~\ref{thm_invariantstab} and Theorem~\ref{theorem-equivariant-stability} the stability condition $\sigma$ induces a stability condition $\sigma_X$  on $\Ku(X)$ defined as in \eqref{eq_stabinducedonordinaryGM}, which is Serre invariant by Corollary~\ref{cor_sigmaassigmaX}. On the other hand, since $w$ is primitive, there exist coprime integers $p, q$ such that $w=p\lambda_1+q\lambda_2$. By Lemma~\ref{lemma_Forgkappa} we have $(\Forg \circ \Phi_0)_*(p \kappa_1+q \kappa_2)=w$, and by Theorem~\ref{theorem-Msigma-GM}\eqref{nonempty-GM} the moduli stack $\cM_{\sigma_X}(\Ku(X), p \kappa_1+q \kappa_2)$ is nonempty.  
By \cite[Proposition 5.7]{BP} the lattice $-A_1^{\oplus 2}$ is fixed by $\iota$, so $w$ is $\iota$-fixed. Then by Theorem~\ref{theorem-moduli}\eqref{theorem-CY2-moduli} there is a surjective double cover 
$$\bigsqcup_{(\Forg \circ \Phi_0)_*(v) = w} \cM_{\sigma_X}(\Ku(X), v) \to \cM_{\sigma}(\Ku(W), w)^{\langle \iota \rangle}.$$
Hence the above morphism maps the nonempty moduli stack $\cM_{\sigma_X}(\Ku(X), p \kappa_1+q \kappa_2)$ to $\cM_{\sigma}(\Ku(W), w)^{\langle \iota \rangle}$, implying that the latter is nonempty as well.

Since $\cM_{\sigma}(\Ku(W), w) \to M_{\sigma}(\Ku(W), w)$ is a $\bG_m$-gerbe, we obtain that the fixed locus of $\iota$ in $M_{\sigma}(\Ku(W), w)$ is nonempty. The fact that it is a smooth Lagrangian subvariety follows from \cite[Lemma 1]{Beauville}.

\begin{remark}
Note that the projectivity of the moduli space in Theorem~\ref{theorem-Msigma-GM}\eqref{Msigma-GM-smooth} can be obtained by applying the strategy used in the proof of Theorem~\ref{theorem-fixed-locus}. Indeed, we can assume that $X$ is generic, in the sense that it is the branch of a very general GM fourfold $W$. Then every $\sigma \in \Stab^\circ(\Ku(W))$ is $(\Forg \circ \Phi_0)_*(v)$-generic. Hence by \cite{PPZ} the moduli space $M_{\sigma}(\Ku(W), w)$ is projective for $w:=(\Forg \circ\Phi_0)_*(v)$. Thus the fixed locus of $\iota$ in $M_{\sigma}(\Ku(W), w)$ is projective and since we have a surjective double covering
$$\bigsqcup_{(\Forg \circ \Phi_0)_*(v') = w} \cM_{\sigma_X}(\Ku(X), v') \to \cM_{\sigma}(\Ku(W), w)^{\langle \iota \rangle},$$
we conclude that $M_{\sigma_X}(\Ku(X), v)$ is projective.
\end{remark}


\newcommand{\etalchar}[1]{$^{#1}$}
\providecommand{\bysame}{\leavevmode\hbox to3em{\hrulefill}\thinspace}
\providecommand{\MR}{\relax\ifhmode\unskip\space\fi MR }
\providecommand{\MRhref}[2]{%
  \href{http://www.ams.org/mathscinet-getitem?mr=#1}{#2}
}
\providecommand{\href}[2]{#2}


\end{document}